\documentclass[12pt]{amsart}
\usepackage{amssymb}
\usepackage{rotating}
\usepackage{mathtools}
\def\multiset#1#2{\ensuremath{\left(\kern-.5em\left(\genfrac{}{}{0pt}{}{#1}{#2}\right)\kern-.5em\right)}}
\usepackage{caption}

\usepackage{geometry}
\usepackage{tikz,graphicx}
\tikzstyle{vertex}=[circle, draw, inner sep=0pt, minimum size=4pt]
\newcommand{\vertex}{\node[vertex]}

\definecolor{applegreen}{rgb}{0.55,0.71,0.0}

\newtheorem{theorem}{Theorem}[section]
\newtheorem{lemma}[theorem]{Lemma}
\newtheorem{prop}[theorem]{Proposition}
\newtheorem{corollary}[theorem]{Corollary}

\theoremstyle{definition}
\newtheorem{defn}[theorem]{Definition}
\newtheorem{example}[theorem]{Example}
\newtheorem{remark}[theorem]{Remark}

\newcommand\ba{\mathbf{a}}
\newcommand\bb{\mathbf{b}}
\newcommand\bc{\mathbf{c}}
\newcommand\bd{\mathbf{d}}
\newcommand\be{\mathbf{e}}
\newcommand\bh{\mathbf{h}}

\newcommand\bp{\mathbf{p}}
\newcommand\bq{\mathbf{q}}

\newcommand\bs{\mathbf{s}}
\newcommand\bt{\mathbf{t}}
\newcommand\bu{\mathbf{u}}
\newcommand\bv{\mathbf{v}}

\newcommand\bbN{\mathbb{N}}

\newcommand\bbR{\mathbb{R}}

\newcommand\bbZ{\mathbb{Z}}

\newcommand\calC{\mathcal{C}}
\newcommand\calD{\mathcal{D}}

\newcommand\calF{\mathcal{F}}
\newcommand\calG{\mathcal{G}}

\newcommand\calO{\mathcal{O}}

\newcommand\calT{\mathcal{T}}
\newcommand\calU{\mathcal{U}}

\newcommand\fS{\mathfrak{S}}

\newcommand\lp{\mathrm{lp}}
\newcommand\rp{\mathrm{rp}}
\newcommand\indeg{\mathrm{in}}
\newcommand\outdeg{\mathrm{out}}
\newcommand\ingrav{\mathcal{G}}
\newcommand\outgrav{\mathcal{G}}

\newcommand\PS{\mathrm{PS}}

\DeclareMathOperator{\MCar}{MCar}
\DeclareMathOperator{\Car}{Car}
\DeclareMathOperator{\Cat}{Cat}
\DeclareMathOperator{\PF}{PF}
\DeclareMathOperator{\vol}{vol}

\def\car#1#2{\ensuremath{\Car_{#1}^{(#2)}}}
\def\mcar#1#2{\ensuremath{\MCar_{#1}^{(#2)}}}

\definecolor{Cerulean}{cmyk}{0.94,0.11,0,0}
\definecolor{lava}{rgb}{0.81, 0.06, 0.13}
\definecolor{bleudefrance}{rgb}{0.19, 0.55, 0.91}
\definecolor{kellygreen}{rgb}{0.3, 0.73, 0.09}

\title{A Fuss-Catalan variation of the caracol flow polytope}
\author{Martha Yip}
\address{Department of Mathematics, University of Kentucky, Lexington KY 40506.}
\email{martha.yip@uky.edu}
\date{\today}

\thanks{The author is partially supported by a Simons Collaboration Grant.  She also thanks C. Benedetti, R. Gonz\'alez D'L\'eon, C. Hanusa, P. Harris, A. Khare, and A. Morales for inspiring conversations related to this work. This is dedicated to RFC}

\begin{document}
\begin{abstract}
Recently, a combinatorial interpretation of Baldoni and Vergne's generalized Lidskii formula for the volume of a flow polytope was developed by Benedetti et al.. This converts the problem of computing Kostant partition functions into a problem of enumerating a set of objects called unified diagrams. 
We devise an enhanced version of this combinatorial model to compute the volumes of flow polytopes defined on a family of graphs called the $k$-caracol graphs, resulting in the first application of the model to non-planar graphs. At $k=1$ and $k=n-1$, we recover results for the classical caracol graph and the Pitman--Stanley graph. Furthermore, we introduce the notion of in-degree gravity diagrams for flow polytopes, which are equinumerous with (out-degree) gravity diagrams considered by Benedetti et al.. We show that for the $k$-caracol flow polytopes, these two kinds of gravity diagrams satisfy a natural combinatorial correspondence, which raises an intriguing question on the relationship in the geometry of two related polytopes.


\end{abstract}
\maketitle 
\tableofcontents

\section{Introduction}

In the paper~\cite{BGHHKMY}, we developed a combinatorial model for computing the volume of flow polytopes $\calF_G(\ba)$ which is based on the {\em generalized Lidskii volume formula} due to Baldoni and Vergne~\cite{BV}.  We defined combinatorial objects called {\em unified diagrams}, whose enumeration gives the normalized volume of the flow polytope. We showed that the model can be applied to compute the volume of the flow polytopes of the Pitman--Stanley graph, the zigzag graph, and the caracol graph at various net flows, without the need for constant term identities. 

In this paper, we show that the combinatorial model can be applied to compute the volume of the flow polytopes of a family of graphs which we call the {\em $k$-caracol graphs}. Setting $k=1$ recovers the results obtained for the caracol graph developed in~\cite{BGHHKMY}, and setting $k=n-1$ recovers some of the results for the Pitman--Stanley polytope~\cite{PS}.

We note that this is the first application of the combinatorial model to non-planar graphs.  Indeed, the motivation for studying the flow polytopes of the $k$-caracol graphs was borne from the desire to understand the combinatorics of the flow polytope of the complete graph.  The Chan--Robbins-Yuen polytope $\mathrm{CRY}(n)$~\cite{CRY} can be realized as the flow polytope of the complete graph $K_{n+1}$ with net flow vector $\ba=\varepsilon_1-\varepsilon_{n+1}$.  A well-known result due to Zeilberger~\cite{Z} states that the volume of $\mathrm{CRY}(n)$ is the product of the first $n-2$ Catalan numbers.  Despite the combinatorial nature of the formula, the proof relies on an application of the Morris constant term identity, and no combinatorial proof is known. 

Other generalizations of the volume of the flow polytope of the complete graph $K_{n+1}$ also involve products of combinatorial quantities.  At net flow $\ba=\varepsilon_1+\varepsilon_2-2\varepsilon_{n+1}$, Corteel, Kim, and M\'esz\'aros~\cite{CKM} showed that the volume is \hbox{$2^{{n\choose2}-1}$} times the product of the first $n-2$ Catalan numbers, while at net flow $\ba=\sum_{i=1}^n \varepsilon_i - n\varepsilon_{n+1}$, M\'esz\'aros, Morales, and Rhoades~\cite{MMR} showed that the volume is the number of standard Young tableaux of staircase shape $(n-1,n-2,\ldots, 2,1,0)$ times the product of the first $n-1$ Catalan numbers.  Both of these generalizations rely on the Morris constant term identity as well. 

Combinatorial objects such as Dyck paths and parking functions appeared naturally in the study of the Pitman--Stanley polytope, 
which is affinely equivalent to the flow polytope of the Pitman--Stanley graph.
These objects play a central role in the unified diagrams for flow polytopes, and we saw in~\cite{BGHHKMY} that the volume of the flow polytope of the caracol graph with net flow $\ba=\varepsilon_1-\varepsilon_{n+1}$ is the Catalan number $\Cat(n-2)$, while with net flow $\ba= \sum_{i=1}^n \varepsilon_i - n\varepsilon_{n+1}$, the volume is $\Cat(n-2)\cdot n^{n-2}$, the product of a Catalan number and the number of parking functions of length $n-1$. A main result of this paper is a generalization of this to the $k$-caracol family of graphs.\\

\noindent{\bf Theorem~\ref{thm.oneoneone}.} For $k\in\bbN$ and $n>k$,
$$\vol\calF_{\car{n+1}{k}} (1,\ldots, 1,-n)
=\Cat(n-k,k(n-k)-1)\cdot k^{k(n-k)-2} \cdot n^{n-k-1},$$
where $\Cat(a,b)$ is a {\em rational Catalan number} (see~\cite{ALW}).  For the special case when $b=ka-1$, it is a {\em generalized Fuss-Catalan number}, and when $b=a+1$, it is the classical Catalan number.

Many of the ideas from~\cite{BGHHKMY} are generalized in this present paper, but a refinement of the original combinatorial model is necessary in order to explain the appearance of the factor $k^{k(n-k)-2}$ in Theorem~\ref{thm.oneoneone}, which is undetected when $k=1$.  We therefore introduce the {\em truncated unified diagrams}, whose `completions' are the standardized diagrams.  The truncated diagrams are enumerated by the {\em $k$-parking numbers} (see the Appendix) 
$$T_k(r,i) = (r+1)^{i-1} \multiset{k(r+1)}{r-i},$$
which interpolate between the generalized Fuss-Catalan numbers and the number of parking functions. We give these numbers a combinatorial interpretation involving a vehicle-parking scenario in Thereom~\ref{thm.multilabel}.  The formula in Theorem~\ref{thm.oneoneone} is then obtained by a binomial transform of these numbers, up to a power of $k$.

This power of $k$ arises from counting the completions of the truncated diagrams to standardized diagrams.  The factor of $k$ which appears in the formula of Theorem~\ref{thm.oneoneone} can be explained by considering a cyclic group action on the set of truncated diagrams, together with a delightful partitioning of the `$N$-th multinomial $(k-1)$-simplex' (better known as the `$N$-th row of Pascal's triangle' in the case $k=2$), whose entries sum to $k^N$ (Lemma~\ref{lem.thelemma}). 



This paper is organized as follows. In Section~\ref{sec.lidskii}, we introduce the family of $k$-caracol graphs, state the generalized Lidskii volume formulas, and introduce one of the key ingredients of a unified diagram, called a {\em gravity diagram}.  These are a combinatorial interpretation of the {\em Kostant partition function}.  We also discuss the necessary background on rational Catalan combinatorics, and then give two bijective proofs to show that the volume of the flow polytope of the $k$-caracol graph $\car{n+1}{k}$ at unit flow $\ba=(1,0,\ldots,0,-1)$ is the generalized Fuss-Catalan number $\Cat(n-k,k(n-k)-1)$.
The combinatorics arising from Theorems~\ref{thm.onezerozero} and~\ref{prop.outgrav} give rise to an interesting geometric question (Remark~\ref{rem.duality}).
In Section~\ref{sec.kcaracol}, we introduce the unified diagrams for the flow polytope of the $k$-caracol graphs, and its variations.  We define the $k$-parking numbers, and show that they enumerate the truncated unified diagrams (Theorem~\ref{thm.multilabel}). 
Having developed all the necessary enumerative tools, we end the section with a proof of Theorem~\ref{thm.oneoneone}.
In Section~\ref{sec.abbb}, we discuss a generalization of the volume of the flow polytopes of the $k$-caracol graphs at more general net flow vectors in Theorem~\ref{thm.abbb}, and show that $k$-parking numbers form {\em log-concave} sequences.
Finally in Section~\ref{sec.mcar}, we discuss some results for a multigraph which we call the {\em $k$-multicaracol graph}, whose volume formulas are closely related to those for the $k$-caracol graph at various net flows (Theorem~\ref{thm.mcarabbb}).  We end with a suggestion of an alternative pathway towards another combinatorial proof of Theorem~\ref{thm.oneoneone}.

\section{Volume of the $k$-caracol polytope with unit flow}\label{sec.lidskii}

\subsection{Flow polytopes and the $k$-caracol graphs}

We define the family of {\em $k$-caracol graphs}.
\begin{defn} \label{defn.kcar}
Let $k\in\bbN=\bbZ_{\geq0}$ and $n>k$.  The directed graph $G=\car{n+1}{k}$ has vertex set $V(G)=\{1,\ldots, n+1\}$ and edge set
$$
E(G)=\left\{(i,i+1), (i,k+1), \ldots, (i,n) \mid 1\leq i\leq k\right\}
\cup
\left\{(i,i+1), (i,n+1) \mid k+1 \leq i\leq n \right\}.
$$
For clarity, we point out that $\car{n+1}{k}$ is a graph without multiple edges, and note that the number of edges in $\car{n+1}{k}$ is $m=(k+1)(n-k)+n-2$ for. The graph $\car{n+1}{1}$ is the caracol graph, and the graph $\car{n+1}{n-1}$ is the Pitman--Stanley graph $\PS_n$ with an additional edge $(n,n+1)$. The flow polytopes of both graphs were previously studied in~\cite{BGHHKMY}. We point out that our definition of $\PS_n$ differs from others found in the literature in that the edge $(n-1,n)$ is not repeated.
\end{defn}

\begin{figure}[ht]
\begin{center}
\begin{tikzpicture}[scale=0.8]
\begin{scope}[yshift=0]
	\vertex[fill,label=below:\footnotesize{$1$}](a0) at (0,0) {};
	\vertex[fill,label=below:\footnotesize{$2$}](a1) at (1,0) {};
	\vertex[fill,label=below:\footnotesize{$3$}](a2) at (2,0) {};
	\vertex[fill,label=below:\footnotesize{$4$}](a3) at (3,0) {};
	\vertex[fill,label=below:\footnotesize{$5$}](a4) at (4,0) {};
		\node at (4.5,0) {$\cdots$};
	\vertex[fill,label=above :\footnotesize{$n-1$}](a10) at (5,0) {};
	\vertex[fill,label=above:\footnotesize{$n$}](a11) at (6,0) {};
	\vertex[fill,label=above:\footnotesize{$n+1$}](a12) at (7,0) {};
	\draw[->, >=stealth, thick] (a0)--(a1);		
	\draw[-stealth, thick] (a1)--(a2);
	\draw[-stealth, thick] (a2)--(a3);
	\draw[-stealth, thick] (a3)--(a4);
	\draw[thick] (a4)--(4.15,0); \draw[thick] (4.75,0)--(4.85,0);
		\draw[-stealth, thick] (4.85,0)--(a10);		
	\draw[-stealth, thick] (a10)--(a11);
	\draw[-stealth, thick] (a11) to (a12);
	\draw[-stealth, thick] (a10) to[out=-50,in=240] (a12);
	\draw[-stealth, thick] (a4) to[out=-50,in=240] (a12);
	\draw[-stealth, thick] (a3) to[out=-50,in=240] (a12);
	\draw[-stealth, thick] (a2) to[out=-50,in=240] (a12);
	\draw[-stealth, thick] (a1) to[out=-50,in=240] (a12);
	\draw[-stealth, thick] (a0) to[out=-50,in=240] (a12);
\end{scope}
\begin{scope}[xshift=250]
	\vertex[fill,label=below:\footnotesize{$1$}](a0) at (0,0) {};
	\vertex[fill,label=below:\footnotesize{$2$}](a1) at (1,0) {};
	\vertex[fill,label=below:\footnotesize{$3$}](a2) at (2,0) {};
	\vertex[fill,label=below:\footnotesize{$4$}](a3) at (3,0) {};
	\vertex[fill,label=below:\footnotesize{$5$}](a4) at (4,0) {};
		\node at (4.5,0) {$\cdots$};
	\vertex[fill,label=above :\footnotesize{$n-1$}](a10) at (5,0) {};
	\vertex[fill,label=above:\footnotesize{$n$}](a11) at (6,0) {};
	\vertex[fill,label=above:\footnotesize{$n+1$}](a12) at (7,0) {};
	\draw[-stealth, thick] (0,0)--(.95,0);		
	\draw[-stealth, thick] (1,0)--(1.95,0);
	\draw[-stealth, thick] (2,0)--(2.95,0);
	\draw[-stealth, thick] (3,0)--(3.95,0);
	\draw[thick] (4,0)--(4.15,0); \draw[thick] (4.75,0)--(4.85,0);
		\draw[-stealth, thick] (4.85,0)--(4.95,0);		
	\draw[-stealth, thick] (5,0)--(5.95,0);
	\draw[-stealth, thick] (6,0) to (6.95,0);
	\draw[-stealth, thick] (a0) to[out=25,in=130] (a2);
	\draw[-stealth, thick] (a0) to[out=30,in=130] (a3);
	\draw[-stealth, thick] (a0) to[out=35,in=130] (a4);
	\draw[-stealth, thick] (a0) to[out=40,in=130] (a10);
	\draw[-stealth, thick] (a0) to[out=45,in=130] (a11);
	\draw[-stealth, thick] (a10) to[out=-50,in=240] (a12);
	\draw[-stealth, thick] (a4) to[out=-50,in=240] (a12);
	\draw[-stealth, thick] (a3) to[out=-50,in=240] (a12);
	\draw[-stealth, thick] (a2) to[out=-50,in=240] (a12);
	\draw[-stealth, thick] (a1) to[out=-50,in=240] (a12);
\end{scope}
\end{tikzpicture}
\end{center}
\caption{The graphs $\PS_{n+1}$ and $\car{n+1}{1}=\Car_{n+1}$.}
\end{figure}
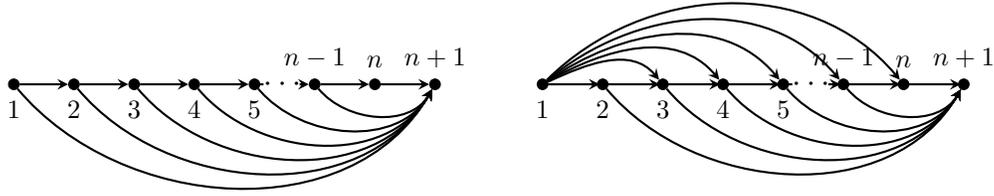

\begin{figure}[th]
\begin{center}
\begin{tikzpicture}[scale=0.8]
\begin{scope}	
	\vertex[fill,label=below:\footnotesize{$1$}](a0) at (0,0) {};
	\vertex[fill,label=below:\footnotesize{$2$}](a1) at (1,0) {};
	\vertex[fill,label=below:\footnotesize{$3$}](a2) at (2,0) {};
	\vertex[fill,label=below:\footnotesize{$4$}](a3) at (3,0) {};
	\vertex[fill,label=below:\footnotesize{$5$}](a4) at (4,0) {};
	\vertex[fill,label=below:\footnotesize{$6$}](a10) at (5,0) {};
	\vertex[fill,label=above:\footnotesize{$7$}](a11) at (6,0) {};
	\vertex[fill,label=above:\footnotesize{$8$}](a12) at (7,0) {};
	
	\draw[-stealth, thick] (0,0)--(.95,0);		
	\draw[-stealth, thick] (1,0)--(1.95,0);
	\draw[-stealth, thick] (2,0)--(2.95,0);
	\draw[-stealth, thick] (3,0)--(3.95,0);
	\draw[-stealth, thick] (4,0)--(4.95,0);	
	\draw[-stealth, thick] (5,0)--(5.95,0);
	\draw[-stealth, thick] (6,0) to (6.95,0);
	\draw[-stealth, thick] (a0) to[out=25,in=130] (a2);
	\draw[-stealth, thick] (a0) to[out=30,in=130] (a3);
	\draw[-stealth, thick] (a0) to[out=35,in=130] (a4);
	\draw[-stealth, thick] (a0) to[out=40,in=130] (a10);
	\draw[-stealth, thick] (a0) to[out=45,in=130] (a11);
	
	\draw[-stealth, thick] (a1) to[out=25,in=130] (a3);
	\draw[-stealth, thick] (a1) to[out=30,in=130] (a4);
	\draw[-stealth, thick] (a1) to[out=35,in=130] (a10);
	\draw[-stealth, thick] (a1) to[out=40,in=130] (a11);
	\draw[-stealth, thick] (a10) to[out=-50,in=230] (a12);
	\draw[-stealth, thick] (a4) to[out=-50,in=235] (a12);
	\draw[-stealth, thick] (a3) to[out=-50,in=240] (a12);
	\draw[-stealth, thick] (a2) to[out=-50,in=240] (a12);
\end{scope}

\begin{scope}[xshift=250, yshift=0]
	\vertex[fill,label=below:\footnotesize{$1$}](a0) at (0,0) {};
	\vertex[fill,label=below:\footnotesize{$2$}](a1) at (1,0) {};
	\vertex[fill,label=below:\footnotesize{$3$}](a2) at (2,0) {};
	\vertex[fill,label=below:\footnotesize{$4$}](a3) at (3,0) {};
	\vertex[fill,label=below:\footnotesize{$5$}](a4) at (4,0) {};
	\vertex[fill,label=below:\footnotesize{$6$}](a10) at (5,0) {};
	\vertex[fill,label=above:\footnotesize{$7$}](a11) at (6,0) {};
	\vertex[fill,label=above:\footnotesize{$8$}](a12) at (7,0) {};
	
	\draw[-stealth, thick] (0,0)--(.95,0);		
	\draw[-stealth, thick] (1,0)--(1.95,0);
	\draw[-stealth, thick] (2,0)--(2.95,0);
	\draw[-stealth, thick] (3,0)--(3.95,0);
	\draw[-stealth, thick] (4,0)--(4.95,0);	
	\draw[-stealth, thick] (5,0)--(5.95,0);
	\draw[-stealth, thick] (6,0) to (6.95,0);
	\draw[-stealth, thick] (a0) to[out=30,in=130] (a3);
	\draw[-stealth, thick] (a0) to[out=35,in=130] (a4);
	\draw[-stealth, thick] (a0) to[out=40,in=130] (a10);
	\draw[-stealth, thick] (a0) to[out=45,in=130] (a11);
	
	\draw[-stealth, thick] (a1) to[out=25,in=130] (a3);
	\draw[-stealth, thick] (a1) to[out=30,in=130] (a4);
	\draw[-stealth, thick] (a1) to[out=35,in=130] (a10);
	\draw[-stealth, thick] (a1) to[out=40,in=130] (a11);
	
	\draw[-stealth, thick] (a2) to[out=30,in=130] (a4);
	\draw[-stealth, thick] (a2) to[out=35,in=130] (a10);
	\draw[-stealth, thick] (a2) to[out=40,in=130] (a11);
	\draw[-stealth, thick] (a10) to[out=-50,in=230] (a12);
	\draw[-stealth, thick] (a4) to[out=-50,in=235] (a12);
	\draw[-stealth, thick] (a3) to[out=-50,in=240] (a12);
\end{scope}
\end{tikzpicture}

\end{center}
\caption{The graphs $\car82$ and $\car83$.}
\end{figure}
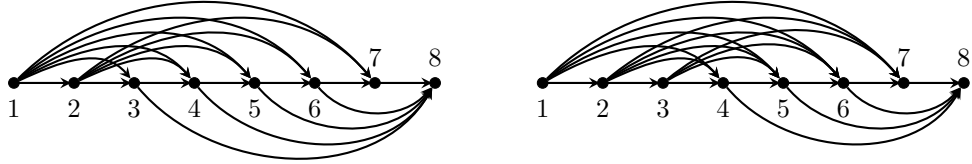

\begin{defn}\label{defn.graph}
Let $G$ be a connected acyclic directed graph with vertex set $V(G)=\{1,\ldots, n+1\}$ and edge multiset $E(G)$ with $m$ edges.  Further assume that
\begin{itemize}
\item[(a)] the out-degree of each of the vertices $1$ through $n$ is at least one,
\item[(b)] the in-degree of each of the vertices $2$ through $n+1$ is at least one,
\item[(c)] the edges of $G$ are each directed from $i$ to $j$ if $i<j$.
\end{itemize}
For $i=1,\ldots, n$, let $t_i = \outdeg_i-1$ be one less than the out-degree of the vertex $i$. The {\em shifted out-degree vector} of $G$ is $\bt=(t_1,\ldots,t_n)\in\bbZ_{\geq0}^n$. 
Similarly for $i=2,\ldots, n+1$, let $u_i=\indeg_i-1$ be one less than the in-degree of the vertex $i$. The {\em shifted in-degree vector} of $G$ is $\bu = (u_2,\ldots,u_{n+1})\in\bbZ_{\geq0}^n$.
Note that $\sum_{i=1}^n t_i= \sum_{i=2}^{n+1}u_i = m-n$.

Given $\ba=(a_1,\ldots, a_n, -\sum_{i=1}^n a_i)$ with $a_i \in \bbZ_{\geq0}$, an {\em $\ba$-flow on $G$} is tuple $(f_e)_{e\in E(G)} \in \bbR_{\geq0}^{m}$ such that
$$\sum_{(j,k)\in E(G)} f_{(j,k)}  - \sum_{(i,j)\in E(G)} f_{(i,j)} = a_j $$
for $j=1,\ldots, n$. The {\em flow polytope of $G$ with net flow $\ba$} is the set $\calF_G(\ba)$ of $\ba$-flows on $G$.  Note that $\dim\calF_G(\ba) = m-n$.
\end{defn}

\subsection{The generalized Lidskii formulas}

\begin{defn}
For $i=1,\ldots, n$, let $\alpha_i = \varepsilon_i-\varepsilon_{i+1}$, where $\varepsilon_i$ is the $i$-th standard basis vector of $\bbR^{n+1}$.
To each edge $e=(i,j)$ of $G$, we associate the vector 
$$\alpha_e = \alpha_{(i,j)} = \alpha_i+\cdots + \alpha_{j-1} = \varepsilon_i-\varepsilon_j,$$
and let $\Phi_G^+ = \{\alpha_e \mid e\in E(G)\}$ be the multiset of {\em positive roots associated to $G$}.

The {\em Kostant partition function of $G$ evaluated at $\ba$}, denoted by $K_G(\ba)$, is the number of ways of expressing the vector $\ba\in \bbZ^{n+1}$ as a linear combination of the vectors in $\Phi_G^+$.  
\end{defn}

In this setting, integral $\ba$-flows on $G$ are equivalent to vector partitions of $\ba$.  Thus, the number of integral $\ba$-flows on $G$ is $K_G(\ba)$.

The {\em normalized volume} of a $d$-dimensional lattice polytope is $d!$ times its Euclidean volume.

\begin{theorem}[Lidskii formulas, {\cite[Theorem 38]{BV}}, {\cite[Theorem 1.1]{MM18}}]
\label{thm.lidskii}
Let $G$ be a connected acyclic directed graph with vertex set $\{1,\ldots, n+1\}$ and $m$ edges, along with the additional properties as outlined in Definition~\ref{defn.graph}.
The normalized volume of the flow polytope $\calF_G(\ba)$ of $G$ with net flow vector $\ba=(a_1,\ldots,a_n, -\sum_{i=1}^n a_i)$ is
$$\vol\calF_G(\ba) = \sum_{\bs \rhd \bt} {m-n \choose s_1,\ldots,s_n}\cdot 
	a_1^{s_1}\cdots a_n^{s_n}\cdot 
		K_{G}(s_1-t_1, \ldots, s_n-t_n, 0),$$
and the number of lattice points of $\calF_G(\ba)$ is
\begin{align*}
K_G(\ba) 
&= \sum_{\bs\rhd\bt} 
	{a_1+t_1\choose s_1} \cdots {a_n+t_n\choose s_n} \cdot
	K_G(s_1-t_1, \ldots, s_n-t_n, 0),\\
&= \sum_{\bs\rhd\bt} 
	\multiset{a_1-u_1}{s_1} \cdots \multiset{a_n-u_n}{s_n} \cdot
	K_G(s_1-t_1, \ldots, s_n-t_n, 0),
\end{align*}
where $K_G$ is the Kostant partition function of $G$, and the sum is over weak compositions $\bs=(s_1,\ldots,s_n) \vDash m-n$ such that $\sum_{i=1}^k s_i \geq \sum_{i=1}^k t_i$ for every $k$.
\end{theorem}

The special case of the Lidskii volume formula at $\ba=(1,0,\ldots, 0,-1)$ plays a central role in the following sections.
\begin{corollary}\label{cor.main} Let $G$ be a directed graph with $n+1$ vertices and $m$ edges, with shifted out-degree and in-degree vectors $\bt=(t_1,\ldots, t_n)$ and $\bu=(u_2,\ldots, u_{n+1})$. Then
\begin{align*}
\vol\calF_G(1,0,\ldots,0,-1)
	&= K_G(m-n-t_1, -t_2,\ldots, -t_n, 0), \label{eqn.outgrav}\\
	&= K_G(0, u_2, \ldots, u_n, u_{n+1}-(m-n)). 
\end{align*}
\end{corollary}
Thus, the volume of the flow polytope $\calF_G(1,0,\ldots,0,-1)$ can be computed by counting the number of lattice points of two related polytopes, as noted by M\'esz\'aros and Morales in~\cite{MM18}.

\subsection{Kostant partition functions and gravity diagrams}

In~\cite{BGHHKMY}, we introduced a combinatorial interpretation of the Kostant partition function of a graph $G$, which we called (out-degree) gravity diagrams.  Here, we introduce an analogous notion of in-degree gravity diagrams.  

\begin{defn}
A vector $\bv = (v_1,\ldots, v_{n+1}) = c_1\alpha_1 + \cdots + c_n\alpha_n$ can be  represented by an array of $c_j$ dots in the $j$-th column, with the dots drawn so that the column is justified upward.
A positive root $\alpha_e = \alpha_{(i,j)} = \alpha_i+\cdots + \alpha_{j-1} \in \Phi_G^+$ can then be viewed as a line segment that joins dots in consecutive columns, from the $i$-th column to the $(j-1)$-th column.
So, given a partition $\bv=\sum_{\alpha_{(i,j)}\in \Phi_G^+} p_{(i,j)}[\alpha_{(i,j)}]$ of the vector $\bv$ using roots from $\Phi_G^+$, a {\em line-dot diagram} for $\bv$ with respect to $\Phi_G^+$ is a pictorial representation of the vector partition that consists of the array of dots for $\bv$, and $p_{(i,j)}$ line segments from the $i$-th column to the $(j-1)$-th column for each edge $(i,j)\in E(G)$ in which each dot is incident to at most one nontrivial line segment. We consider a single dot to be a line of length zero, or a trivial line segment.
\end{defn}

Of course, a given vector partition may have multiple line-dot diagram representations.  Two line-dot diagrams are {\em equivalent} if they represent the same vector partition,  and a {\em gravity diagram} is a representative of an equivalence class of line-dot diagrams. Let $\calG_G(\bv)$ denote a set of gravity diagrams for the vector $\bv$ with respect to the graph $G$.

\begin{theorem}[{\cite[Theorem 3.1]{BGHHKMY}}] \label{thm.gravitydiagrams}
Let $G$ be a directed graph with $n+1$ vertices and whose edges are directed from $i$ to $j$ if $i<j$. For any vector $\bv = (v_1,\ldots, v_{n+1})$ such that $\sum_{i=1}^{n+1}v_i=0$, 
$$K_G(\bv) = |\outgrav_G(\bv)|.$$
\end{theorem}

By Corollary~\ref{cor.main}, we see that the volume of the flow polytope $\calF_G(1,0,\ldots,0,-1)$ can be computed by counting a set of associated gravity diagrams, provided that they can be described systematically.
There are two vectors which are most pertinent to the study of volumes of flow polytopes with unit flow $\ba=(1,0,\ldots,0,-1)$.

\begin{defn} Given a directed graph $G$ with shifted out-degree vector $\bt$,
the set of {\em out-degree gravity diagrams} of $G$ is a set of gravity diagrams for the vector 
\begin{align*}
\bv_{\mathrm{out}}
&= (m-n-t_1,-t_2,\ldots, -t_n,0)\\
&= (t_2+\cdots+t_n)\alpha_1 + (t_3+\cdots+t_n)\alpha_2 + \cdots + t_{n}\alpha_{n-1}
\end{align*}
with respect to the set of positive roots $\Phi_G^+$.  This is denoted by $\outgrav_G(\bv_{\mathrm{out}})$.

In a similar vein, given a directed graph $G$ with shifted in-degree vector $\bu$, the set of {\em in-degree gravity diagrams} of $G$ is a set of gravity diagrams for the vector
\begin{align*}
\bv_{\mathrm{in}}
&=(0,u_2, \ldots, u_n, u_{n+1}-(m-n))\\
&= u_2\alpha_2 + (u_2+u_3)\alpha_3 + \cdots + (u_2+\cdots+u_n)\alpha_n
\end{align*}
with respect to the set of positive roots $\Phi_G^+$.  This is denoted by $\ingrav_G(\bv_{\mathrm{in}})$.
\end{defn}

\begin{corollary} \label{cor.gravitydiagrams}
Combining Corollary~\ref{cor.main} and Theorem~\ref{thm.gravitydiagrams}, the volume of the flow polytope of $G$ with unit flow $\ba=(1,0,\ldots,0,-1)$ is equal to the number of out-degree gravity diagrams of $G$ and the number of in-degree gravity diagrams of $G$.
$$\vol\calF_G(1,0,\ldots,0,-1) = |\outgrav_G(\bv_{\mathrm{out}})| = |\ingrav_G(\bv_{\mathrm{in}})|. $$
\end{corollary}

In the next sections, we describe a canonical way to define out-degree and in-degree gravity diagrams for the $k$-caracol family of graphs.  We note that some of our conventions differ from the ones originally chosen for the (classical) caracol graph $\Car_{n+1} = \car{n+1}{1}$ in~\cite{BGHHKMY}.

\subsection{Gravity diagrams for the $k$-caracol graphs} \label{sec.gravity}
The $k$-caracol graph $\car{n+1}{k}$ has $n+1$ vertices and $m=(k+1)(n-k)+n-2$ edges. Its shifted out-degree vector $\bt$ and shifted in-degree vector $\bu$ are
\begin{align*}
\bt 
&= (\underbrace{n-k, \ldots, n-k}_{k-1}, n-k-1, \underbrace{1,\ldots, 1}_{n-k-1},0 ) \qquad\hbox{ and} 
\\
\bu 
&= (\underbrace{0, \ldots, 0}_{k-1}, k-1, \underbrace{k,\ldots, k}_{n-k-1}, n-k-1 ),
\end{align*}
and their coordinates sum to $m-n=(k+1)(n-k)-2$.  We also have
\begin{align*}
\bv_{\mathrm{out}} 
&= \sum_{j=1}^{k-1} ((k+1-j)(n-k)-2)\alpha_j + \sum_{j=k}^{n-2} (n-j-1)\alpha_j
\qquad\hbox{ and} \\
\bv_{\mathrm{in}} 
&= \sum_{j=k+1}^{n} ((j-k)k-1)\alpha_{j}.
\end{align*}

The in-degree gravity diagrams for $\bv_{\mathrm{in}}$ are defined on a triangular array of $(j-k)k-1$ dots in the $j$-th column for $j=k+1,\ldots,n$.  Since the dots lie in columns indexed by $\alpha_j$ only for $j=k+1,\ldots, n$, then for the purposes of defining the set of in-degree gravity diagrams, we only need to consider the positive roots which correspond to the edges in the graph $G=\car{n+1}{k}$ when restricted to the vertex set $\{k+1,k+2,\ldots,n+1\}$.  These positive roots are
$$\Phi_{G|_{V=\{k+1,\ldots, n+1\}}}^+ = \{ \alpha_j\}_{j=k+1}^n \cup \{\alpha_j + \cdots +\alpha_n \}_{j=k+1}^{n-1}, $$ 
so we see that each nontrivial line segment must end in the $n$-th column.
This leads us to choose the following conventions for the in-degree gravity diagrams:
\begin{enumerate}
\item[(a)] each line segment must be horizontal,
\item[(b)] a longer line segment must be in a row above that of a shorter line segment.
\end{enumerate}
This uniquely defines a representative for each equivalence class of in-degree line-dot diagrams for $\car{n+1}{k}$.
See Figures~\ref{fig.car61} and~\ref{fig.car62} for some examples.

The out-degree gravity diagrams for $\bv_{\mathrm{out}}$ are defined on the array consisting of $(k+1-j)(n-k)-2$ dots in the $j$-th column for $j=1,\ldots, k-1$, and $n-j-1$ dots in the $j$-th column for $j=k,\ldots, n-2$, where the latter portion forms a right-triangular array. 
Given this, we only need to consider the positive roots which correspond to the edges in the graph $G=\car{n+1}{k}$ when restricted to the vertex set $\{1,2,\ldots, n-1\}$.

From this, we see that every nontrivial line segment for the out-degree gravity diagram begins an $i$-th column for some $i=1,\ldots, k$, and ends in a $j$-th column for some $j = k, \ldots, n-2$. In other words, every nontrivial line segment contains a dot from the $k$-th column of the array, and this leads us to choose the following conventions for the out-degree gravity diagrams: 
\begin{enumerate}
\item[(a)] each line segment must be horizontal,
\item[(b)] the line segments are ordered from top to bottom so that the line segments which end at the $q$-th column are above the line segments which end at the $p$-th column if $q>p$, and if two line segments end at the same column, then the longer line segment is above the shorter line segment.
\end{enumerate}
This uniquely defines a representative for each equivalence class of out-degree line-dot diagrams for $\car{n+1}{k}$.
See Figures~\ref{fig.car61} and~\ref{fig.car62} for some examples.

\begin{figure}[ht!]
\begin{center}
\begin{tikzpicture}[scale=0.5]
\tikzstyle{every node}+=[fill, circle, inner sep=0, minimum size=4.5pt]
\begin{scope}[xshift=0]
\vertex[fill, label=above:\tiny$\alpha_1$](a11) at (0,0) {}; 
\vertex[fill, label=above:\tiny$\alpha_2$](a12) at (1,0) {}; 
\vertex[fill, label=above:\tiny$\alpha_3$](a13) at (2,0) {}; 
\node(a21) at (0,-1) {};
\node(a22) at (1,-1) {};
\node(a31) at (0,-2) {};
\end{scope}

\begin{scope}[xshift=100]
\vertex[fill, label=above:\tiny$\tiny\alpha_1$](a11) at (0,0) {}; 
\vertex[fill, label=above:\tiny$\alpha_2$](a12) at (1,0) {}; 
\vertex[fill, label=above:\tiny$\alpha_3$](a13) at (2,0) {}; 
\node(a21) at (0,-1) {};
\node(a22) at (1,-1) {};
\node(a31) at (0,-2) {};
\draw (a11)--(a12);
\end{scope}

\begin{scope}[xshift=200]
\vertex[fill, label=above:\tiny$\alpha_1$](a11) at (0,0) {}; 
\vertex[fill, label=above:\tiny$\alpha_2$](a12) at (1,0) {}; 
\vertex[fill, label=above:\tiny$\alpha_3$](a13) at (2,0) {}; 
\node(a21) at (0,-1) {};
\node(a22) at (1,-1) {};
\node(a31) at (0,-2) {};
\draw (a11)--(a13);
\end{scope}

\begin{scope}[xshift=300]
\vertex[fill, label=above:\tiny$\alpha_1$](a11) at (0,0) {}; 
\vertex[fill, label=above:\tiny$\alpha_2$](a12) at (1,0) {}; 
\vertex[fill, label=above:\tiny$\alpha_3$](a13) at (2,0) {}; 
\node(a21) at (0,-1) {};
\node(a22) at (1,-1) {};
\node(a31) at (0,-2) {};
\draw (a11)--(a12);
\draw (a21)--(a22);
\end{scope}

\begin{scope}[xshift=400]
\vertex[fill, label=above:\tiny$\alpha_1$](a11) at (0,0) {}; 
\vertex[fill, label=above:\tiny$\alpha_2$](a12) at (1,0) {}; 
\vertex[fill, label=above:\tiny$\alpha_3$](a13) at (2,0) {}; 
\node(a21) at (0,-1) {};
\node(a22) at (1,-1) {};
\node(a31) at (0,-2) {};
\draw (a11)--(a13);
\draw (a21)--(a22);
\end{scope}
%
%
\begin{scope}[xshift=0, yshift=-120]
\vertex[fill, label=above:\tiny$\alpha_3$](a11) at (0,0) {}; 
\vertex[fill, label=above:\tiny$\alpha_4$](a12) at (1,0) {}; 
\vertex[fill, label=above:\tiny$\alpha_5$](a13) at (2,0) {}; 
\node(a22) at (1,-1) {};
\node(a23) at (2,-1) {};
\node(a33) at (2,-2) {};
\end{scope}

\begin{scope}[xshift=100, yshift=-120]
\vertex[fill, label=above:\tiny$\alpha_3$](a11) at (0,0) {}; 
\vertex[fill, label=above:\tiny$\alpha_4$](a12) at (1,0) {}; 
\vertex[fill, label=above:\tiny$\alpha_5$](a13) at (2,0) {}; 
\node(a22) at (1,-1) {};
\node(a23) at (2,-1) {};
\node(a33) at (2,-2) {};
\draw (a12)--(a13);
\end{scope}

\begin{scope}[xshift=200, yshift=-120]
\vertex[fill, label=above:\tiny$\alpha_3$](a11) at (0,0) {}; 
\vertex[fill, label=above:\tiny$\alpha_4$](a12) at (1,0) {}; 
\vertex[fill, label=above:\tiny$\alpha_5$](a13) at (2,0) {}; 
\node(a22) at (1,-1) {};
\node(a23) at (2,-1) {};
\node(a33) at (2,-2) {};
\draw (a12)--(a13);
\draw (a22)--(a23);
\end{scope}

\begin{scope}[xshift=300, yshift=-120]
\vertex[fill, label=above:\tiny$\alpha_3$](a11) at (0,0) {}; 
\vertex[fill, label=above:\tiny$\alpha_4$](a12) at (1,0) {}; 
\vertex[fill, label=above:\tiny$\alpha_5$](a13) at (2,0) {}; 
\node(a22) at (1,-1) {};
\node(a23) at (2,-1) {};
\node(a33) at (2,-2) {};
\draw (a11)--(a13);
\end{scope}

\begin{scope}[xshift=400, yshift=-120]
\vertex[fill, label=above:\tiny$\alpha_3$](a11) at (0,0) {}; 
\vertex[fill, label=above:\tiny$\alpha_4$](a12) at (1,0) {}; 
\vertex[fill, label=above:\tiny$\alpha_5$](a13) at (2,0) {}; 
\node(a22) at (1,-1) {};
\node(a23) at (2,-1) {};
\node(a33) at (2,-2) {};
\draw (a11)--(a13);
\draw (a22)--(a23);
\end{scope}
\end{tikzpicture}
\end{center}
\caption{The out-degree and in-degree gravity diagrams for $\car61$.}
\label{fig.car61}
\end{figure}

\begin{figure}[ht!]
\begin{center}
\begin{tikzpicture}[scale=0.5]
\tikzstyle{every node}+=[fill, circle, inner sep=0, minimum size=4.5pt]
\begin{scope}[scale=1]
	\vertex[fill, label=above:\tiny$\alpha_1$](a00) at (0,0) {}; 	
	\vertex[fill, label=above:\tiny$\alpha_2$](a10) at (1,0) {}; 
	\vertex[fill, label=above:\tiny$\alpha_3$](a20) at (2,0) {};	
	\vertex(a01) at (0,-1) {}; \vertex(a02) at (0,-2) {}; 	
	\vertex(a03) at (0,-3) {}; \vertex(a11) at (1,-1) {};	
\end{scope}
\begin{scope}[scale=1, xshift=100]
	\vertex[fill, label=above:\tiny$\alpha_1$](a00) at (0,0) {}; 	
	\vertex[fill, label=above:\tiny$\alpha_2$](a10) at (1,0) {}; 
	\vertex[fill, label=above:\tiny$\alpha_3$](a20) at (2,0) {};	
	\vertex(a01) at (0,-1) {}; \vertex(a02) at (0,-2) {}; 	
	\vertex(a03) at (0,-3) {}; \vertex(a11) at (1,-1) {};	
	\draw (a00) to (a10);
\end{scope}
\begin{scope}[scale=1, xshift=200]
	\vertex[fill, label=above:\tiny$\alpha_1$](a00) at (0,0) {}; 	
	\vertex[fill, label=above:\tiny$\alpha_2$](a10) at (1,0) {}; 
	\vertex[fill, label=above:\tiny$\alpha_3$](a20) at (2,0) {};	
	\vertex(a01) at (0,-1) {}; \vertex(a02) at (0,-2) {}; 	
	\vertex(a03) at (0,-3) {}; \vertex(a11) at (1,-1) {};	
	\draw (a10) to (a20);
\end{scope}
\begin{scope}[scale=1, xshift=300]
	\vertex[fill, label=above:\tiny$\alpha_1$](a00) at (0,0) {}; 	
	\vertex[fill, label=above:\tiny$\alpha_2$](a10) at (1,0) {}; 
	\vertex[fill, label=above:\tiny$\alpha_3$](a20) at (2,0) {};	
	\vertex(a01) at (0,-1) {}; \vertex(a02) at (0,-2) {}; 	
	\vertex(a03) at (0,-3) {}; \vertex(a11) at (1,-1) {};	
	\draw (a00) to (a10);
	\draw (a01) to (a11); 
\end{scope}
\begin{scope}[scale=1, xshift=400]
	\vertex[fill, label=above:\tiny$\alpha_1$](a00) at (0,0) {}; 	
	\vertex[fill, label=above:\tiny$\alpha_2$](a10) at (1,0) {}; 
	\vertex[fill, label=above:\tiny$\alpha_3$](a20) at (2,0) {};	
	\vertex(a01) at (0,-1) {}; \vertex(a02) at (0,-2) {}; 	
	\vertex(a03) at (0,-3) {}; \vertex(a11) at (1,-1) {};	
	\draw(a01) to (a11);
	\draw(a10) to (a20); 
\end{scope}
\begin{scope}[scale=1, xshift=500]
	\vertex[fill, label=above:\tiny$\alpha_1$](a00) at (0,0) {}; 	
	\vertex[fill, label=above:\tiny$\alpha_2$](a10) at (1,0) {}; 
	\vertex[fill, label=above:\tiny$\alpha_3$](a20) at (2,0) {};	
	\vertex(a01) at (0,-1) {}; \vertex(a02) at (0,-2) {}; 	
	\vertex(a03) at (0,-3) {}; \vertex(a11) at (1,-1) {};	
	\draw (a00) to (a20);
\end{scope}
\begin{scope}[scale=1, xshift=600]
	\vertex[fill, label=above:\tiny$\alpha_1$](a00) at (0,0) {}; 	
	\vertex[fill, label=above:\tiny$\alpha_2$](a10) at (1,0) {}; 
	\vertex[fill, label=above:\tiny$\alpha_3$](a20) at (2,0) {};	
	\vertex(a01) at (0,-1) {}; \vertex(a02) at (0,-2) {}; 	
	\vertex(a03) at (0,-3) {}; \vertex(a11) at (1,-1) {};	
	\draw (a00) to (a20);
	\draw (a01) to (a11); 
\end{scope}

\begin{scope}[scale=1, yshift=-150]
	\vertex[fill, label=above:\tiny$\alpha_3$] at (0,0) {}; 	
	\vertex[fill, label=above:\tiny$\alpha_4$] at (1,0) {}; 
		\vertex at (1,-1) {}; \vertex at (1,-2) {};
	\vertex[fill, label=above:\tiny$\alpha_5$] at (2,0) {}; 
		\vertex at (2,-1) {}; \vertex at (2,-2) {};
		\vertex at (2,-3) {}; \vertex at (2,-4) {};
\end{scope}
\begin{scope}[scale=1, xshift=100, yshift=-150]
	\vertex[fill, label=above:\tiny$\alpha_3$] at (0,0) {}; 	
	\vertex[fill, label=above:\tiny$\alpha_4$] at (1,0) {}; 
		\vertex at (1,-1) {}; \vertex at (1,-2) {};
	\vertex[fill, label=above:\tiny$\alpha_5$] at (2,0) {}; 
		\vertex at (2,-1) {}; \vertex at (2,-2) {};
		\vertex at (2,-3) {}; \vertex at (2,-4) {}; 
	\draw (2,0)--(1,0); 
\end{scope}
\begin{scope}[scale=1, xshift=200, yshift=-150]
	\vertex[fill, label=above:\tiny$\alpha_3$] at (0,0) {}; 	
	\vertex[fill, label=above:\tiny$\alpha_4$] at (1,0) {}; 
		\vertex at (1,-1) {}; \vertex at (1,-2) {};
	\vertex[fill, label=above:\tiny$\alpha_5$] at (2,0) {}; 
		\vertex at (2,-1) {}; \vertex at (2,-2) {};
		\vertex at (2,-3) {}; \vertex at (2,-4) {}; 
	\draw (2,0)--(1,0); \draw (2,-1)--(1,-1);
\end{scope}
\begin{scope}[scale=1, xshift=300, yshift=-150]
	\vertex[fill, label=above:\tiny$\alpha_3$] at (0,0) {}; 	
	\vertex[fill, label=above:\tiny$\alpha_4$] at (1,0) {}; 
		\vertex at (1,-1) {}; \vertex at (1,-2) {};
	\vertex[fill, label=above:\tiny$\alpha_5$] at (2,0) {}; 
		\vertex at (2,-1) {}; \vertex at (2,-2) {};
		\vertex at (2,-3) {}; \vertex at (2,-4) {}; 
	\draw (2,0)--(0,0);		
\end{scope}
\begin{scope}[scale=1, xshift=400, yshift=-150]
	\vertex[fill, label=above:\tiny$\alpha_3$] at (0,0) {}; 	
	\vertex[fill, label=above:\tiny$\alpha_4$] at (1,0) {}; 
		\vertex at (1,-1) {}; \vertex at (1,-2) {};
	\vertex[fill, label=above:\tiny$\alpha_5$] at (2,0) {}; 
		\vertex at (2,-1) {}; \vertex at (2,-2) {};
		\vertex at (2,-3) {}; \vertex at (2,-4) {}; 
	\draw (2,0)--(0,0); \draw (2,-1)--(1,-1);		
\end{scope}
\begin{scope}[scale=1, xshift=500, yshift=-150]
	\vertex[fill, label=above:\tiny$\alpha_3$] at (0,0) {}; 	
	\vertex[fill, label=above:\tiny$\alpha_4$] at (1,0) {}; 
		\vertex at (1,-1) {}; \vertex at (1,-2) {};
	\vertex[fill, label=above:\tiny$\alpha_5$] at (2,0) {}; 
		\vertex at (2,-1) {}; \vertex at (2,-2) {};
		\vertex at (2,-3) {}; \vertex at (2,-4) {}; 
	\draw (2,0)--(1,0); \draw (2,-1)--(1,-1); \draw (2,-2)--(1,-2);
\end{scope}
\begin{scope}[scale=1, xshift=600, yshift=-150]
	\vertex[fill, label=above:\tiny$\alpha_3$] at (0,0) {}; 	
	\vertex[fill, label=above:\tiny$\alpha_4$] at (1,0) {}; 
		\vertex at (1,-1) {}; \vertex at (1,-2) {};
	\vertex[fill, label=above:\tiny$\alpha_5$] at (2,0) {}; 
		\vertex at (2,-1) {}; \vertex at (2,-2) {};
		\vertex at (2,-3) {}; \vertex at (2,-4) {}; 
	\draw (2,0)--(0,0);
	\draw (2,-1)--(1,-1); \draw (2,-2)--(1,-2);
\end{scope}

\end{tikzpicture}
\end{center}
\caption{The out-degree and in-degree gravity diagrams for $\car62$.}
\label{fig.car62}
\end{figure}

\begin{remark}\label{rem.trapezoid}
An out-degree gravity diagram of $\car{n+1}{k}$ is defined on an array consisting of $(k+1-j)(n-k)-2$ dots in the $j$-th column for $j=1,\ldots, k-1$, and $n-j-1$ dots in the $j$-th column for $j=k,\ldots, n-2$.  However, note that we can truncate the dots in the first $k-1$ columns of the out-degree gravity diagram below the first $n-k-1$ rows of dots without loss of generality since every nontrivial line segment must contain a dot from the column indexed by $\alpha_k$, and so no line segments can be drawn on those dots below the first $n-k-1$ rows.  In other words, we view the out-degree gravity diagrams of $\car{n+1}{k}$ as a trapezoidal array of $k-1+i$ dots in the $i$-th row for $i=1,\ldots, n-k-1$. See the left side of Figure~\ref{fig.3-car_bijection_out_d} for an example.
\end{remark}

\subsection{Fuss-Catalan volumes}

In the paper~\cite{BGHHKMY}, we computed the volume of the flow polytope of 
the caracol graph $\car{n+1}{1}$ with unit flow $\ba=(1,0,\ldots, 0, -1)$ by describing a bijection between its out-degree gravity diagrams and a set of Dyck paths.
We now generalize this method and compute the volume of the flow polytope of the $k$-caracol graph $\car{n+1}{k}$ with unit flow $\ba=(1,0,\ldots, 0, -1)$ in two ways.  The first is a bijection between in-degree gravity diagrams and a set of rational Catalan Dyck paths, and the second is a bijection between the out-degree gravity diagrams and the same set of rational Catalan Dyck paths.

Before we do this, we introduce some basic background on rational Catalan combinatorics, which is a generalization due to Armstrong, Loehr, and Warrington~\cite{ALW} of the classical Catalan numbers.  

\begin{defn} Let $a,b$ be nonnegative integers such that $b\geq a$.
A {\em lattice path} from $(0,0)$ to $(b,a)$ is a path comprised of $a$ north steps $N=(0,1)$ and $b$ east steps $E=(1,0)$. We may equivalently represent the lattice path as a word $N^{s_1}EN^{s_2}E \cdots N^{s_b}E$, so that $\bs = (s_1,\ldots, s_b)\in \bbZ_{\geq0}$ is a weak composition of $|\bs| = s_1+\cdots + s_b=a$ of length $\ell(\bs) = b$.  In this paper, we will often view lattice paths as weak compositions.
\end{defn}

\begin{defn}
Given two weak compositions $\bs=(s_1,\ldots, s_b)\vDash a$ and $\bt = (t_1,\ldots, t_b) \vDash a$, we say that $\bs$ {\em dominates} $\bt$ and we write $\bs \rhd \bt$
if $s_1+ \cdots + s_j \geq t_1 + \cdots + t_j$ for each $j=1,\ldots, b$.
A {\em $\bt$-Dyck path} is a weak composition $\bs = (s_1,\ldots, s_b)$ that dominates $\bt$.
\end{defn}
Visually, $\bs = (s_1,\ldots, s_b)\vDash a$ is the lattice path $N^{s_1}E \cdots N^{s_b}E$ on the rectangular grid from $(0,0)$ to $(b, a)$.
On this grid, the composition $\bt = (t_1,\ldots, t_b)\vDash a$ is represented by shading $t_j$ squares in the $j$-th column of squares, starting at height $t_1+\cdots+t_{j-1}$. The set of $\bt$-Dyck paths is then the set of lattice paths from $(0,0)$ to $(b, a)$ which lie above the shaded $\bt$-region. 
The {\em area} of a $\bt$-Dyck path is the number of squares lying between the path and the shaded $\bt$-region.

\begin{defn}
For coprime positive integers $b>a$, a {\em rational $(a,b)$-Dyck path} 
is a lattice path from $(0,0)$ to $(b,a)$ in the integer lattice $\bbZ^2$ comprised of north steps $N=(0,1)$ and east steps $E=(1,0)$ that stays above the diagonal line from $(0,0)$ to $(b,a)$. Let $\calD(a,b)$ denote the set of rational $(a,b)$-Dyck paths.
\end{defn}
\begin{remark} \label{rem.rationalDyck}
Rational $(a,b)$-Dyck paths are a special case of $\bt$-Dyck paths. By shading the squares on the $b$ by $a$ grid which intersect the line $y=\frac{a}{b}x$, we obtain the (row) signature $(r_1,\ldots, r_a)$ of the path, where $r_i$ is the number of shaded squares in the $i$-th row. The associated weak composition $\bt$ is then the transpose of $(r_1-1,\ldots, r_{a-1}-1, r_a)$.  In the proof of Theorem~\ref{thm.onezerozero}, we will use the fact that for $a=n-k$ and $b=k(n-k)-1$, rational $(a,b)$-Dyck paths are $\bt$-Dyck paths where $\bt$ is the transpose of $(k-1,k^{n-k-1})$.
\end{remark}

\begin{defn} For coprime positive integers $b>a$, the {\em rational Catalan number}
\begin{equation}
\Cat(a,b)=\frac{1}{a+b}{a+b\choose a} = \frac{1}{b}{a+b-1\choose a} = \frac{1}{a}{a+b-1\choose b}
\end{equation}
enumerates rational $(a,b)$-Dyck paths~\cite[Section 3.2]{ALW}.
\end{defn}
\begin{remark} Two well-known special cases of the rational Catalan numbers are the classical Catalan numbers
$$\Cat(n) = \Cat(n,n+1) = \frac{1}{2n+1}{2n+1\choose n} 
	= \frac{1}{n+1}{2n\choose n}
	= \frac{1}{n}{2n\choose n+1},$$
and the classical Fuss-Catalan numbers
$$\Cat(n,kn+1) = \frac{1}{(k+1)n+1}{(k+1)n+1\choose n} 
	= \frac{1}{kn+1}{(k+1)n\choose n}
	= \frac{1}{n}{(k+1)n\choose kn+1},$$
	for $k\in\bbN$.
\end{remark}

It turns out that the volume of the flow polytope of the $k$-caracol graph with unit flow $\ba=(1,0,\ldots,0,-1)$ is a generalized Fuss-Catalan number.
\begin{theorem} \label{thm.onezerozero}
For $k\in \bbN$ and $n>k$,
$$\vol \calF_{\car{n+1}{k}}(1,0,\ldots,0,-1) 
=\Cat(n-k,k(n-k)-1).$$
\end{theorem}

\begin{proof} Let $a=n-k$ and $b=ka-1$.
We construct a bijection $\Psi_{\mathrm{in}}:\ingrav_G(\bv_{\mathrm{in}}) \rightarrow\calD(a,b)$ from the set of in-degree gravity diagrams of $\car{n+1}{k}$ to the set of rational $(a,b)$-Dyck paths. 

Recall from Section~\ref{sec.gravity} that an in-degree gravity diagram of $\car{n+1}{k}$ is defined on a triangular array of $(j-k)k-1$ dots in the $j$-column for $j=k+1,\ldots, n$.  Given an in-degree gravity diagram $\Gamma$, we embed 
it into the squares of the $\bbZ^2$ grid by rotating $\Gamma$ counterclockwise by ninety degrees, aligned so that the dots in the column indexed by $\alpha_n$ lie in the squares just above the line $y=a$.  See Figure~\ref{fig.3-car_bijection_in} for an illustration.

As noted in Remark~\ref{rem.rationalDyck}, the set of $(a,b)$-Dyck paths is the set of $\bt$-Dyck paths where $\bt$ is the transpose of $(k-1,k^{n-k-1})$.  With this interpretation, one can see that the columns of $\Gamma$ embed into the squares of $\bbZ^2$ precisely so that the lower boundary of $\Gamma$ consists of the shaded squares in the $(a,b)$-Dyck path diagram.

Line segments of the embedded $\Gamma$ extend along columns from the top row of the Dyck path diagram, and by the convention chosen for the in-degree gravity diagrams, the lengths of these columns are non-increasing from left to right.  Thus, the line segments of $\Gamma$ define a unique rational $(a,b)$-Dyck path associated to $\Gamma$.  Conversely, any $(a,b)$-Dyck path defines an in-degree gravity diagram for $\car{n+1}{k}$ whose line segments occupy every square on the northwest side of the Dyck path.

Therefore, $|\ingrav_{\car{n+1}{k}}(\bv_{\mathrm{in}})| = |\calD(a,b)| = \Cat(a,b)$, and 
we conclude by Corollary~\ref{cor.gravitydiagrams} that $\vol\calF_{\car{n+1}{k}}(1,0,\ldots,0,-1) = \Cat(n-k,k(n-k)-1)$.
\end{proof}

\begin{corollary} We recover the following formulas as special cases.
At $k=1$, 
$$\vol \calF_{\Car_{n+1}}(1,0,\ldots, 0,-1) = \Cat(n-1, n-2) = \frac{1}{n-1}{2n-4\choose n-2}$$
is a classical Catalan number.
At $k=n-1$,
$$\vol \calF_{\PS_n}(1,0,\ldots, 0,-1)
= \vol \calF_{\car{n+1}{n-1}}(1,0,\ldots, 0,-1) = \Cat(1, n-2) = 1.$$
\end{corollary}
\begin{proof}
Earlier, we observed that when $k=n-1$, the graph $\car{n+1}{n-1}$ is the graph $\PS_n$ with an extra edge $(n,n+1)$. Note that this edge does not affect the equations defining the polytope, so
$\calF_{\PS_n}(a_1,\ldots,a_{n-1}, \hbox{$-\sum_{i=1}^{n-1} a_i$}) 
= \calF_{\car{n+1}{n-1}}(a_1,\ldots,a_{n-1}, a_n, \hbox{$-\sum_{i=1}^{n} a_i$}).$
\end{proof}

\begin{remark}
M\'esz\'aros~\cite{M} developed a method for expressing the volumes of flow polytopes with unit flow as the number of certain triangular arrays, and as an application, used it to construct a family of flow polytopes $\calF_G$ with Fuss-Catalan volume $\Cat(a, ka+1)$.
Using relationship between rational-Dyck paths and in-degree gravity diagrams as a guide, then the graph $H_{n+1}^{(k)}$, obtained by taking $G=\car{n+1}{k}$ and adding one more copy of the edge $(k, k+1)$, has shifted in-degree vector $\bu= (0, k^{n-k}, n-k)$, and $\bv_{\mathrm{in}}=\sum_{j=k+1}^n (j-k)k\alpha_j$. This means that an in-degree gravity diagram for $H_{n+1}^{(k)}$ can be embedded in the squares of a $k(n-k)$ by $n-k$ grid. We can, without loss of generality, extend this to a $b=k(n-k)+1$ by $a=n-k$ grid to ensure that $a$ and $b$ are coprime and the bijection between the rational $(a,b)$-Dyck paths and in-degree gravity diagrams will remain unchanged.  We hope to explore this variation of the $k$-caracol graphs in future work.
\end{remark}

By Corollary~\ref{cor.gravitydiagrams}, we can obtain a second proof of Theorem~\ref{thm.onezerozero} by constructing a bijection from the set of out-degree gravity diagrams to the same set of rational Dyck paths as above.

\begin{defn} We set some notation that will be used in the proof of the next result.
Recall from Remark~\ref{rem.trapezoid} that without loss of generality, we may consider the out-degree gravity diagrams for $\car{n+1}{k}$ to be defined on a trapezoidal array of $k-1+i$ dots in the $i$-th row for $i=1,\ldots,n-k-1$. Since the line segments of the gravity diagram are horizontal, we let $L_i=[\ell_i,r_i]$ denote the line segment in the $i$-th row, from the $\ell_i$-th column to the $r_i$-th column. The {\em length} of $L_i$ is $d(L_i) = r_i-\ell_i$. 
\end{defn}

For $k\in \bbN$ and $n>k$, let $a=n-k$ and $b=ka-1$.
We will define a map $\Psi_{\mathrm{out}}:\outgrav_G(\bv_{\mathrm{out}})\rightarrow \calD(a,b)$ (Definition~\ref{defn.psiout}) from the set of out-degree gravity diagrams for $G=\car{n+1}{k}$ to the set of rational $(a,b)$-Dyck paths in several steps.

Let $\Gamma\in \outgrav_G(\bv_{\mathrm{out}})$ be an out-degree gravity diagram with line segments $L_1,\ldots, L_{a-1}$.  
Again, we view an $(a,b)$-Dyck path as a $\bt$-Dyck path where $\bt$ is the transpose of $(k-1, k^{a-1})$.  Let $Z$ denote the $\bbZ^2$ grid from $(0,0)$ to $(b,a)$, with the $\bt$-region shaded (this is the lattice on which we can draw an $(a,b)$-Dyck path).  Note that $Z$ has exactly $a-1$ nonempty rows of squares lying above its shaded $\bt$-region, so we will show that we can embed the line segments of $\Gamma$ into the rows of squares of $Z$ appropriately, which in turn will define the $(a,b)$-Dyck path associated to $\Gamma$.

To begin with, we label the $j(k-1)$-th column of squares of $Z$ by $\alpha_{k+j}$, for $j=0,\ldots, a-2$. The zeroth column lies to the left of the diagram.  See the right side of Figure~\ref{fig.3-car_bijection_out_d} for an example.
We embed the line segment $L_i=[\ell_i,r_i]$ into the $(i+1)$-th row of squares of $Z$ by placing its left endpoint in the column indexed by $\alpha_{r_i}$.

To see that this procedure indeed embeds $L_i$ into the squares lying above the shaded $\bt$-region of $Z$, we first list a few properties which are satisfied by these embedded line segments.
\begin{lemma}\label{lem.lineproperties} 
Let $L_i=[\ell_i,r_i]$ be the line segment in the $i$-th row of an out-degree gravity diagram $\Gamma\in \outgrav_G(\bv_{\mathrm{out}})$. Let  $\lp(L_i)$, respectively $\rp(L_i)$,  denote the column of the Dyck path diagram $Z$ that is occupied by the left (respectively right) endpoint of the embedded line segment $L_i$. Then
\begin{enumerate}
\item[(a)] $1\leq \ell_i\leq k$ and $k\leq r_i\leq k+i-1$,
\item[(b)] $r_i-k \leq d(L_i) \leq r_i-1 \leq k+i-2$,
\item[(c)] $\lp(L_i) \in \{ k-1, 2(k-1), \ldots, (i-1)(k-1)\}$, and
\item[(d)] $\rp(L_i) \leq ik-1$.
\end{enumerate}
\end{lemma}
\begin{proof}
Parts (a) and (b) follow directly from the conventions for $\Gamma$ as a gravity diagram. Part (c) holds because $\lp(L_i)$ occupies the column labeled by $\alpha_{r_i}$, which is the $(r_i-k)(k-1)$-th column of $Z$, and by part (a), $k\leq r_i\leq k+i-1$. Finally, part (d) follows because by part (c), the rightmost column which can be occupied by $\lp(L_i)$ is $(i-1)(k-1)$, and by part (b), the maximum length of $L_i$ is $k+i-2$, so $\rp(L_i) \leq (i-1)(k-1) + k+i-2 = ik-1$.
\end{proof}

Note that there are precisely $ik-1$ squares lying in the $(i+1)$-th row of $Z$, above the shaded $\bt$-region, so by part (d) of the above Lemma, each line segment $L_i$ of $\Gamma$ is embedded into the squares lying above the shaded $\bt$-region of $Z$, as claimed.

\begin{lemma} \label{lem.rightendpoints}
Let $G=\car{n+1}{k}$, and let $\Gamma\in \outgrav_G(\bv_{\mathrm{out}})$ be an out-degree gravity diagram with line segments $L_1,\ldots, L_{a-1}$.  Then $\rp(L_1) \leq \cdots \leq \rp(L_{a-1})$.
\end{lemma}
\begin{proof}
We proceed by induction on $a$.  
The base cases are for $k\geq1$ and $a=n-k=1$.  In these cases, the only out-degree gravity diagram is the empty diagram, and the only $(1,b)$-Dyck path is $NE^b$, so the base cases hold.
 
Now given $k\geq1$, suppose $\Gamma$ has $a-1$ rows with line segments $L_1,\ldots, L_{a-1}$.  By the induction hypothesis, the line segments $L_1,\ldots, L_{a-2}$ of $\Gamma$ embeds into rows $2$ through $a-1$ of the Dyck path grid $Z$, and the shape of these embedded line segments defines a (partial) rational $(a-1,b-k)$-Dyck path from $(0,0)$ to $(c,a-1)$ for some $0\leq c \leq b-k$.  We now consider embedding the last line segment $L_{a-1}$.

If $L_{a-1}$ and $L_{a-2}$ are embedded into $Z$ so that $\lp(L_{a-1}) = \lp(L_{a-2})$, then by the conventions defining the out-degree gravity diagrams, $d(L_{a-1}) \geq d(L_{a-2})$, and so $\rp(L_{a-1}) \geq \rp(L_{a-2})$.
Otherwise, by construction, $L_{a-1}$ must be embedded so that $\lp(L_{a-1})\geq \lp(L_{a-2})+k-1$.  In other words, if $r_{a-2} = k+h$ for some $h\geq0$, then $r_{a-1} \geq k+h+1$. By part (b) of the previous Lemma, we have $d(L_{a-2})\leq k+h-1$ and $h+1 \leq d(L_{a-1})$. Putting this altogether, 
\begin{align*}
\rp(L_{a-1})
= \lp(L_{a-1})+d(L_{a-1})
&\geq \lp(L_{a-2}) + k-1 + h+1\\
&> \lp(L_{a-2}) + d(L_{a-2})
= \rp(L_{a-2}),
\end{align*} 
so the right endpoint of $L_{a-1}$ lies (strictly) to the right of the right endpoint of $L_{a-2}$ in this case also.  
\end{proof}

\begin{defn} \label{defn.psiout}
Lemma~\ref{lem.rightendpoints} shows that the line segments $L_1,\ldots, L_{a-1}$ of a gravity diagram $\Gamma\in \outgrav_G(\bv_{\mathrm{out}})$ are embedded into the $(a,b)$-Dyck path grid $Z$ so that the right endpoints of the line segments move weakly to the right.  Therefore, we can define $\Psi_{\mathrm{out}}(\Gamma)$ to be the rational $(a,b)$-Dyck path defined by the `rectilinear convex hull' of the embedded line segments of $\Gamma$.

In other words, consider the region of squares that lie above the shaded $\bt$-region as the Ferrers diagram of the partition $\lambda(a,k) = ((a-1)k-1, (a-2)k-1, \ldots, 2k-1, k-1)$. 
Then the right endpoints of the embedded line segments coming from $\Gamma$  define a subpartition of $\lambda$, and this subpartition defines the $(a,b)$-Dyck path associated to $\Gamma$.
\end{defn}
See Example~\ref{eg.pairofDycks} and Figure~\ref{fig.3-car_bijection_out_d} for an illustration.

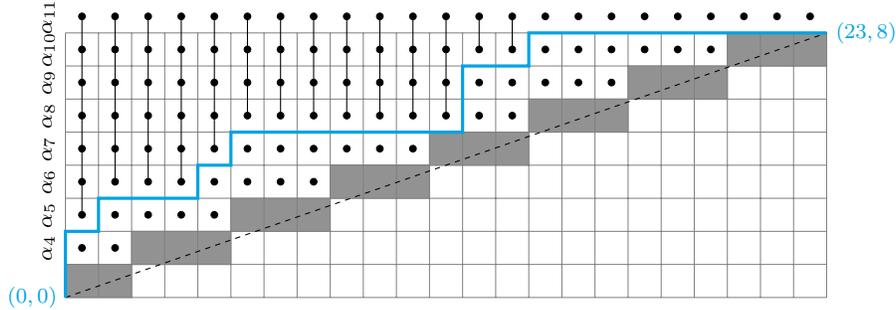
\begin{figure} [ht!]
\begin{center}
\begin{tikzpicture}[scale=0.5]

\begin{scope}[scale=0.88]
	\draw[fill, color=gray!85] (0,0) rectangle (2,1);
	\draw[fill, color=gray!85] (2,1) rectangle (5,2);
	\draw[fill, color=gray!85] (5,2) rectangle (8,3);
	\draw[fill, color=gray!85] (8,3) rectangle (11,4);
	\draw[fill, color=gray!85] (11,4) rectangle (14,5);
	\draw[fill, color=gray!85] (14,5) rectangle (17,6);
	\draw[fill, color=gray!85] (17,6) rectangle (20,7);
	\draw[fill, color=gray!85] (20,7) rectangle (23,8);
	\draw[very thin, color=gray!100] (0,0) grid (23,8);
	
	\foreach \x in {0.5,1.5}
	\foreach \y in {1.5,2.5,3.5,4.5,5.5,6.5,7.5,8.5} 
		\vertex[fill, minimum size=2.5pt] at (\x,\y) {};
	
	\foreach \x in {2.5,3.5,4.5}
		\foreach \y in {2.5,3.5,4.5,5.5,6.5,7.5,8.5} 
		\vertex[fill, minimum size=2.5pt] at (\x,\y) {};
	
	\foreach \x in {5.5,6.5,7.5}
		\foreach \y in {3.5,4.5,5.5,6.5,7.5,8.5} 
		\vertex[fill, minimum size=2.5pt] at (\x,\y) {};
	
	\foreach \x in {8.5,9.5,10.5}
		\foreach \y in {4.5,5.5,6.5,7.5,8.5} 
		\vertex[fill, minimum size=2.5pt] at (\x,\y) {};
	
	\foreach \x in {11.5,12.5,13.5}
		\foreach \y in {5.5,6.5,7.5,8.5} 
		\vertex[fill, minimum size=2.5pt] at (\x,\y) {};
		
	\foreach \x in {14.5,15.5,16.5}
		\foreach \y in {6.5,7.5,8.5}
		\vertex[fill, minimum size=2.5pt] at (\x,\y) {};
	
	\foreach \x in {17.5,18.5,19.5}
		\foreach \y in {7.5,8.5}
		\vertex[fill, minimum size=2.5pt] at (\x,\y) {};
		
	\foreach \x in {20.5,21.5,22.5}
		\vertex[fill, minimum size=2.5pt] at (\x,8.5) {};		
		
	\draw (0.5,8.5) to (0.5,2.5);
	\foreach \x in {1.5,2.5,3.5} \draw (\x,8.5) to (\x,3.5);
	\draw (4.5,8.5) to (4.5,4.5);
	\foreach \x in {5.5,6.5,7.5,8.5,9.5,10.5,11.5} \draw (\x,8.5) to (\x,5.5);
	\foreach \x in {12.5,13.5} \draw (\x,8.5) to (\x,7.5);

	\draw[dash pattern={on 2pt off 2pt}] (0,0) -- (23,8);	
				
	\draw[very thick, color=Cerulean] (0,0)--(0,2)--(1,2)--(1,3)--(4,3)--(4,4)--(5,4)
	--(5,5)--(12,5)--(12,7)--(14,7)--(14,8)--(23,8);
	
	\node at (-0.5, 1.5) {\begin{turn}{90}\tiny$\alpha_4$\end{turn}};
	\node at (-0.5, 2.5) {\begin{turn}{90}\tiny$\alpha_5$\end{turn}};	
	\node at (-0.5, 3.5) {\begin{turn}{90}\tiny$\alpha_6$\end{turn}};	
	\node at (-0.5, 4.5) {\begin{turn}{90}\tiny$\alpha_7$\end{turn}};	
	\node at (-0.5, 5.5) {\begin{turn}{90}\tiny$\alpha_8$\end{turn}};	
	\node at (-0.5, 6.5) {\begin{turn}{90}\tiny$\alpha_9$\end{turn}};
	\node at (-0.5, 7.5) {\begin{turn}{90}\tiny$\alpha_{10}$\end{turn}};
	\node at (-0.5, 8.5) {\begin{turn}{90}\tiny$\alpha_{11}$\end{turn}};
	
	\node[color=Cerulean] at (-1,0) {\tiny$(0,0)$};
	\node[color=Cerulean] at (24.2,8) {\tiny$(23,8)$};	

\end{scope}

\end{tikzpicture}
\end{center}
\caption{An in-degree gravity diagram of $\car{12}{3}$, rotated to embed in its corresponding rational $(8,23)$-Dyck path under the bijection $\Psi_{\mathrm{in}}$.}
\label{fig.3-car_bijection_in}
\end{figure}

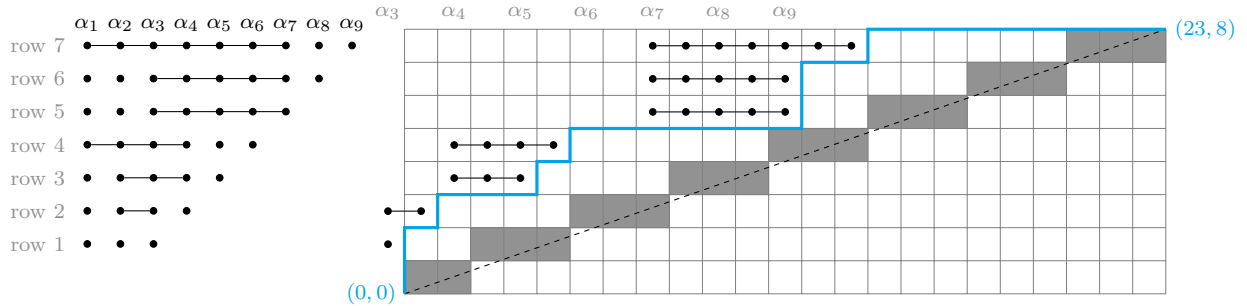
\begin{figure} [ht!]
\begin{center}
\begin{tikzpicture}[scale=0.5]

\begin{scope}
[xshift=0, scale=0.88]
	\vertex[fill, minimum size=2.5pt, label=above:\tiny$\alpha_1$](a10) at (0,0) {}; 	
	\vertex[fill, minimum size=2.5pt, label=above:\tiny$\alpha_2$](a20) at (1,0) {}; 
	\vertex[fill, minimum size=2.5pt, label=above:\tiny$\alpha_3$](a30) at (2,0) {};
	\vertex[fill, minimum size=2.5pt, label=above:\tiny$\alpha_4$](a40) at (3,0) {};
	\vertex[fill, minimum size=2.5pt, label=above:\tiny$\alpha_5$](a50) at (4,0) {};
	\vertex[fill, minimum size=2.5pt, label=above:\tiny$\alpha_6$](a60) at (5,0) {};
	\vertex[fill, minimum size=2.5pt, label=above:\tiny$\alpha_7$](a70) at (6,0) {};
	\vertex[fill, minimum size=2.5pt, label=above:\tiny$\alpha_8$](a80) at (7,0) {};
	\vertex[fill, minimum size=2.5pt, label=above:\tiny$\alpha_9$](a90) at (8,0) {};			
	\vertex[fill, minimum size=2.5pt](a11) at (0,-1) {}; 
	\vertex[fill, minimum size=2.5pt](a12) at (0,-2) {}; 
	\vertex[fill, minimum size=2.5pt](a13) at (0,-3) {}; 
	\vertex[fill, minimum size=2.5pt](a14) at (0,-4) {}; 
	\vertex[fill, minimum size=2.5pt](a15) at (0,-5) {}; 
	\vertex[fill, minimum size=2.5pt](a16) at (0,-6) {}; 
	\vertex[fill, minimum size=2.5pt](a21) at (1,-1) {}; 
	\vertex[fill, minimum size=2.5pt](a22) at (1,-2) {}; 
	\vertex[fill, minimum size=2.5pt](a23) at (1,-3) {};	
	\vertex[fill, minimum size=2.5pt](a24) at (1,-4) {}; 
	\vertex[fill, minimum size=2.5pt](a25) at (1,-5) {}; 
	\vertex[fill, minimum size=2.5pt](a26) at (1,-6) {};
	\vertex[fill, minimum size=2.5pt](a31) at (2,-1) {}; 
	\vertex[fill, minimum size=2.5pt](a32) at (2,-2) {}; 
	\vertex[fill, minimum size=2.5pt](a33) at (2,-3) {};	
	\vertex[fill, minimum size=2.5pt](a34) at (2,-4) {}; 
	\vertex[fill, minimum size=2.5pt](a35) at (2,-5) {}; 
	\vertex[fill, minimum size=2.5pt](a36) at (2,-6) {};
	\vertex[fill, minimum size=2.5pt](a41) at (3,-1) {}; 
	\vertex[fill, minimum size=2.5pt](a42) at (3,-2) {}; 
	\vertex[fill, minimum size=2.5pt](a43) at (3,-3) {};
	\vertex[fill, minimum size=2.5pt](a44) at (3,-4) {}; 
	\vertex[fill, minimum size=2.5pt](a45) at (3,-5) {};
	\vertex[fill, minimum size=2.5pt](a51) at (4,-1) {}; 
	\vertex[fill, minimum size=2.5pt](a52) at (4,-2) {}; 
	\vertex[fill, minimum size=2.5pt](a53) at (4,-3) {};
	\vertex[fill, minimum size=2.5pt](a54) at (4,-4) {};
	\vertex[fill, minimum size=2.5pt](a61) at (5,-1) {}; 
	\vertex[fill, minimum size=2.5pt](a62) at (5,-2) {}; 
	\vertex[fill, minimum size=2.5pt](a63) at (5,-3) {};
	\vertex[fill, minimum size=2.5pt](a71) at (6,-1) {}; 
	\vertex[fill, minimum size=2.5pt](a72) at (6,-2) {};
	\vertex[fill, minimum size=2.5pt](a81) at (7,-1) {};	
	\draw (a10) to (a70);
	\draw (a31) to (a71);
	\draw (a32) to (a72);
	\draw (a13) to (a43);
	\draw (a24) to (a44);
	\draw (a25) to (a35);
	\node[color=gray!85] at (-1.5,0) {\hbox{\tiny\textrm{row} $7$}};
	\node[color=gray!85] at (-1.5,-1) {\hbox{\tiny\textrm{row} $6$}};	
	\node[color=gray!85] at (-1.5,-2) {\hbox{\tiny\textrm{row} $5$}};
	\node[color=gray!85] at (-1.5,-3) {\hbox{\tiny\textrm{row} $4$}};
	\node[color=gray!85] at (-1.5,-4) {\hbox{\tiny\textrm{row} $3$}};
	\node[color=gray!85] at (-1.5,-5) {\hbox{\tiny\textrm{row} $2$}};	
	\node[color=gray!85] at (-1.5,-6) {\hbox{\tiny\textrm{row} $1$}};					\end{scope}

\begin{scope}
[xshift=240, yshift=-188, scale=0.88]
	\draw[fill, color=gray!85] (0,0) rectangle (2,1);
	\draw[fill, color=gray!85] (2,1) rectangle (5,2);
	\draw[fill, color=gray!85] (5,2) rectangle (8,3);
	\draw[fill, color=gray!85] (8,3) rectangle (11,4);
	\draw[fill, color=gray!85] (11,4) rectangle (14,5);
	\draw[fill, color=gray!85] (14,5) rectangle (17,6);
	\draw[fill, color=gray!85] (17,6) rectangle (20,7);
	\draw[fill, color=gray!85] (20,7) rectangle (23,8);
	\draw[very thin, color=gray!100] (0,0) grid (23,8);
	
	\vertex(v01)[fill, circle, inner sep=0, minimum size=2.5pt] at (-.5,1.5) {}; 	
	\vertex(v02)[fill, circle, inner sep=0, minimum size=2.5pt] at (-.5,2.5) {}; 	
	\vertex(v12)[fill, circle, inner sep=0, minimum size=2.5pt] at (.5,2.5) {};
	\vertex(v23)[fill, circle, inner sep=0, minimum size=2.5pt] at (1.5,3.5) {}; 	
	\vertex(v33)[fill, circle, inner sep=0, minimum size=2.5pt] at (2.5,3.5) {};
	\vertex(v43)[fill, circle, inner sep=0, minimum size=2.5pt] at (3.5,3.5) {};
	\vertex(v24)[fill, circle, inner sep=0, minimum size=2.5pt] at (1.5,4.5) {}; 	
	\vertex(v34)[fill, circle, inner sep=0, minimum size=2.5pt] at (2.5,4.5) {};
	\vertex(v44)[fill, circle, inner sep=0, minimum size=2.5pt] at (3.5,4.5) {};
	\vertex(v54)[fill, circle, inner sep=0, minimum size=2.5pt] at (4.5,4.5) {};
		
	\vertex(v85)[fill, circle, inner sep=0, minimum size=2.5pt] at (7.5,5.5) {}; 	
	\vertex(v95)[fill, circle, inner sep=0, minimum size=2.5pt] at (8.5,5.5) {};
	\vertex(v105)[fill, circle, inner sep=0, minimum size=2.5pt] at (9.5,5.5) {};
	\vertex(v115)[fill, circle, inner sep=0, minimum size=2.5pt] at (10.5,5.5) {};
	\vertex(v125)[fill, circle, inner sep=0, minimum size=2.5pt] at (11.5,5.5) {};	
	\vertex(v86)[fill, circle, inner sep=0, minimum size=2.5pt] at (7.5,6.5) {}; 	
	\vertex(v96)[fill, circle, inner sep=0, minimum size=2.5pt] at (8.5,6.5) {};
	\vertex(v106)[fill, circle, inner sep=0, minimum size=2.5pt] at (9.5,6.5) {};
	\vertex(v116)[fill, circle, inner sep=0, minimum size=2.5pt] at (10.5,6.5) {};
	\vertex(v126)[fill, circle, inner sep=0, minimum size=2.5pt] at (11.5,6.5) {};	
	\vertex(v87)[fill, circle, inner sep=0, minimum size=2.5pt] at (7.5,7.5) {}; 	
	\vertex(v97)[fill, circle, inner sep=0, minimum size=2.5pt] at (8.5,7.5) {};
	\vertex(v107)[fill, circle, inner sep=0, minimum size=2.5pt] at (9.5,7.5) {};
	\vertex(v117)[fill, circle, inner sep=0, minimum size=2.5pt] at (10.5,7.5) {};
	\vertex(v127)[fill, circle, inner sep=0, minimum size=2.5pt] at (11.5,7.5) {};	
	\vertex(v137)[fill, circle, inner sep=0, minimum size=2.5pt] at (12.5,7.5) {};
	\vertex(v147)[fill, circle, inner sep=0, minimum size=2.5pt] at (13.5,7.5) {};
	
	\draw(v02) to (v12);
	\draw(v23) to (v43);
	\draw(v24) to (v54);
	\draw(v85) to (v125);
	\draw(v86) to (v126);
	\draw(v87) to (v147);		
	
	\draw[dash pattern={on 2pt off 2pt}] (0,0) -- (23,8);
				
	\draw[very thick, color=Cerulean] (0,0)--(0,2)--(1,2)--(1,3)--(4,3)--(4,4)--(5,4)
	--(5,5)--(12,5)--(12,7)--(14,7)--(14,8)--(23,8);
	
	\node[color=Cerulean] at (-1,0) {\tiny$(0,0)$};
	\node[color=Cerulean] at (24.2,8) {\tiny$(23,8)$};	
	\node[color=gray!85] at (-0.5, 8.5) {\tiny$\alpha_3$};
	\node[color=gray!85] at (1.5, 8.5) {\tiny$\alpha_4$};
	\node[color=gray!85] at (3.5, 8.5) {\tiny$\alpha_5$};
	\node[color=gray!85] at (5.5, 8.5) {\tiny$\alpha_6$};		
	\node[color=gray!85] at (7.5, 8.5) {\tiny$\alpha_7$};
	\node[color=gray!85] at (9.5, 8.5) {\tiny$\alpha_8$};	
	\node[color=gray!85] at (11.5, 8.5) {\tiny$\alpha_9$};	

\end{scope}
\end{tikzpicture}
\end{center}
\caption{An out-degree gravity diagram of $\car{12}{3}$ and its corresponding rational $(8,23)$-Dyck path under the bijection $\Psi_{\mathrm{out}}$. 
}
\label{fig.3-car_bijection_out_d}
\end{figure}

\begin{prop}\label{prop.outgrav}
For $k\in \bbN$ and $n>k$, the map
$\Psi_{\mathrm{out}}:\outgrav_G(\bv_{\mathrm{out}})\rightarrow \calD(a,b)$ from the set of out-degree gravity diagrams of $G=\car{n+1}{k}$ to the set of rational $(a,b)$-Dyck paths is a bijection.
\end{prop}
\begin{proof}
Let $a=n-k$ and $b=ka-1$.
To see that $\Psi_{\mathrm{out}}$ is a bijection, we will describe the inverse map by reconstructing the line segments $L_1,\ldots, L_{a-1}$ for an out-degree gravity diagram $\Gamma \in \outgrav_G(\bv_{\mathrm{out}})$.
Let $\bs$ be a rational $(a,b)$-Dyck path, and let $L_i=[\ell_i,r_i]$ denote the line segment that we will reconstruct from the $(i+1)$-th row of the Dyck path.
The shape of $\bs$ immediately dictates the location of the right endpoint of each embedded line segment, so we need only to determine the location of the left endpoint. 

Suppose $jk \leq \rp(L_i) \leq (j+1)k-1$ for some $j=0,\ldots, i-2$.
By construction, $\lp(L_i)= (r_i-k)(k-1)$. We claim that $r_i-k=j$.  From there, we would have the length $d(L_i)=\rp(L_i)-\lp(L_i)$, and then we can fully determine the line segment $L_i=[k+j-d(L_i),k+j]$.

As seen in Lemma~\ref{lem.lineproperties}(b), $r_i-k \leq d(L_i) \leq r_i-1$. And since $d(L_i) = \rp(L_i)-\lp(L_i)$, it follows that
$$(r_i-k)k \leq \rp(L_i) \leq (r_i-k+1)k -1.$$
Since we assumed that $\rp(L_i) \leq (j+1)k-1$, then the inequality on the right side implies $r_i-k\leq j$.  
On the other hand, we also assumed that $jk\leq\rp(L_i)$, so the inequality of the left side implies $j\leq r_i-k$. Thus we have $r_i-k=j$, as claimed.

Since we can uniquely recover the embedded line segments and therefore the out-degree gravity diagram from any $(a,b)$-Dyck path $\bs$, then $\Psi_{\mathrm{out}}$ is a bijection.
\end{proof}

\begin{corollary} \label{cor.inout}
The composition $\Psi_{\mathrm{in}}^{-1}\circ\Psi_{\mathrm{out}}$ is a bijection between the sets of out-degree and in-degree gravity diagrams of $\car{n+1}{k}$.
\end{corollary}
\begin{proof}
An out-degree and an in-degree gravity diagram of $\car{n+1}{k}$ correspond to each other if they have the same associated rational $(a,b)$-Dyck path. 
\end{proof}

\begin{example}\label{eg.pairofDycks}
Figures~\ref{fig.3-car_bijection_in} and~\ref{fig.3-car_bijection_out_d} show a pair of in-degree and out-degree gravity diagrams for $G=\car{12}{3}$ which correspond to each other under the bijection $\Psi_{\mathrm{in}}^{-1}\circ\Psi_{\mathrm{out}}$ because they have the same associated rational $(8,23)$-Dyck path. 

In Figure~\ref{fig.3-car_bijection_out_d}, we have an out-degree gravity diagram $\Gamma \in \outgrav_{G}(\bv_{\mathrm{out}})$.
The $(2j)$-th column of squares of the Dyck path diagram are indexed by $\alpha_{3+j}$ for $j=0,\ldots, 6$. These are indicated in light grey across the top row of the diagram.
The line segments of $\Gamma$ are $L_1,\ldots, L_7=[3,3], [2,3], [2,4], [1,4], [3,7], [3,7], [1,7]$, and $L_i=[\ell_i,r_i]$ is embedded into the $(i+1)$-th row of squares of the associated rational $(8,23)$-Dyck path with their left endpoints occupying the column labeled by $\alpha_{r_i}$ for each $i$. The `rectilinear convex hull' of the embedded line segments forms the subpartition $(14,12,12,5,4,1)$ of the partition $\lambda(8,3)=(20,17,14,11,8,5,2)$.
\end{example}

\begin{remark}
We make a few comments regarding the special case of the classical caracol graph $G=\car{n+1}{1}$. In this case, the out-degree gravity diagrams are defined on a triangular array of $n-j-1$ dots in the $j$-th column for $j=1,\ldots, n-2$, and each horizontal line segment extends from the first column to the $j$-th column for some $j=1,\ldots, n-2$.  Similarly, the in-degree gravity diagrams are defined on a triangular array of $j-2$ dots in the $j$-th column for $j=3,\ldots, n$, and each horizontal line segment extends from the last column to the $j$-th column for some $j=3,\ldots, n$.  See Figure~\ref{fig.car61} for the full sets of out-degree and in-degree gravity diagrams for $\car61$.

Given this, an `obvious' bijection between the out-degree and in-degree gravity diagrams for $\car{n+1}{1}$ is the reflection about a vertical axis.  Corollary~\ref{cor.inout} gives a second bijection; the out-degree diagrams can equivalently be thought of as subpartitions of the staircase partition $\delta_{n-2}=(n-3,n-4,\ldots, 2,1)$, with the parts of the subpartition defined by the lengths of the horizontal line segments.  The bijection $\Psi_{\mathrm{in}}^{-1}\circ\Psi_{\mathrm{out}}$ amounts to being the conjugation of partitions.  See Figure~\ref{fig.car81} for an example for $\car81$.
\end{remark}

\begin{figure}[ht!]
\begin{tikzpicture}[scale=0.5]
\begin{scope}
	\draw[fill, color=gray!85] (0,0) rectangle (1,1);
	\draw[fill, color=gray!85] (1,1) rectangle (2,2);
	\draw[fill, color=gray!85] (2,2) rectangle (3,3);
	\draw[fill, color=gray!85] (3,3) rectangle (4,4);
	\draw[fill, color=gray!85] (4,4) rectangle (5,5);
	\draw[very thin, color=gray!100] (0,0) grid (5,5);
	
	\vertex[fill, minimum size=3pt]  at (-0.5, 4.5) {}; 
	\vertex[fill, minimum size=3pt]  at (0.5, 4.5) {}; 
	\vertex[fill, minimum size=3pt]  at (1.5, 4.5) {};
	\vertex[fill, minimum size=3pt]  at (2.5, 4.5) {}; 
	\vertex[fill, minimum size=3pt]  at (3.5, 4.5) {};	
	\vertex[fill, minimum size=3pt]  at (-0.5, 3.5) {}; 	
	\vertex[fill, minimum size=3pt]  at (0.5, 3.5) {}; 
	\vertex[fill, minimum size=3pt]  at (1.5, 3.5) {};
	\vertex[fill, minimum size=3pt]  at (2.5, 3.5) {};
	\vertex[fill, minimum size=3pt]  at (-0.5, 2.5) {}; 	
	\vertex[fill, minimum size=3pt]  at (0.5, 2.5) {}; 
	\vertex[fill, minimum size=3pt]  at (1.5, 2.5) {};
	\vertex[fill, minimum size=3pt]  at (-0.5, 1.5) {}; 	
	\vertex[fill, minimum size=3pt]  at (0.5, 1.5) {};		
	\vertex[fill, minimum size=3pt]  at (-0.5, 0.5) {};
	\node at (-0.5, 5.5) {\tiny$\alpha_1$};
	\node at (0.5, 5.5) {\tiny$\alpha_2$};
	\node at (1.5, 5.5) {\tiny$\alpha_3$};
	\node at (2.5, 5.5) {\tiny$\alpha_4$};		 	
	\node at (3.5, 5.5) {\tiny$\alpha_5$};
	\draw (-0.5,4.5)--(2.5,4.5);
	\draw (-0.5,3.5)--(1.5,3.5);
	\draw (-0.5,2.5)--(1.5,2.5);
	\draw (-0.5,1.5)--(0.5,1.5);
	\node at (2.5,-1) {$\lambda=(3,2,2,1)$};
	
	\draw[very thick, color=Cerulean] (0,0)--(0,1)--(1,1)--(1,2)--(2,2)--(2,4)--(3,4)--(3,5)--(5,5);	
\end{scope}

\begin{scope}[xshift=250, yshift=65]
	\node[label=above:$\Psi_{\mathrm{in}}^{-1}\circ\Psi_{\mathrm{out}}$] at (0,0) {$\longrightarrow$};
\end{scope}

\begin{scope}[xshift=380]
	\draw[fill, color=gray!85] (0,0) rectangle (1,1);
	\draw[fill, color=gray!85] (1,1) rectangle (2,2);
	\draw[fill, color=gray!85] (2,2) rectangle (3,3);
	\draw[fill, color=gray!85] (3,3) rectangle (4,4);
	\draw[fill, color=gray!85] (4,4) rectangle (5,5);
	\draw[very thin, color=gray!100] (0,0) grid (5,5);
	
	\vertex[fill, minimum size=3pt] at (0.5, 5.5) {}; 
	\vertex[fill, minimum size=3pt] at (1.5, 5.5) {};
	\vertex[fill, minimum size=3pt] at (2.5, 5.5) {}; 
	\vertex[fill, minimum size=3pt] at (3.5, 5.5) {}; 
	\vertex[fill, minimum size=3pt] at (4.5, 5.5) {};	
	\vertex[fill, minimum size=3pt] at (0.5, 4.5) {}; 
	\vertex[fill, minimum size=3pt] at (1.5, 4.5) {};
	\vertex[fill, minimum size=3pt] at (2.5, 4.5) {}; 
	\vertex[fill, minimum size=3pt] at (3.5, 4.5) {};	
	\vertex[fill, minimum size=3pt] at (0.5, 3.5) {}; 
	\vertex[fill, minimum size=3pt] at (1.5, 3.5) {};
	\vertex[fill, minimum size=3pt] at (2.5, 3.5) {};
	\vertex[fill, minimum size=3pt] at (0.5, 2.5) {}; 
	\vertex[fill, minimum size=3pt] at (1.5, 2.5) {};
	\vertex[fill, minimum size=3pt] at (0.5, 1.5) {};	
	\draw (0.5,1.5)--(0.5,5.5);
	\draw (1.5,2.5)--(1.5,5.5);
	\draw (2.5,4.5)--(2.5,5.5);	
	\node at (-0.5, 1.5) {\begin{turn}{90}\tiny$\alpha_3$\end{turn}};	
	\node at (-0.5, 2.5) {\begin{turn}{90}\tiny$\alpha_4$\end{turn}};	
	\node at (-0.5, 3.5) {\begin{turn}{90}\tiny$\alpha_5$\end{turn}};	
	\node at (-0.5, 4.5) {\begin{turn}{90}\tiny$\alpha_6$\end{turn}};	
	\node at (-0.5, 5.5) {\begin{turn}{90}\tiny$\alpha_7$\end{turn}};	
	\node at (2.5,-1) {$\lambda'=(4,3,1)$};
						
	\draw[very thick, color=Cerulean] (0,0)--(0,1)--(1,1)--(1,2)--(2,2)--(2,4)--(3,4)--(3,5)--(5,5);
\end{scope}
\end{tikzpicture}
\caption{The bijection $\Psi_{\mathrm{in}}^{-1}\circ\Psi_{\mathrm{out}}$ for the out-degree and in-degree gravity diagrams for $\car{n+1}1$ amounts to the conjugation of partitions.}
\label{fig.car81}
\end{figure}
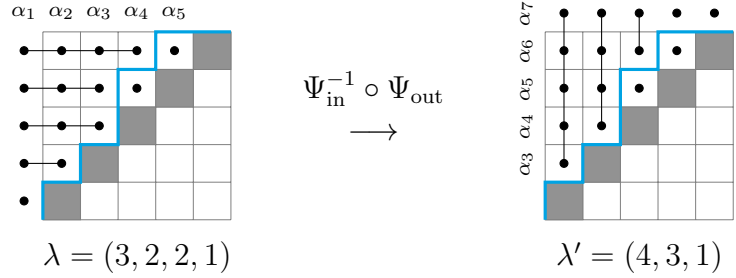

\begin{remark} \label{rem.duality}
To summarize, we have seen that the volume of the flow polytope of the $k$-caracol graph $G=\car{n+1}{k}$ with unit flow can be computed by counting the number lattice points of two different graphs. When $k\geq2$,
\begin{align*}
\big|\calF_{G_{\mathrm{in}}} 
&\left(k-1, k^{n-k-1}, -k(n-k)+1\right) \cap \bbZ^{\dim G_{\mathrm{in}} } \big|\\
&=K_{G_{\mathrm{in}}} \left(k-1, k^{n-k-1}, -k(n-k)+1\right)\\
&=\vol \calF_G(1,0,\ldots, 0,-1)\\
&=K_{G_{\mathrm{out}}} \left(k(n-k)-2, -(n-k)^{k-2}, -(n-k-1), (-1)^{n-k-1}\right)\\
&=\big|\calF_{G_{\mathrm{out}}}\left(k(n-k)-2, -(n-k)^{k-2}, -(n-k-1), (-1)^{n-k-1}\right) \cap \bbZ^{\dim G_{\mathrm{out}} } \big|,
\end{align*}
where $G_{\mathrm{in}}$ is the restriction of $G=\car{n+1}{k}$ to the vertices $\{k,\ldots,n+1\}$ and $G_{\mathrm{out}}$ is the restriction of $G=\car{n+1}{k}$ to the vertices $\{1,\ldots,n-1\}$.
We point out that at $k=2$, $G_{\mathrm{in}}=\PS_{n-1}$, and for $k\geq3$, $G_{\mathrm{in}} = \Car_{n-k+2}$. 
The case $k=1$ is trivial since $G_{\mathrm{in}} = \PS_{n-1} = G_{\mathrm{out}}^{\mathrm{rev}}$, and $\calF_{G_{\mathrm{in}}}(1^{n-2},-(n-2))$ is equal to $\calF_{G_{\mathrm{out}}}(n-2, (-1)^{n-2})$ by reversing the flow.
For $k\geq2$, it may be interesting to investigate any geometric implications behind the combinatorial correspondence given by $\Psi_{\mathrm{in}}^{-1} \circ \Psi_{\mathrm{out}}$ on the lattice points of these flow polytopes of different dimensions.
\end{remark}

\section{Volume of the $k$-caracol polytope with net flow $(1,\ldots,1,-n)$} \label{sec.kcaracol}
In~\cite{BGHHKMY}, we introduced a combinatorial interpretation of the Lidskii volume formula (Theorem~\ref{thm.lidskii}) and called the objects unified diagrams. In this section, we define unified diagrams for the $k$-caracol graph, and compute the volume of the flow polytope of $\car{n+1}{k}$ with net flow $\ba=(1,\ldots, 1,-n)$. As a corollary, we recover the analogous result for the classical caracol graph and the Pitman--Stanley graph.

We point out that the results in this section is the first application of using unified diagrams to compute volumes of flow polytopes whose underlying graphs are not planar.

\subsection{Unified diagrams}
In this section, we restrict ourselves to defining unified diagrams for flow polytopes with net flow $\ba=(1,\ldots, 1,-n)$. We will discuss unified diagrams in full generality in Section~\ref{sec.abbb}.

\begin{defn} Let $\bt = (t_1,\ldots, t_p) \vDash q$.
A {\em labeled $\bt$-Dyck path} is a pair $(\bs,\sigma)$ where $\bs=(s_1,\ldots, s_p)$ is a $\bt$-Dyck path and $\sigma$ is a permutation in the symmetric group $\fS_q$, whose descent set is contained in $\{s_1+\cdots+s_j\mid j=1,\ldots, p-1\}$.  Let $\PF_\bt$ denote the set of labeled $\bt$-Dyck paths, which are also known as {\em generalized parking functions}.
\end{defn}
\begin{defn}\label{defn.unified}
Let $G$ be an acyclic directed graph with $n+1$ vertices and shifted out-degree vector $\bt$. A {\em unified diagram} for the flow polytope $\calF_G(1,\ldots,1,-n)$ is a triple $(\bs, \sigma, \Gamma)$ where $(\bs,\sigma)$ is a labeled $\bt$-Dyck path and $\Gamma$ is an out-degree gravity diagram for $\outgrav_{G}(\bs-\bt,0)$. Let $\calU_G$ denote this set of unified diagrams.
\end{defn}

Visually, if $\bt=(t_1,\ldots, t_p)\vDash q$, then $(\bs,\sigma)$ is the lattice path $N^{s_1}E \cdots N^{s_p}E$ on the rectangular grid form $(0,0)$ to $(p, q)$ which lies above the shaded $\bt$-region, and whose north steps are labeled by the permutation $\sigma$ so that the labels on consecutive north steps are nondecreasing.
See Figure~\ref{fig.UDcar73} for an example where $G=\car73$. There, the shifted out-degree vector is $\bt=(3,3,2,1,1,0)$, indicated by the shaded squares, and the $\sigma$-labeled $\bt$-Dyck path $\bs=(5,4,0,1,0,0)$ is indicated in red.  The gravity diagram $\Gamma$, which represents a vector partition of $\bs-\bt=2\alpha_1+3\alpha_2+\alpha_3+\alpha_4$ with respect to graph $\car73$, is embedded in the squares bounded between the $\bt$-Dyck path and the shaded $\bt$-region.

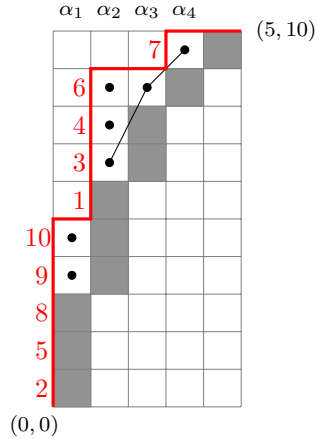
\begin{figure}[ht!]
\begin{tikzpicture}[scale=0.5]
\begin{scope}[xshift=0]

	\draw[fill, color=gray!85] (0,0) rectangle (1,3);
	\draw[fill, color=gray!85] (1,3) rectangle (2,6);
	\draw[fill, color=gray!85] (2,6) rectangle (3,8);
	\draw[fill, color=gray!85] (3,8) rectangle (4,9);
	\draw[fill, color=gray!85] (4,9) rectangle (5,10);
	\draw[very thin, color=gray!100] (0,0) grid (5,10);	

	\vertex[fill, minimum size=3pt] at (.5,3.5){};
	\vertex[fill, minimum size=3pt] at (.5,4.5){};
	\vertex[fill, minimum size=3pt] at (1.5,6.5){};
	\vertex[fill, minimum size=3pt] at (1.5,7.5){};
	\vertex[fill, minimum size=3pt] at (1.5,8.5){};				
	\vertex[fill, minimum size=3pt] at (2.5,8.5){};
	\vertex[fill, minimum size=3pt] at (3.5,9.5){};
	\draw (1.5,6.5)--(2.5,8.5)--(3.5,9.5);
								
	\draw[very thick, color=red] (0,0)--(0,5)--(1,5)--(1,9)--(3,9)--(3,10)--(5,10);
	\node at (-.3,.5) {\footnotesize\textcolor{red}{$2$}};
	\node at (-.3,1.5) {\footnotesize\textcolor{red}{$5$}};
	\node at (-.3,2.5) {\footnotesize\textcolor{red}{$8$}};
	\node at (-.3,3.5) {\footnotesize\textcolor{red}{$9$}};
	\node at (-.4,4.5) {\footnotesize\textcolor{red}{$10$}};
	\node at (.7,5.5) {\footnotesize\textcolor{red}{$1$}};
	\node at (.7,6.5) {\footnotesize\textcolor{red}{$3$}};
	\node at (.7,7.5) {\footnotesize\textcolor{red}{$4$}};
	\node at (.7,8.5) {\footnotesize\textcolor{red}{$6$}};
	\node at (2.7,9.5) {\footnotesize\textcolor{red}{$7$}};
	
	\node at (0.5, 10.5) {\tiny$\alpha_1$};
	\node at (1.5, 10.5) {\tiny$\alpha_2$};
	\node at (2.5, 10.5) {\tiny$\alpha_3$};
	\node at (3.5, 10.5) {\tiny$\alpha_4$};
	\node at (-.5,-.5) {\tiny$(0,0)$};
	\node at (6.2,10) {\tiny$(5,10)$};	
\end{scope}
\end{tikzpicture}
\caption{A unified diagram $U=(\bs,\sigma,\Gamma)$ for $\car73$.}
\label{fig.UDcar73}
\end{figure}

\begin{remark} Since $\sum_{j=1}^n t_j= m-n$, then $(m-n)\be_1=(m-n,0,\ldots,0) \rhd \bt$, and $\bv_{\mathrm{out}}=(m-n)\be_1-\bt$.  All other $\bs\vDash m-n$ which dominate $\bt$ satisfy $(m-n)\be_1 \rhd \bs \rhd \bt$, and $\outgrav_G(\bs-\bt,0) \subseteq \outgrav_G(\bv_{\mathrm{out}})$ for all $\bs\rhd\bt$.
\end{remark}

Unified diagrams were created for the purpose of combinatorializing the generalized Lidskii volume formula. We restate the formula in a way that is convenient for us to use later on.  This next result follows from the fact that the number of labeled $\bt$-Dyck paths is $|\PF_\bt| = \sum_{\bs\rhd\bt}{|\bt|\choose \bs}$.
\begin{theorem}[Parking function version of the Lidskii volume formula, {\cite[Theorems 4.3, 4.4]{BGHHKMY}}] 
\label{thm.parkinglidskii}
With the same conditions as in Theorem~\ref{thm.lidskii}, the volume of the flow polytope $\calF_G(\ba)$ of $G$ with net flow vector $\ba$ is
$$\vol\calF_G(\ba) = \sum_{(\bs,\sigma)\in \PF_\bt} \ba^\bs \cdot K_G(\bs-\bt,0) = |\,\calU_G(\ba)|. $$
\end{theorem}


\subsection{Refinements of unified diagrams}
In this section, we set up the combinatorial tools necessary for enumerating the unified diagrams for the flow polytope of the $k$-caracol graphs. This will be achieved by stratifying the set of unified diagrams according to level.  

\begin{defn} Let $\bt=(t_1,\ldots, t_p)\vDash q$. Given a $\bt$-Dyck path $\bs=(s_1,\ldots, s_p)$, its {\em $k$-th column level} is defined to be $q-(s_1+\cdots+s_k)$, for $k=1,\ldots, p$. Visually, this is the height at which the $k$-th east step of $\bs$ occurs, where the zero-th level starts from the top of the Dyck path at $y=q$.  The possible levels in the $k$-th column are $i=0,\ldots, q-(t_1+\cdots+t_k)$. 
\end{defn}

\begin{defn}
Let $G$ be a directed graph with $n+1$ vertices and shifted out-degree vector $\bt \vDash m-n$. 
If $(\bs,\sigma)$ is a labeled $\bt$-Dyck path whose $k$-th column level is $i$, we can decompose it into two labeled Dyck paths $(\bp,\pi)$ and $(\bq,\kappa)$ respectively corresponding to the subpaths before and after the $k$-th east step of $\bs$. 

We can standardize the labelings so that $\kappa \in \fS_i$ and $\pi\in \fS_{m-n-i}$. There are ${m-n\choose i}$ ways to choose a label set of size $i$, so  
\begin{equation}\label{eqn.stdU}
|\,\calU_G| = \sum_{i=0}^{m-n-(t_1+\cdots+t_k)} {m-n\choose i} 
|\,\mathcal{SU}_G^{(k,i)}|
\end{equation}
where $\mathcal{SU}_G^{(k,i)} =  \{ ((\bp,\pi), (\bq,\kappa), \Gamma) \}$ is the set of {\em standardized level-$(k,i)$ unified diagrams} for $\calF_G$; the concatenation of $\bp$ and $\bq$ is a $\bt$-Dyck path $\bs$ with $s_1+\cdots+s_k = m-n-i$, the labels $\kappa\in \fS_i$ and $\pi\in \fS_{m-n-i}$, and $\Gamma\in \outgrav_G(\bs-\bt,0)$ is an out-degree gravity diagram with $m-n-(t_1+\cdots+t_k)-i$ dots in the $k$-th column.
\end{defn}

\begin{example}
The labeled $\bt$-Dyck path $(\bs,\sigma)$ in the unified diagram $U$ for $\car73$ from Figure~\ref{fig.UDcar73} has level $i=1$ in the third column, and it decomposes into the two labeled Dyck paths $(\bp,\pi)$ and $(\bq,\kappa)$ with standardized labelings where $\bp = (5,4,0)$, $\sigma=257891346\in \fS_9$, $\bq=(1,0,0)$, and $\kappa = 1\in \fS_1$.
\end{example}

We need one further refinement on the set of unified diagrams.
\begin{defn}  
Let $G$ be a directed graph with $n+1$ vertices and shifted out-degree vector $\bt \vDash m-n$. For $k\in \bbN$ and $i\in \bbZ_{\geq0}$, a {\em truncated level-$(k,i)$ unified diagram} for the flow polytope $\calF_G$ is obtained by taking a standardized level-$(k,i)$ unified diagram $(\bs,\sigma, \Gamma)$ for $\calF_G$ and erasing the {\em initial part} $(\bp,\pi)$ of the labeled $\bt$-Dyck path $(\bs,\sigma)$ which occurs before (and including) the $k$-th east step of $\bs$. 

In other words, this is a triple $(\bq,\kappa,\Gamma)$ where $(\bq, \kappa)$ is a labeled $\bt'=(t_{k+1},\ldots, t_n)$-Dyck path that begins at the coordinates $(k, m-n-i)$ 
and is labeled by $\kappa\in \fS_i$ so that the labels on consecutive north steps of 
$\bq$ are non-decreasing, and $\Gamma$ is an out-degree gravity diagram for $G$ with $m-n-(t_1+\cdots+t_k)-i$ dots in its $k$-th column.
Let $\calU_G^{(k,i)}$ denote the set of truncated level-$(k,i)$ unified diagrams for $G$. 
\end{defn}

\begin{example}
The left side of Figure~\ref{fig.3-caracol_ud_103} shows a truncated level-$(3,2)$ unified diagram for $G = \car{10}3$.  Note that the only requirement on how the line segments of the embedded gravity diagram $\Gamma$ are depicted is that the line segments must occupy the lowest possible dots in each column.  That is, `gravity' drags the line segments downwards.
\end{example}

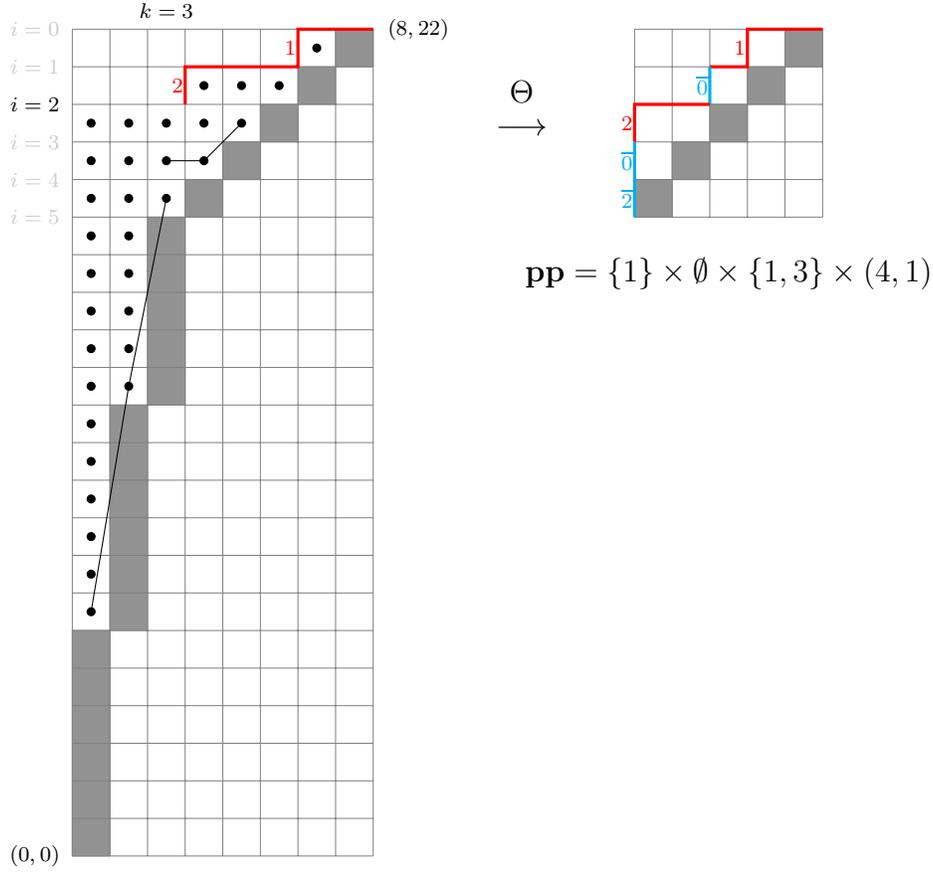
\begin{figure}[ht!]
\begin{tikzpicture}[scale=0.5]
\begin{scope}[xshift=0]

	\draw[fill, color=gray!85] (0,0) rectangle (1,6);
	\draw[fill, color=gray!85] (1,6) rectangle (2,12);
	\draw[fill, color=gray!85] (2,12) rectangle (3,17);
	\draw[fill, color=gray!85] (3,17) rectangle (4,18);
	\draw[fill, color=gray!85] (4,18) rectangle (5,19);
	\draw[fill, color=gray!85] (5,19) rectangle (6,20);
	\draw[fill, color=gray!85] (6,20) rectangle (7,21);
	\draw[fill, color=gray!85] (7,21) rectangle (8,22);	
	\draw[very thin, color=gray!100] (0,0) grid (8,22);	

	\foreach \y in {6.5,7.5,8.5,9.5,10.5,11.5,12.5,13.5,14.5,15.5,16.5,17.5,18.5,19.5} 
		\vertex[fill, minimum size=3pt] at (0.5,\y) {};
	\foreach \y in {12.5,13.5,14.5,15.5,16.5,17.5,18.5,19.5} 
		\vertex[fill, minimum size=3pt] at (1.5,\y) {};
	\foreach \y in {17.5,18.5,19.5} 
		\vertex[fill, minimum size=3pt] at (2.5,\y) {};
	\foreach \y in {18.5,19.5,20.5} 
		\vertex[fill, minimum size=3pt] at (3.5,\y) {};
	\foreach \y in {19.5,20.5}					
		\vertex[fill, minimum size=3pt] at (4.5,\y) {};
	\vertex[fill, minimum size=3pt] at (5.5,20.5) {};
	\vertex[fill, minimum size=3pt] at (6.5,21.5) {};	
	\draw (4.5,19.5)--(3.5,18.5)--(2.5,18.5);
	\draw (2.5,17.5)--(1.5,12.5)--(0.5,6.5);	
						
								
	\draw[very thick, color=red] (3,20)--(3,21)--(5,21)--(6,21)--(6,22)--(8,22);
	\node at (2.8,20.5) {\tiny\textcolor{red}{$2$}};
	\node at (5.8,21.5) {\tiny\textcolor{red}{$1$}};									
	
	\node at (-1,22) {\tiny\textcolor{black!20}{$i=0$}};
	\node at (-1,21) {\tiny\textcolor{black!20}{$i=1$}};	
	\node at (-1,20) {\tiny$i=2$};
	\node at (-1,19) {\tiny\textcolor{black!20}{$i=3$}};
	\node at (-1,18) {\tiny\textcolor{black!20}{$i=4$}};
	\node at (-1,17) {\tiny\textcolor{black!20}{$i=5$}};			
	\node at (2.5,22.5) {\tiny$k=3$};
	\node at (-1,0) {\tiny$(0,0)$};
	\node at (9.2,22) {\tiny$(8,22)$};		
\end{scope}

\begin{scope}[xshift=340, yshift=550]
	\node[label=above:$\Theta$] at (0,0) {$\longrightarrow$};
\end{scope}

\begin{scope}[xshift=340]

	\draw[fill, color=gray!85] (3,17) rectangle (4,18);
	\draw[fill, color=gray!85] (4,18) rectangle (5,19);
	\draw[fill, color=gray!85] (5,19) rectangle (6,20);
	\draw[fill, color=gray!85] (6,20) rectangle (7,21);
	\draw[fill, color=gray!85] (7,21) rectangle (8,22);	
	\draw[very thin, color=gray!100] (3,17) grid (8,22);	

								
	\draw[very thick, color=red] (5,21)--(6,21)--(6,22)--(8,22);	
	\draw[very thick, color=cyan] (5,20)--(5,21);
	\draw[very thick, color=red] (3,19)--(3,20)--(5,20);
	\draw[very thick, color=cyan] (3,17)--(3,19);			

	\node at (5.8,21.5) {\tiny\textcolor{red}{$1$}};
	\node at (4.8,20.5) {\tiny\textcolor{cyan}{$\overline{0}$}};		
	\node at (2.8,19.5) {\tiny\textcolor{red}{$2$}};
	\node at (2.8,18.5) {\tiny\textcolor{cyan}{$\overline{0}$}};
	\node at (2.8,17.5) {\tiny\textcolor{cyan}{$\overline{2}$}};
	
	\node at (5.5,15.5) {$\mathbf{pp} = \{1\} \times \emptyset \times\{1,3\} \times (4,1)$};

\end{scope}
\end{tikzpicture}
\caption{On the left is a truncated unified diagram $U=(\bq,\kappa,\Gamma) \in \calU_G^{(3,2)}$ for $G=\car{10}3$, and on the right is its corresponding $3$-multi-labeled Dyck path $M$ under the bijection $\Theta: \calU_G^{(3,2)} \rightarrow \calT_3(3,2)$. $M$ encodes the parking preferences $\mathbf{pp}$.}
\label{fig.3-caracol_ud_103}
\end{figure}

\begin{defn}
For each truncated unified diagram $U=(\bq,\kappa, \Gamma)\in \calU_G^{(k,i)}$, let $S(U)$ be the number of ways to complete $U$ to obtain a standardized unified diagram $((\bp,\pi),(\bq,\kappa), \Gamma) \in \mathcal{SU}_G^{(k,i)}$.
\end{defn}

To be clear, a completion $(\bp,\pi)$ is a labeled $(t_1,\ldots, t_k)$-Dyck path $\bp$ from $(0,0)$ to $(k, m-n-i)$ whose last step is the east step $(k-1, m-n-i)$ to $(k, m-n-i)$, the label $\pi\in \fS_{m-n-i}$, and $\Gamma$ is contained in the region between $\bp$ and the shaded $\bt$-region. 
We then have
\begin{equation}\label{eqn.SUcompletion}
\left|\, \mathcal{SU}_{G}^{(k,i)}\right| = \sum_{U\in \,\calU_{G}^{(k,i)}} S(U). 
\end{equation}

\subsection{Completions of truncated unified diagrams for the $k$-caracol graph}
In the remainder of this section, we let $G=\car{n+1}{k}$.


\begin{defn}
Let $U= (\bq,\kappa, \Gamma)\in \calU_{G}^{(k,i)}$ be a truncated unified diagram for $G$, with the gravity diagram drawn so that its line segments occupy the lowest possible dots in each column.  The {\em $k$-hull} of $U$ is the weak composition $\bc=(c_1,\ldots, c_k)\vDash m-n-i$ which represents the shape of the $(t_1,\ldots, t_k)$-Dyck path $\bp=N^{c_1}E \cdots N^{c_k}E$ from $(0,0)$ to $(k, m-n-i)$ having the smallest possible area.
\end{defn}

Recall from Remark~\ref{rem.trapezoid} that without loss of generality, we can consider out-degree gravity diagrams for $\car{n+1}{k}$ to be defined on a trapezoidal array of dots with $k-1+i$ dots in the $i$-th row for $i=1,\ldots, n-k-1$.  
In particular, the columns of the gravity diagram indexed by $\alpha_1,\ldots, \alpha_k$ form a $(n-k-1)\times k$ rectangle $R$.  Given $\Gamma \in \outgrav_G(\bv_{\mathrm{out}})$, let $\Gamma|_R$ denote the restriction of the gravity diagram to the dots in $R$. Note that every line segment of $\Gamma|_R$ has its right endpoint in the $k$-th column.

\begin{lemma}\label{lem.cvector}
Let $G=\car{n+1}{k}$, and let $U=(\bq,\kappa,\Gamma)\in \calU_{G}^{(k,i)}$ be a truncated unified diagram for $G$.  Let $L_1,\ldots, L_{n-k-1-i}$ be the (possibly trivial) line segments of $\Gamma|_R$, where $L_j=[\ell_j, k]$ for $1\leq \ell_j \leq k$ and $j=1,\ldots, n-k-1-i$.  Let $\bh=(n-k,\ldots, n-k, 2(n-k-1)-i)$.  The $k$-hull of $U$ is
$$\bc(U) = \bh + \sum_{j=1}^{n-k-1-i} (\be_{\ell_j} - \be_k). $$
\end{lemma}
\begin{proof}
Recalling from Section~\ref{sec.gravity} that the shifted out-degree vector for $\car{n+1}{k}$ is 
$$\bt=(t_1,\ldots, t_n) = (\underbrace{n-k,\ldots, n-k}_{k-1},n-k-1, \underbrace{1,\ldots,1}_{n-k-1},0)\vDash m-n,$$
then $\bh=(h_1,\ldots, h_k) =(n-k,\ldots, n-k, 2(n-k-1)-i)$ is a composition of $m-n-i$ with $k$ parts that represents the hull of a truncated level-$(k,i)$ unified diagram having an empty gravity diagram.
Now, with the gravity diagram $\Gamma$ embedded into $U$, then $\bc(U)$ is determined by $\bh$, together with the line segments of $\Gamma|_R$. For each line segment $L_j = [\ell_j,k]$ beginning in the $\ell_j$-th column and ending in the $k$-th column, the $k$-hull of the truncated unified diagram is obtained by altering $\bh$ by $\be_{\ell_j}-\be_k$.
\end{proof}

\begin{example}
For the truncated unified diagram $U$ in Figure~\ref{fig.3-caracol_ud_103}, its gravity diagram $\Gamma$ is
$$
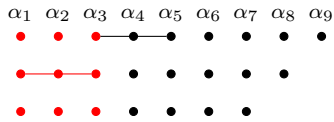
\begin{tikzpicture}[scale=0.5]
\begin{scope}[xshift=0, scale=1.0]

	\vertex[fill, minimum size=3pt, label=above:\tiny$\alpha_1$, color=red](a10) at (0,0) {}; 	
	\vertex[fill, minimum size=3pt, label=above:\tiny$\alpha_2$, color=red](a20) at (1,0) {}; 
	\vertex[fill, minimum size=3pt, label=above:\tiny$\alpha_3$, color=red](a30) at (2,0) {};
	\vertex[fill, minimum size=3pt, label=above:\tiny$\alpha_4$](a40) at (3,0) {};
	\vertex[fill, minimum size=3pt, label=above:\tiny$\alpha_5$](a50) at (4,0) {};
	\vertex[fill, minimum size=3pt, label=above:\tiny$\alpha_6$](a60) at (5,0) {};
	\vertex[fill, minimum size=3pt, label=above:\tiny$\alpha_7$](a70) at (6,0) {};
	\vertex[fill, minimum size=3pt, label=above:\tiny$\alpha_8$](a80) at (7,0) {};
	\vertex[fill, minimum size=3pt, label=above:\tiny$\alpha_9$](a90) at (8,0) {};		
	\vertex[fill, minimum size=3pt, color=red](a11) at (0,-1) {}; 
	\vertex[fill, minimum size=3pt, color=red](a12) at (0,-2) {}; 
	\vertex[fill, minimum size=3pt, color=red](a21) at (1,-1) {}; 
	\vertex[fill, minimum size=3pt, color=red](a22) at (1,-2) {}; 
	\vertex[fill, minimum size=3pt, color=red](a31) at (2,-1) {}; 
	\vertex[fill, minimum size=3pt, color=red](a32) at (2,-2) {}; 
	\vertex[fill, minimum size=3pt](a41) at (3,-1) {}; 
	\vertex[fill, minimum size=3pt](a42) at (3,-2) {}; 
	\vertex[fill, minimum size=3pt](a51) at (4,-1) {}; 
	\vertex[fill, minimum size=3pt](a52) at (4,-2) {}; 
	\vertex[fill, minimum size=3pt](a61) at (5,-1) {}; 
	\vertex[fill, minimum size=3pt](a62) at (5,-2) {}; 
	\vertex[fill, minimum size=3pt](a71) at (6,-1) {}; 
	\vertex[fill, minimum size=3pt](a72) at (6,-2) {};
	\vertex[fill, minimum size=3pt](a81) at (7,-1) {};
	\draw (a30) to (a50);
	\draw[color=red] (a11) to (a31);		
\end{scope}
\end{tikzpicture}
$$
with the restriction $\Gamma|_R$ depicted in red. We have $\bh= (6,6,8)$, and the $3$-hull of $U$ is
$$\bc(U) = (6,6,8) + (0,0,1-1) + (1,0,-1) + (0,0,1-1) = (7,6,7).$$
\end{example}


\begin{lemma}\label{lem.completions} Let $G=\car{n+1}{k}$, and let $\bc(U)=(c_1,\ldots, c_k)$ be the $k$-hull of the truncated unified diagram $U=(\bq,\kappa,\Gamma)\in \calU_{G}^{(k,i)}$. The number of ways to complete $U$ to a standardized unified diagram in $\mathcal{SU}_G^{(k,i)}$ is
$$S(U) = \sum_{\bd \in\calC(\bc(U))} {m-n-i\choose \bd},$$
where $\calC(\bc(U)) = \{\bd\vDash m-n-i \mid d_1 + \cdots +d_j \geq c_1+\cdots+c_j, \hbox{ for } j=1,\ldots, k-1 \}$.
\end{lemma}
\begin{proof} The $k$-hull $\bc(U)$ of $\Gamma$ represents a $(t_1,\ldots, t_k)$-Dyck path $(\bp,\pi)$ having the smallest area which completes $U$ to a standardized unified diagram.  Thus a Dyck path completion of $U$ is a weak composition $\bd$ that dominates $\bc(U)$, and the claim follows since there are ${m-n-i\choose \bd}$ ways to label the north steps of $\bd$ so that the labels from $\pi$ are nondecreasing on consecutive north steps.
\end{proof}

\subsection{The $k$-parking numbers}

In this section, we enumerate the truncated level-$(k,i)$ unified diagrams for $G=\car{n+1}{k}$ by bijecting them to another family of combinatorial objects that we now define.  After this is completed, we will show that for each truncated diagram $U \in \calU_G^{(k,i)}$, there are `on average' $k^{(k+1)(n-k)-3-i}$ ways to complete it to a standardized unified diagram.

\begin{defn}\label{defn.kparkingtrianglenumbers}
For $k\in \bbN$, $r\in \bbZ_{\geq0}$, and $i=0,\ldots, r$, let
$$T_k(r,i) = (r+1)^{i-1} \multiset{k(r+1)}{r-i} 
	= (r+1)^{i-1}{k(r+1)+r-1-i \choose r-i}.$$
For fixed $k$, the numbers $T_k(r,i)$ form the entries of the {\em $k$-parking triangle}.  Tables of values for $T_k(r,i)$ are given in the Appendix, for $k=1,2,3,4$.
\end{defn}

\begin{remark} We note some special values of $T_k(r,i)$.
\begin{enumerate}
\item[(a)] At $i=0$, 
$$T_k(r,0) = \frac{1}{r+1}{k(r+1) + r-1 \choose r} 
	= \Cat(r+1, k(r+1)-1)$$
is a generalized Fuss-Catalan number. This is equal to $\vol\calF_{\car{n+1}{k}}(1,0,\ldots,0,-1)$ if we let $r=n-k-1$.

\item[(b)] At $i=r$,
$T_k(r,r) = (r+1)^{r-1}$
is the number of parking functions of length $r$.

\item[(c)] At $i=r-1$,
$$T_k(r,r-1) = (r+1)^{r-2}{k(r+1)\choose k(r+1)-1} = k(r+1)^{r-1} =kT_k(r,r)$$
is $k$ times the number of parking functions of length $r$.
\end{enumerate}
\end{remark}

\begin{defn}
For $k\in \bbN$, $r\in \bbZ_{\geq0}$, and $i=0,\ldots, r$, let $\calT_k(r,i)$ be the set of classical Dyck paths from $(0,0)$ to $(r,r)$ with labeled north steps so that each of the labels from the set $\{1,2,\ldots, i\}$ appear exactly once, and the remaining $r-i$ labels are chosen (possibly with repeats) from the set $\{\overline{k-1}, \ldots, \overline{1}, \overline{0}\}$, and the labels are nondecreasing on consecutive north steps. These labels are ordered by $\overline{k-1} < \cdots < \overline{1} < \overline{0} < 1 < 2< \cdots < i$.  We call these the {\em $k$-multi-labeled Dyck paths}.
\end{defn}

\begin{theorem}\label{thm.multilabel}
For $k\in \bbN$, $r\in \bbZ_{\geq0}$, and $i=0,\ldots, r$,
$$|\calT_k(r,i)|= T_k(r,i). $$
\end{theorem}
\begin{proof}
Consider the scenario where there are $r+1$ parking spaces on a circular one way street whose single entrance/exit is just before the first parking space. There are $r$ vehicles: $d_s$ identical motorcycles of the same model $s$ for $s=\overline{k-1},\ldots, \overline{0}$, and $i$ distinct cars, so that $d_{k-1}+\cdots+d_{1}+d_{0}+i=r$. Each group of model $s$ motorcycles has a multiset of $d_s$ preferred parking spaces, and each car has a preferred parking space as well. The motorcycles arrive in groups and park, followed by each car, and if the vehicle's preferred spot is already taken, then it parks in the next available space down the circular street.  Since there are $r+1$ spaces and $r$ vehicles, every vehicle will be able to park.  

We record the parking preferences as
$$\mathbf{pp}= \{p_{k-1,1},\ldots, p_{k-1,d_{k-1}} \} \times \cdots \times  \{p_{0,1},\ldots, p_{0,d_0}\}\times (q_1,\ldots, q_i),$$
where $\{p_{s,1},\ldots, p_{s,d_s} \}$ is a multiset of parking space preferences for the model $s$ motorcycles, and $(q_1,\ldots, q_i)$ is the list of parking preferences for the $i$ cars.
The cyclic group $\bbZ/(r+1)\bbZ$ acts on the set of parking preferences by 
$$z\cdot \mathbf{pp} =  \{p_{k-1,1}+z,\ldots, p_{k-1,d_{k-1}}+z \} \times \cdots \times \{p_{0,1}+z,\ldots, p_{0,d_0}+z\} \times (q_1+z,\ldots, q_i+z) \!\!\! \mod r+1,$$
for $z\in \bbZ/(r+1)\bbZ$. If the parking preferences $\mathbf{pp}$ lead to the $j$-th vehicle parking in space $S_j$, then the parking preferences $z\cdot \mathbf{pp}$ leads to the $j$-th vehicle parking in space $(S_j+z) \mod r+1$. Thus each orbit of the cyclic group action on the set of parking preferences has size $r+1$. In each orbit, there is a unique parking configuration where the $(r+1)$-st space is empty, and this corresponds to an element in $\calT_k(r,i)$.

There are $(r+1)^i$ preference lists $(q_1,\ldots, q_i)$ for the cars, and
$$\sum_{d_0+\cdots+d_{k-1} =r-i }\multiset{r+1}{d_0}\cdots\multiset{r+1}{d_{k-1}} 
= \multiset{k(r+1)}{r-i}$$
preference sets for the $k$ models of motorcycles. Therefore, 
$$|\calT_k(r,i)|
= (r+1)^{i-1} \multiset{k(r+1)}{r-i}.$$
\end{proof}

\begin{example}
The right side of Figure~\ref{fig.3-caracol_ud_103} shows a multi-labeled Dyck path $M$.   $M$ encodes the parking preferences $\mathbf{pp} = \{1 \} \times \emptyset \times\{1,3\} \times (4,1)$ for one model-$2$ motorcycle $\mathtt{M}_{\overline{2}}$, two identical model-$0$ motorcycles $\mathtt{M}_{\overline{0}}$, and two distinct cars $\mathtt{C}_1$ and $\mathtt{C}_2$. The resulting parked configuration is $(\mathtt{M}_{\overline{2}}, \mathtt{M}_{\overline{0}},\mathtt{M}_{\overline{0}}, \mathtt{C}_1, \mathtt{C}_2)$ for the vehicles.
\end{example}

\begin{theorem} \label{thm.theta}
Let $k\in \bbN$ and $n>k$. The number of truncated level-$(k,i)$ unified diagrams for $G=\car{n+1}{k}$ is
$$\left|\,\calU_{\car{n+1}{k}}^{(k,i)}\right| 
	= T_k(n-k-1,i).$$
\end{theorem}
\begin{proof} Let $G=\car{n+1}{k}$. 
We construct a bijection $\Theta: \calU_G^{(k,i)} \rightarrow \calT_k(n-k-1,i)$.

Let $U=(\bq, \kappa, \Gamma)$ be a truncated level-$(k,i)$ unified diagram.  Recall that the embedded gravity diagram $\Gamma$ has $n-k-1-i$ dots in the $k$-th column, and every (possibly trivial) line segment in $\Gamma$ contains a dot from the $k$-th column, so we consider $\Gamma$ as having $n-k-1-i$ line segments.

From $U$, we create a $k$-multi-labeled Dyck path $M\in\calT_k(n-k-1,i)$ in the following way.
Let $\mathbf{1} = (1,\ldots, 1) \vDash n-k-1$. We may view $(\bq,\kappa)$ as a labeled $\mathbf{1}$-Dyck path with starting point $(0, n-k-1-i)$, and $\bq$ has $i$ north steps labeled by the permutation $\kappa\in \fS_i$.  To create $M$, we need to add $n-k-1-i$ more north steps to $\bq$, and the $n-k-1-i$ line segments embedded between $\bq$ and the shaded region in $U$ define these uniquely; given one such line segment $L=[\ell,k+h]$ that begins in the $\ell$-th column for some $\ell=1,\ldots, k$, and ends in a $(k+h)$-th column for some $h = 0,\ldots, n-k-2$, create a new north step at $x=h$ with the label $\overline{k-\ell}$ so that the labels remain nondecreasing on consecutive north steps of $M$, with respect to the order $\overline{k-1} < \cdots < \overline{1} < \overline{0} <1<\cdots <i$. 

We may visualize this construction of $M$ from $U$ as `sliding' the label $\overline{k-\ell}$ along the line segment $L=[\ell,k+h]$ of the gravity diagram to its end to create a new north step with that label.

To see that $M$ indeed is a $k$-multi-labeled Dyck path in $\calT_k(n-k-1,i)$, note that by virtue of the fact that the line segments of $\Gamma$ are embedded between $\bq$ and the shaded region, it is ensured that adding north steps dicted by the right endpoints of the line segments creates a Dyck path from $(0,0)$ to $(n-k-1,n-k-1)$ that remains above the line $y=x$.  The conditions on the labels of the north steps of $M$ are clearly satisfied by construction.

To see that $\Theta$ is a bijection, we describe the inverse construction. Let $M\in \calT_k(n-k-1,i)$.  It has $n-k-1-i$ north steps with labels in $\{\overline{k-1},\ldots, \overline{1},\overline{0} \}$, so by removing those, we can recover the labeled Dyck path $(\bq,\kappa)$ with $\kappa\in \fS_i$. It remains to recover the embedded gravity diagram $\Gamma$, but this is easy as well, since each north step with label $\overline{k-\ell}$ at $x=h$ gives rise to a line segment $L=[\ell,k+h]$.

Since $\Theta$ is a bijection, then the result follows from Theorem~\ref{thm.multilabel}.
\end{proof}

\begin{example} Figure~\ref{fig.3-caracol_ud_103} shows a truncated unified diagram $U=(\bq,\kappa,\Gamma) \in \calU_G^{(3,2)}$ for $G=\car{10}3$, and its corresponding $3$-multi-labeled Dyck path $M$ under the bijection $\Theta: \calU_G^{(3,2)} \rightarrow \calT_3(3,2)$. Note that the embedded gravity diagram $\Gamma$ contains three line segments; the two which begin in the third column carry the label $\textcolor{cyan}{\overline{0}}$ and the one which begins in the first column carries the label $\textcolor{cyan}{\overline{2}}$.  These labels `slide' along their line segments from left to right to form the $3$-multi-labeled Dyck path $M$.

\end{example}

\subsection{Partitioning the $N$-th multinomial $(k-1)$-simplex}
The main result of this section is to finish the computation of the number of level-$(k,i)$ standardized unified diagrams for $G=\car{n+1}{k}$.
We shall see in Theorem~\ref{thm.kparkingtriangle} that `on average' there are $k^{(k+1)(n-k)-3-i}$ ways to complete any truncated level-$(k,i)$ unified diagram to a standardized unified diagram, but first we need a Lemma.

\begin{lemma} \label{lem.thelemma} 
Let $N\in \bbN$ be a positive integer, and let $k\in \bbZ_{\geq2}$.   
Let $\calC(N,k)$ denote the set of weak compositions of $N$ with $k$ parts.
Given $\bc = (c_0,\ldots, c_{k-1}) \in \calC(N,k)$, and letting $\be_0=(1,0,0,\ldots)$, $\be_1 =(0,1,0,\ldots)$ etc., define 
\begin{align*}
\bc_0 &= (c_{0,0},\ldots, c_{0,k-1}) =\bc,\\
\bc_j &= (c_{j,0},\ldots, c_{j,k-1})= \bc+\be_{k-1}-\be_{j-1},
\end{align*} 
for $j=1,\ldots, k-1$. 
Let 
\begin{align*}
\calC(\bc_j) &= \left\{ \bd=(d_0,\ldots, d_{k-1}) \vDash N \mid d_j+\cdots +d_{j+i} \geq c_{j,j}+\cdots+c_{j,j+i}, \hbox{ for } i=0,\ldots, k-2 \right\},
\end{align*}
with the understanding that the indices of $d_{j+i}$ and $c_{j,j+i}$ are defined mod $k$, and $\calC(\bc_j)$ is empty if $\bc_j$ has negative entries.
Then 
$$\calC(N,k) = \coprod_{0\leq j\leq k-1} \calC(\bc_j),$$
is a partition of the set of weak compositions of $N$ with $k$ parts. 
\end{lemma}
\begin{proof}
Since $d_0+\cdots+d_{k-1}=N$, then we can rewrite the $k-1$ defining inequalities for each set $\calC(\bc_j)$ in terms of $d_0,\ldots, d_{k-2}$.
That is, an inequality involving $d_{k-1}$, generically of the form 
$$d_j+\cdots + d_{k-1}+d_0 + \cdots +d_{\ell-1} \geq c_{j,j} + \cdots + c_{j,k-1+\ell} =M,$$
where $j\leq k-1$,
appears as a defining inequality only in $\calC(\bc_j)$, and it can be replaced by
$$d_{\ell} + \cdots + d_{j-1} \leq c_{j,k+\ell}+\cdots + c_{j,j-1} < N-M+1,$$
where $\ell <j$.
The only other set in which the expression $d_{\ell} + \cdots + d_{j-1}$ appears in a defining inequality is $\calC(\bc_\ell)$, and there, the inequality is
$$d_{\ell} + \cdots + d_{j-1} \geq c_{\ell,\ell}+\cdots+c_{\ell, j-1}.$$
Note that by definition,
\begin{align*}
\bc_\ell &= (c_{0,0}, \ldots, c_{0,\ell-1}-1, c_{0,\ell},\ldots\ldots\ldots  \ldots, c_{0,k-1}+1),\\
\bc_j &= (c_{0,0}, \ldots \ldots\ldots  \ldots, c_{0,j-1}-1, c_{0,j}, \ldots, c_{0,k-1}+1),
\end{align*}
so if $M=c_{j,j} + \cdots + c_{j,\ell-1} =c_{0,j} +\cdots + (c_{0,k-1}+1) + \cdots + c_{0,\ell-1}$, then
$$c_{\ell,\ell}+\cdots+c_{\ell, j-1} =  c_{0,\ell}+ \cdots + c_{0,j-1} = N-M+1.$$
Therefore, the sets $\calC(\bc_j)$ are disjoint.


Since the sets $\calC(\bc_j)$ are partitioned by ${k-1\choose 2}$ hyperplanes, each of the form $d_a+\cdots+d_b = c_{0,a}+\cdots+c_{0,b}$ for $0\leq a\leq b \leq k-2$, and each of these hyperplanes contain the point $\bc_0$, then $\calC(N,k)= \coprod_{0\leq j\leq k-1} \calC(\bc_j)$, and the result follows.
\end{proof}

\begin{corollary} \label{cor.thecorollary}
With $\calC(\bc_j)$ defined as in Lemma~\ref{lem.thelemma}, let $S(\bc_j) = \sum_{\bd \in \calC(\bc_j)} {N\choose \bd}$. Then
$\sum_{j} S(\bc_j)=k^N.$
\end{corollary}
\begin{proof}
This follows from Lemma~\ref{lem.thelemma} and the multinomial theorem, 
$\sum_{\bd\in \calC(N,k)} {N\choose \bd}= k^N$.
\end{proof}

\begin{example} The essence of Lemma~\ref{lem.thelemma} is to partition multinomial coefficients in a specific way that will be useful in the proof of Theorem~\ref{thm.kparkingtriangle}.
When $k=2$, this is simply a partition of the binomial coefficients for a fixed $N$.  For example let $\bc_0=(c,N-c)$, so that $\bc_1=(c-1,N-c+1)$. We have
\begin{align*}
\calC(\bc_0) &= \{ (d,N-d) \mid d \geq c \},\\
\calC(\bc_1) &= \{ (d,N-d) \mid N-d \geq N-c+1\} = \{ (d,N-d) \mid d\leq c-1\},
\end{align*}
and $S(\bc_0) = \sum_{d=c}^{N} {N \choose d}$, $S(\bc_1) = \sum_{d=0}^{c-1} {N \choose d}$.

Simply put, we are partitioning the $N$-th row of Pascal's triangle into the set of binomial coefficients ${N \choose d}$ with $d\geq c$, and the set of binomial coefficients ${N\choose d}$ with $d< c$.  Summing over the entire row of Pascal's triangle yields $S(\bc_0)+S(\bc_1) = \sum_{d=0}^N {N\choose d}=2^N$.  
\end{example}

\begin{example}
This example explains the title of this section. Generalizing the previous example, for $k=3$, the multinomial coefficients ${N \choose d_1,\ldots, d_k}$ can be arranged on the the lattice points $(d_1,\ldots, d_k)\in \bbZ^d$, forming a $(k-1)$-simplex in $\bbZ^k$.

The left side of Figure~\ref{fig.trinomial_triangle} depicts the multinomial triangle for $N=6$ and $k=3$, with the weak composition $(d_1,d_2,d_3)$ listed below each entry. This partition of the triangle corresponds to the one defined by $\bc_0=(2,2,2)$, $\bc_1=(1,2,3)$, and $\bc_2=(2,1,3)$.
\end{example}
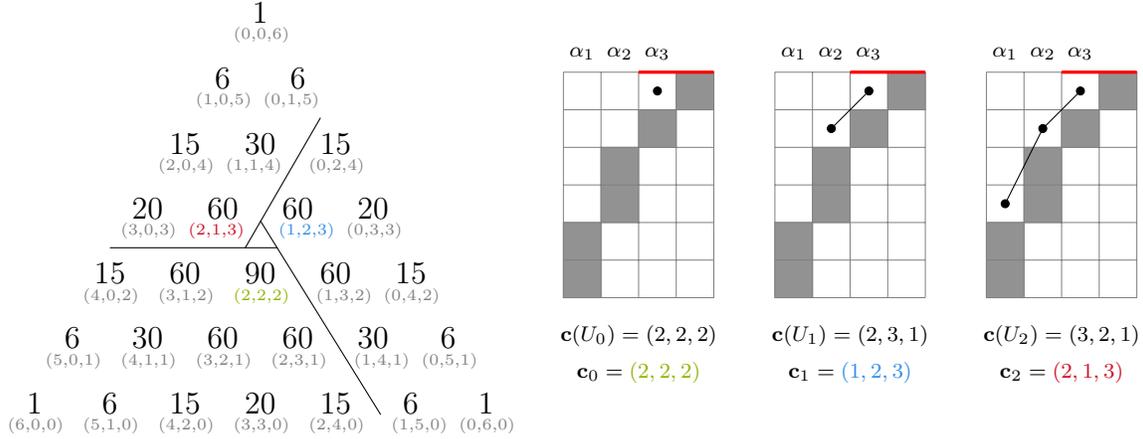
\begin{figure}[ht!]
\begin{tikzpicture}
\begin{scope}[scale=1, xshift=-200, yshift=-40]
\node at (3,5.2) {$1$}; \node at (3,4.9) {\tiny\textcolor{gray}{$\,_{(0,0,6)}$}};
\node at (2.5,4.33) {$6$}; \node at (2.5,4.03) {\tiny\textcolor{gray}{$\,_{(1,0,5)}$}};
\node at (3.5,4.33) {$6$}; \node at (3.4,4.03) {\tiny\textcolor{gray}{$\,_{(0,1,5)}$}};
\node at (2,3.46) {$15$}; \node at (2,3.16) {\tiny\textcolor{gray}{$\,_{(2,0,4)}$}};
\node at (3,3.46) {$30$}; \node at (2.9,3.16) {\tiny\textcolor{gray}{$\,_{(1,1,4)}$}};
\node at (4,3.46) {$15$}; \node at (4,3.16) {\tiny\textcolor{gray}{$\,_{(0,2,4)}$}};
\node at (1.5,2.6) {$20$}; \node at (1.5,2.3) {\tiny\textcolor{gray}{$\,_{(3,0,3)}$}};
\node at (2.5,2.6) {$60$}; \node at (2.4,2.3) {\tiny\textcolor{lava}{$\,_{(2,1,3)}$}};
\node at (3.5,2.6) {$60$}; \node at (3.6,2.3) {\tiny\textcolor{bleudefrance}{$\,_{(1,2,3)}$}};
\node at (4.5,2.6) {$20$}; \node at (4.5,2.3) {\tiny\textcolor{gray}{$\,_{(0,3,3)}$}};
\node at (1,1.73) {$15$}; \node at (1,1.43) {\tiny\textcolor{gray}{$\,_{(4,0,2)}$}};
\node at (2,1.73) {$60$}; \node at (2,1.43) {\tiny\textcolor{gray}{$\,_{(3,1,2)}$}};
\node at (3,1.73) {$90$}; \node at (3,1.43) {\tiny\textcolor{applegreen}{$\,_{(2,2,2)}$}};
\node at (4,1.73) {$60$}; \node at (4.1,1.43) {\tiny\textcolor{gray}{$\,_{(1,3,2)}$}};
\node at (5,1.73) {$15$}; \node at (5,1.43) {\tiny\textcolor{gray}{$\,_{(0,4,2)}$}};
\node at (0.5,.87) {$6$}; \node at (.5,.57) {\tiny\textcolor{gray}{$\,_{(5,0,1)}$}};
\node at (1.5,.87) {$30$}; \node at (1.5,.57) {\tiny\textcolor{gray}{$\,_{(4,1,1)}$}};
\node at (2.5,.87) {$60$}; \node at (2.5,.57) {\tiny\textcolor{gray}{$\,_{(3,2,1)}$}};
\node at (3.5,.87) {$60$}; \node at (3.5,.57) {\tiny\textcolor{gray}{$\,_{(2,3,1)}$}};
\node at (4.5,.87) {$30$}; \node at (4.6,.57) {\tiny\textcolor{gray}{$\,_{(1,4,1)}$}};
\node at (5.5,.87) {$6$}; \node at (5.5,.57) {\tiny\textcolor{gray}{$\,_{(0,5,1)}$}};
\node at (0,0) {$1$}; \node at (0,-.3) {\tiny\textcolor{gray}{$\,_{(6,0,0)}$}};
\node at (1,0) {$6$}; \node at (1,-.3) {\tiny\textcolor{gray}{$\,_{(5,1,0)}$}};
\node at (2,0) {$15$}; \node at (2,-.3) {\tiny\textcolor{gray}{$\,_{(4,2,0)}$}};
\node at (3,0) {$20$}; \node at (3,-.3) {\tiny\textcolor{gray}{$\,_{(3,3,0)}$}};
\node at (4,0) {$15$}; \node at (4,-.3) {\tiny\textcolor{gray}{$\,_{(2,4,0)}$}};
\node at (5,0) {$6$}; \node at (5.1,-.3) {\tiny\textcolor{gray}{$\,_{(1,5,0)}$}};
\node at (6,0) {$1$}; \node at (6,-.3) {\tiny\textcolor{gray}{$\,_{(0,6,0)}$}};
\draw (1,2.07)--(3.23,2.07);
\draw (3,2.43)--(4.6,-0.15);
\draw (2.8,2.07)--(3.8,3.80);
\end{scope}

\begin{scope}[xshift=0, scale=0.5]
	\draw[fill, color=gray!85] (0,0) rectangle (1,2);
	\draw[fill, color=gray!85] (1,2) rectangle (2,4);
	\draw[fill, color=gray!85] (2,4) rectangle (3,5);
	\draw[fill, color=gray!85] (3,5) rectangle (4,6);
	\draw[very thin, color=gray!100] (0,0) grid (4,6);	

	\vertex[fill, minimum size=3pt] at (2.5,5.5){};
								
	\draw[very thick, color=red] (2,6)--(4,6);
		
	\node at (0.5, 6.5) {\tiny$\alpha_1$};
	\node at (1.5, 6.5) {\tiny$\alpha_2$};
	\node at (2.5, 6.5) {\tiny$\alpha_3$};
	\node at (2,-1) {\tiny$\bc(U_0)=\textcolor{black}{(2,2,2)}$};
	\node at (2,-2) {\tiny$\bc_0=\textcolor{applegreen}{(2,2,2)}$};
\end{scope}
\begin{scope}[xshift=80, scale=0.5]
	\draw[fill, color=gray!85] (0,0) rectangle (1,2);
	\draw[fill, color=gray!85] (1,2) rectangle (2,4);
	\draw[fill, color=gray!85] (2,4) rectangle (3,5);
	\draw[fill, color=gray!85] (3,5) rectangle (4,6);
	\draw[very thin, color=gray!100] (0,0) grid (4,6);	

	\vertex[fill, minimum size=3pt] at (2.5,5.5){};
	\vertex[fill, minimum size=3pt] at (1.5,4.5){};
	\draw (1.5,4.5)--(2.5,5.5);
								
	\draw[very thick, color=red] (2,6)--(4,6);
		
	\node at (0.5, 6.5) {\tiny$\alpha_1$};
	\node at (1.5, 6.5) {\tiny$\alpha_2$};
	\node at (2.5, 6.5) {\tiny$\alpha_3$};
	\node at (2,-1) {\tiny$\bc(U_1)=\textcolor{black}{(2,3,1)}$};
	\node at (2,-2) {\tiny$\bc_1=\textcolor{bleudefrance}{(1,2,3)}$};	
\end{scope}
\begin{scope}[xshift=160, scale=0.5]
	\draw[fill, color=gray!85] (0,0) rectangle (1,2);
	\draw[fill, color=gray!85] (1,2) rectangle (2,4);
	\draw[fill, color=gray!85] (2,4) rectangle (3,5);
	\draw[fill, color=gray!85] (3,5) rectangle (4,6);
	\draw[very thin, color=gray!100] (0,0) grid (4,6);	

	\vertex[fill, minimum size=3pt] at (.5,2.5){};
	\vertex[fill, minimum size=3pt] at (1.5,4.5){};
	\vertex[fill, minimum size=3pt] at (2.5,5.5){};
	\draw (.5,2.5)--(1.5,4.5)--(2.5,5.5);
								
	\draw[very thick, color=red] (2,6)--(4,6);
		
	\node at (0.5, 6.5) {\tiny$\alpha_1$};
	\node at (1.5, 6.5) {\tiny$\alpha_2$};
	\node at (2.5, 6.5) {\tiny$\alpha_3$};
	\node at (2,-1) {\tiny$\bc(U_2)=\textcolor{black}{(3,2,1)}$};	
	\node at (2,-2) {\tiny$\bc_2=\textcolor{lava}{(2,1,3)}$};		
\end{scope}
\end{tikzpicture}
\caption{A partition of the multinomial triangle for $N=6$ and $k=3$ determined by $\bc_0,\bc_1,\bc_2$. The sum of the entries in each third is the number of completions to standardized unified diagrams for each truncated unified diagram on the right. Note that $\bc(U_j)$ is the backward cyclic shift of $\bc_j$ by $j$ positions.}
\label{fig.trinomial_triangle}
\end{figure}

\begin{theorem} \label{thm.kparkingtriangle}
Let $k\in \bbN$ and $n>k$. 
The number of standardized level-$(k,i)$ unified diagrams for $\car{n+1}{k}$ is
$$\left|\,\mathcal{SU}_{\car{n+1}{k}}^{(k,i)}\right| 
	= k^{(k+1)(n-k)-3-i} \cdot T_k(n-k-1,i).$$
\end{theorem}
\begin{proof} Let $G=\car{n+1}{k}$. 
When $k=1$, there is only one way to complete a truncated unified diagram $U=(\bq,\kappa,\Gamma)\in \calU_G^{(1,i)}$ to a standardized unified diagram because there is only one way to add $m-n-i$ north steps to complete $\bq$ in the first column. Thus it follows from Equation~\eqref{eqn.SUcompletion} and Theorem~\ref{thm.theta} that $|\,\mathcal{SU}_G^{(1,i)}|= |\,\calU_G^{(1,i)}| = T_1(n-2,i)$. 

So suppose $k\geq2$. We first define a $\bbZ/k\bbZ$-action on the set of out-degree line-dot diagrams of $\car{n+1}{k}$ which satisfy the following:
\begin{enumerate}
\item[(a)] each line segment must be horizontal,
\item[(b')] the line segments are ordered from top to bottom so that the line segments with right endpoints in the $q$-th column are above the line segments with right endpoints at the $p$-th column if $q>p$.
\end{enumerate}
We point out that the last property of the out-degree gravity diagrams that specifies a certain ordering of line segments is omitted.

Modifying Remark~\ref{rem.trapezoid} slightly to apply to these line-dot diagrams instead of gravity diagrams, we can still consider the line-dot diagrams for $\car{n+1}{k}$ to be defined on a trapezoidal array of dots with $k-1+i$ dots in the $i$-th row for $i=1,\ldots, n-k-1$.  Let $\Gamma|_R$ denote the restriction of the line-dot diagram to the first $k$ columns, and note that every line segment of $\Gamma|_R$ has its right endpoint in the $k$-th column. Letting $L_1,\ldots, L_{n-k-1}$ be the (possibly trivial) line segments of $\Gamma|_R$ where $L_j=[\ell_j,k]$, we define $\rho(\Gamma) = (\ell_1-1,\ldots, \ell_{n-k-1}-1)$.

For $z\in \bbZ/k\bbZ$, let 
$$z\cdot \rho(\Gamma) = (\tilde \ell_1,\ldots, \tilde \ell_{n-k-1})
	=(\ell_1-1-z,\ldots, \ell_{n-k-1}-1-z) \mod{k},$$
and let $z\cdot \Gamma$ be the line-dot diagram obtained from $\Gamma$ by replacing the line segments $L_1,\ldots,$ $L_{n-k-1}$ in $\Gamma|_R$ by the line segments $[\tilde \ell_1 +1, k], \ldots, [\tilde \ell_{n-k-1}+1,k]$.
The configuration of the line segments in $\Gamma$ restricted to the columns indexed by $\alpha_{k+1},\ldots, \alpha_{n-2}$ remains unchanged. 

We note that each orbit of the cyclic action of $\bbZ/k\bbZ$ on the set of line-dot diagrams of $\car{n+1}{k}$ has size $k$.  As well, there is an action of $\bbZ/k\bbZ$ on the set of truncated unified diagrams $\calU_G^{(k,i)}$ that is induced in the following way.

Fix a labeled level-$(k,i)$ $(t_{k+1},\ldots, t_n)$-Dyck path $(\bq,\kappa)$, and consider the set of truncated unified diagrams $U_j=(\bq,\kappa,\Gamma_j) \in \calU_G^{(k,i)}$ for $j=0,\ldots, k-1$, where $\{\Gamma_0,\ldots, \Gamma_{k-1}\}$ is an orbit of line-dot diagrams under the $\bbZ/k\bbZ$-action.  
Necessarily, each $\Gamma_j$ has at most $n-k-1-i$ line segments, and the cyclic $\bbZ/k\bbZ$-action is defined in the same way as before. 

In each truncated unified diagram $U_j$, the embedded line-dot diagram becomes a gravity diagram as we take the convention that the line segments should occupy the lowest possible dots in each column, so an orbit of line-dot diagrams of size $k$ can induce an orbit of truncated unified diagrams of size less than $k$.

Given a $\bbZ/k\bbZ$-orbit $\calO$ of truncated unified diagrams, we will show that
$$\sum_{U\in \calO} S(U) =  k^{m-n-i-1}|\calO|.$$
We first consider the case where the orbit $\calO = \{U_0,\ldots, U_{k-1}\}$ has size $k$.
Let $\bc(U_j)$ denote the $k$-hull of $U_j$, and let $\bc_0 = (c_1,\ldots, c_k)= \bc(U_0)$. 
Let $\bc_j = \bc_0 + \be_{k}-\be_{j}$ be as in Lemma~\ref{lem.thelemma} (with a shift in indices).
We claim that $\bc(U_j)$ is the backward cyclic shift of $\bc_j$ by $j$ positions.

Suppose $\rho(\Gamma_0) = (\ell_1-1,\ldots, \ell_{n-k-1-i}-1)$ so that $\rho(\Gamma_j) = (\ell_1-1-j, \ldots, \ell_{n-k-1-i}-1-j) \mod{k}$.
Then by Lemma~\ref{lem.cvector},
\begin{align*}
\bc(U_0) 
	&= \bh + \bb -  (n-k-1-i) \be_{k},\\
\bc(U_j) &= \bh + \beta^j(\bb) -(n-k-1-i)\be_{k},
\end{align*}
where $\bb=(b_1,\ldots, b_k)=\sum_{p=1}^{n-k-1-i} \be_{\ell_p}$, and $\beta^j$ denotes the backward cyclic shift of coordinates by $j$ positions. 
This simplifies to
\begin{align*}
\bc_0
&= (c_1,\ldots, c_k) = \bc(U_0)\\
&= (n-k+b_1, n-k+b_2, \ldots, n-k+b_{k-1}, 2(n-k-1)-i+b_k-(n-k-1-i))\\
&= (n-k+b_1, n-k+b_2, \ldots, n-k+b_{k-1}, n-k+b_k-1),
\end{align*}
and similarly,
\begin{align*}
\bc(U_j) 
&=(n-k+b_{j+1}, \ldots, n-k+b_k, n-k+b_1, \ldots, n-k+b_{j-1}, n-k+b_j-1)\\
&=(c_{j+1},\ldots, c_k+1, c_1, \ldots, c_{j-1}, c_j-1)\\
&=\beta^j(\bc_0 + \be_k - \be_j)
=\beta^j(\bc_j),
\end{align*}
as claimed.

Because $\bc(U_j)$ and $\bc_j$ are simply rearrangements of each other, then the number of ways to complete $U_j$ to a standardized unified diagram is
$$S(U_j) = \sum_{\bd\in \calC(\bc(U_j))} {m-n\choose \bd} = \sum_{\bd\in \calC(\bc_j)} {m-n\choose \bd}.$$
By Corollary~\ref{cor.thecorollary}, we conclude that when $\calO$ is an orbit of size $k$,
$$\sum_{j=0}^{k-1} S(U_j) = k^{m-n-i}.$$
More generally, in the case that the orbit $\calO$ has size less than $k$, the difference is that the $\bbZ/k\bbZ$-action generates $k$ distinct line-dot diagrams but only $|\calO|$ distinct representatives as gravity diagrams, and so  
$$\sum_{U\in \calO} S(U) = \frac{|\calO|}{k} k^{m-n-i}.$$

We finally see that
$$\left|\,\mathcal{SU}_G^{(k,i)}\right|
= \sum_{\calO} \sum_{U\in \calO} S(U)= \sum_{\calO} k^{m-n-1-i} |\calO| 
= k^{m-n-1-i}\cdot T_k(n-k-1,i),$$
where the last equality follows because the sum is over all truncated level-$(k,i)$ unified diagrams for $G$, and by Theorem~\ref{thm.theta} there are $T_k(n-k-1,i)$ of these.
\end{proof}

\begin{example} \label{eg.4car}
Figure~\ref{fig.4-car_orbit} shows a $\bbZ/4\bbZ$-orbit of line-dot diagrams for $G=\car{9}{4}$. 
The rectangular region $R$ of a line-dot diagram is the portion restricted to the columns labeled $\alpha_1,\ldots, \alpha_4$. 
Note that the $\bbZ/4\bbZ$-action on the line-dot diagrams leaves the line segments which are supported on the columns $\alpha_4, \alpha_5, \alpha_6$ unchanged.
 
\begin{figure}[ht!]
\begin{center}
\begin{tikzpicture}[scale=0.5]
\begin{scope}[scale=1]
	\draw[color=red] (0,2)--(3,2); \draw(3,2)--(4,2);
	\draw[color=red] (2,1)--(3,1); \draw(3,1)--(4,1);	
	\vertex[fill, label=above:\tiny$\alpha_1$, color=red] at (0,2) {};
	\vertex[fill, label=above:\tiny$\alpha_2$, color=red] at (1,2) {};
	\vertex[fill, label=above:\tiny$\alpha_3$, color=red] at (2,2) {};
	\vertex[fill, label=above:\tiny$\alpha_4$, color=red] at (3,2) {};
	\foreach \x in {0,1,2,3}
	\foreach \y in {0,1} 
		\vertex[fill, minimum size=4pt, color=red] at (\x,\y) {};
	\vertex[fill, minimum size=4pt] at (4,1) {};
	\vertex[fill, minimum size=4pt] at (4,2) {};
	\vertex[fill, minimum size=4pt] at (5,2) {};
	\node at (2,-1.2) {\footnotesize$\rho(\Gamma_0)=(3,2,0)$};		
\end{scope}
\begin{scope}[scale=1, xshift=220]
	\draw (3,2)--(4,2);
	\draw[color=red] (1,1)--(3,1); \draw (3,1)--(4,1);
	\draw[color=red] (2,0)--(3,0);	
	\vertex[fill, label=above:\tiny$\alpha_1$, color=red] at (0,2) {};
	\vertex[fill, label=above:\tiny$\alpha_2$, color=red] at (1,2) {};
	\vertex[fill, label=above:\tiny$\alpha_3$, color=red] at (2,2) {};
	\vertex[fill, label=above:\tiny$\alpha_4$, color=red] at (3,2) {};
	\foreach \x in {0,1,2,3}
	\foreach \y in {0,1} 
		\vertex[fill, minimum size=4pt, color=red] at (\x,\y) {};
	\vertex[fill, minimum size=4pt] at (4,1) {};
	\vertex[fill, minimum size=4pt] at (4,2) {};
	\vertex[fill, minimum size=4pt] at (5,2) {};
	\node at (2,-1.2) {\footnotesize$\rho(\Gamma_1)=(2,1,3)$};	
\end{scope}
\begin{scope}[scale=1, xshift=440]
	\draw[color=red] (2,2)--(3,2); \draw (3,2)--(4,2);
	\draw[color=red] (0,1)--(3,1); \draw (3,1)--(4,1);
	\draw[color=red] (1,0)--(3,0);		
	\vertex[fill, label=above:\tiny$\alpha_1$, color=red] at (0,2) {};
	\vertex[fill, label=above:\tiny$\alpha_2$, color=red] at (1,2) {};
	\vertex[fill, label=above:\tiny$\alpha_3$, color=red] at (2,2) {};
	\vertex[fill, label=above:\tiny$\alpha_4$, color=red] at (3,2) {};
	\foreach \x in {0,1,2,3}
	\foreach \y in {0,1} 
		\vertex[fill, minimum size=4pt, color=red] at (\x,\y) {};
	\vertex[fill, minimum size=4pt] at (4,1) {};
	\vertex[fill, minimum size=4pt] at (4,2) {};
	\vertex[fill, minimum size=4pt] at (5,2) {};	
	\node at (2,-1.2) {\footnotesize$\rho(\Gamma_2)=(1,0,2)$};	
\end{scope}
\begin{scope}[scale=1, xshift=660]
	\draw[color=red] (1,2)--(3,2); \draw (3,2)--(4,2);
	\draw (3,1)--(4,1);	
	\draw[color=red] (0,0)--(3,0);
	\vertex[fill, label=above:\tiny$\alpha_1$, color=red] at (0,2) {};
	\vertex[fill, label=above:\tiny$\alpha_2$, color=red] at (1,2) {};
	\vertex[fill, label=above:\tiny$\alpha_3$, color=red] at (2,2) {};
	\vertex[fill, label=above:\tiny$\alpha_4$, color=red] at (3,2) {};			
	\foreach \x in {0,1,2,3}
	\foreach \y in {0,1} 
		\vertex[fill, minimum size=4pt, color=red] at (\x,\y) {};
	\vertex[fill, minimum size=4pt] at (4,1) {};
	\vertex[fill, minimum size=4pt] at (4,2) {};
	\vertex[fill, minimum size=4pt] at (5,2) {};	
	\node at (2,-1.2) {\footnotesize$\rho(\Gamma_3)=(0,3,1)$};	
\end{scope}
\end{tikzpicture}
\end{center}
\caption{The $\bbZ/4\bbZ$-orbit of line-dot diagrams for $\car94$.}
\label{fig.4-car_orbit}
\end{figure}

The $4$-hull of a level-$(4,0)$ truncated unified diagram with empty gravity diagram is $\bh = (4,4,4,6)$.  Fixing the level-$(4,0)$ labeled Dyck path $(\bq,\kappa)$ and embedding $\Gamma_0$ into the $8\times 18$ grid to obtain a 
truncated unified diagram $U_0=(\bq,\kappa,\Gamma_0)$, the composition which represents its $4$-hull is
\begin{align*}
\bc(U_0) 
&= (4,4,4,6) + (1,0,0,-1) + (0,0,1,-1) + (0,0,0,1-1)\\
&= (4,4,4,6) + (1,0,1,1) + (0,0,0,-3)\\
&= (5,4,5,4).
\end{align*}
%
%
In all, the compositions representing the $4$-hulls of the truncated unified diagrams in this orbit are
$$\bc(U_0)=(5,4,5,4), \quad
\bc(U_1)=(4,5,5,4), \quad
\bc(U_2)=(5,5,5,3), \quad
\bc(U_3)=(5,5,4,4), $$
and shifting $\bc(U_j)$ forwards by $j$ positions gives
$$\bc_0=(5,4,5,4), \quad
\bc_1=(4,4,5,5), \quad
\bc_2=(5,3,5,5), \quad
\bc_3=(5,4,4,5). $$
The number of ways to complete the truncated unified diagram $U_j=(\bq,\kappa,\Gamma_j)$ is $S(U_j) = \sum_{\bd\in\calC(\bc_j)} {18\choose \bd}$, where
by Lemma~\ref{lem.thelemma}, the sets
\begin{align*}
\calC(\bc_0) &=\{ \bd \vDash 18 \mid d_0\geq5, d_0+d_1\geq9, d_0+d_1+d_2\geq14\},\\
\calC(\bc_1) &=\{ \bd \vDash 18 \mid d_1\geq4, d_1+d_2\geq9, d_0<5\},\\
\calC(\bc_2) &=\{ \bd \vDash 18 \mid d_2\geq5, d_0+d_1<9, d_1<4\},\\
\calC(\bc_3) &=\{ \bd \vDash 18 \mid d_0+d_1+d_2<4, d_1+d_2<9, d_2<5\},
\end{align*}
partition the entire set $\calC(18,4)$ of weak compositions of $m-n=18$ with $k=4$ parts. Therefore, $\sum_{j=0}^3 S(U_j) = 4^{18}. $
\end{example}

\begin{example}
We have seen in Figure~\ref{fig.car62} that there are $\Cat(3,5)=7$ out-degree gravity diagrams for $G=\car62$.  For each truncated unified diagram $U\in \calU_G^{(2,0)}$ with a specified out-degree gravity diagram, we compute the number $S(U)$ of standardized level-$(2,0)$ unified diagrams whose truncation is $U$.

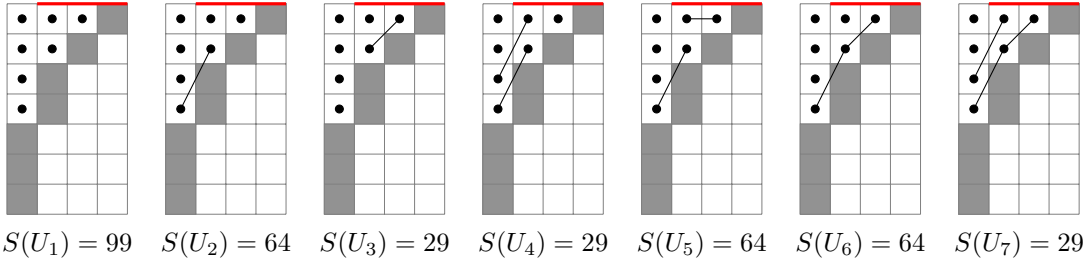
\begin{center}
\begin{tikzpicture}[scale=0.4]
\begin{scope}[scale=1]
	\draw[fill, color=gray!85] (0,0) rectangle (1,3);
	\draw[fill, color=gray!85] (1,3) rectangle (2,5);
	\draw[fill, color=gray!85] (2,5) rectangle (3,6);
	\draw[fill, color=gray!85] (3,6) rectangle (4,7);
	\draw[very thin, color=gray!100] (0,0) grid (4,7);		

	\vertex[fill, minimum size=3pt] at (2.5,6.5) {}; 
	\vertex[fill, minimum size=3pt] at (1.5,5.5) {}; 
	\vertex[fill, minimum size=3pt] at (1.5,6.5) {}; 	
	\vertex[fill, minimum size=3pt] at (0.5,3.5) {}; 
	\vertex[fill, minimum size=3pt] at (0.5,4.5) {};
	\vertex[fill, minimum size=3pt] at (0.5,5.5) {}; 
	\vertex[fill, minimum size=3pt] at (0.5,6.5) {};
	\draw[very thick, color=red] (1,7)--(4,7);
	\node at (2,-1) {\footnotesize$S(U_1)=99$};			
\end{scope}
\begin{scope}[scale=1, xshift=150]
	\draw[fill, color=gray!85] (0,0) rectangle (1,3);
	\draw[fill, color=gray!85] (1,3) rectangle (2,5);
	\draw[fill, color=gray!85] (2,5) rectangle (3,6);
	\draw[fill, color=gray!85] (3,6) rectangle (4,7);
	\draw[very thin, color=gray!100] (0,0) grid (4,7);		

	\vertex[fill, minimum size=3pt] at (2.5,6.5) {}; 
	\vertex[fill, minimum size=3pt] at (1.5,5.5) {}; 
	\vertex[fill, minimum size=3pt] at (1.5,6.5) {}; 	
	\vertex[fill, minimum size=3pt] at (0.5,3.5) {}; 
	\vertex[fill, minimum size=3pt] at (0.5,4.5) {};
	\vertex[fill, minimum size=3pt] at (0.5,5.5) {}; 
	\vertex[fill, minimum size=3pt] at (0.5,6.5) {};
	\draw (0.5,3.5)--(1.5,5.5);
	\draw[very thick, color=red] (1,7)--(4,7);	
	\node at (2,-1) {\footnotesize$S(U_2)=64$};	
\end{scope}
\begin{scope}[scale=1, xshift=300]
	\draw[fill, color=gray!85] (0,0) rectangle (1,3);
	\draw[fill, color=gray!85] (1,3) rectangle (2,5);
	\draw[fill, color=gray!85] (2,5) rectangle (3,6);
	\draw[fill, color=gray!85] (3,6) rectangle (4,7);
	\draw[very thin, color=gray!100] (0,0) grid (4,7);		

	\vertex[fill, minimum size=3pt] at (2.5,6.5) {}; 
	\vertex[fill, minimum size=3pt] at (1.5,5.5) {}; 
	\vertex[fill, minimum size=3pt] at (1.5,6.5) {}; 	
	\vertex[fill, minimum size=3pt] at (0.5,3.5) {}; 
	\vertex[fill, minimum size=3pt] at (0.5,4.5) {};
	\vertex[fill, minimum size=3pt] at (0.5,5.5) {}; 
	\vertex[fill, minimum size=3pt] at (0.5,6.5) {};
	\draw (1.5,5.5)--(2.5,6.5);
	\draw[very thick, color=red] (1,7)--(4,7);	
	\node at (2,-1) {\footnotesize$S(U_3)=29$};	
\end{scope}
\begin{scope}[scale=1, xshift=450]
	\draw[fill, color=gray!85] (0,0) rectangle (1,3);
	\draw[fill, color=gray!85] (1,3) rectangle (2,5);
	\draw[fill, color=gray!85] (2,5) rectangle (3,6);
	\draw[fill, color=gray!85] (3,6) rectangle (4,7);
	\draw[very thin, color=gray!100] (0,0) grid (4,7);		

	\vertex[fill, minimum size=3pt] at (2.5,6.5) {}; 
	\vertex[fill, minimum size=3pt] at (1.5,5.5) {}; 
	\vertex[fill, minimum size=3pt] at (1.5,6.5) {}; 	
	\vertex[fill, minimum size=3pt] at (0.5,3.5) {}; 
	\vertex[fill, minimum size=3pt] at (0.5,4.5) {};
	\vertex[fill, minimum size=3pt] at (0.5,5.5) {}; 
	\vertex[fill, minimum size=3pt] at (0.5,6.5) {};
	\draw (0.5,4.5)--(1.5,6.5);	
	\draw (0.5,3.5)--(1.5,5.5);
	\draw[very thick, color=red] (1,7)--(4,7);	
	\node at (2,-1) {\footnotesize$S(U_4)=29$};	
\end{scope}
\begin{scope}[scale=1, xshift=600]
	\draw[fill, color=gray!85] (0,0) rectangle (1,3);
	\draw[fill, color=gray!85] (1,3) rectangle (2,5);
	\draw[fill, color=gray!85] (2,5) rectangle (3,6);
	\draw[fill, color=gray!85] (3,6) rectangle (4,7);
	\draw[very thin, color=gray!100] (0,0) grid (4,7);		

	\vertex[fill, minimum size=3pt] at (2.5,6.5) {}; 
	\vertex[fill, minimum size=3pt] at (1.5,5.5) {}; 
	\vertex[fill, minimum size=3pt] at (1.5,6.5) {}; 	
	\vertex[fill, minimum size=3pt] at (0.5,3.5) {}; 
	\vertex[fill, minimum size=3pt] at (0.5,4.5) {};
	\vertex[fill, minimum size=3pt] at (0.5,5.5) {}; 
	\vertex[fill, minimum size=3pt] at (0.5,6.5) {};
	\draw (1.5,6.5)--(2.5,6.5);
	\draw (0.5,3.5)--(1.5,5.5);
	\draw[very thick, color=red] (1,7)--(4,7);	
	\node at (2,-1) {\footnotesize$S(U_5)=64$};	
\end{scope}
\begin{scope}[scale=1, xshift=750]
	\draw[fill, color=gray!85] (0,0) rectangle (1,3);
	\draw[fill, color=gray!85] (1,3) rectangle (2,5);
	\draw[fill, color=gray!85] (2,5) rectangle (3,6);
	\draw[fill, color=gray!85] (3,6) rectangle (4,7);
	\draw[very thin, color=gray!100] (0,0) grid (4,7);		

	\vertex[fill, minimum size=3pt] at (2.5,6.5) {}; 
	\vertex[fill, minimum size=3pt] at (1.5,5.5) {}; 
	\vertex[fill, minimum size=3pt] at (1.5,6.5) {}; 	
	\vertex[fill, minimum size=3pt] at (0.5,3.5) {}; 
	\vertex[fill, minimum size=3pt] at (0.5,4.5) {};
	\vertex[fill, minimum size=3pt] at (0.5,5.5) {}; 
	\vertex[fill, minimum size=3pt] at (0.5,6.5) {};
	\draw (0.5,3.5)--(1.5,5.5)--(2.5,6.5);
	\draw[very thick, color=red] (1,7)--(4,7);	
	\node at (2,-1) {\footnotesize$S(U_6)=64$};	
\end{scope}
\begin{scope}[scale=1, xshift=900]
	\draw[fill, color=gray!85] (0,0) rectangle (1,3);
	\draw[fill, color=gray!85] (1,3) rectangle (2,5);
	\draw[fill, color=gray!85] (2,5) rectangle (3,6);
	\draw[fill, color=gray!85] (3,6) rectangle (4,7);
	\draw[very thin, color=gray!100] (0,0) grid (4,7);		

	\vertex[fill, minimum size=3pt] at (2.5,6.5) {}; 
	\vertex[fill, minimum size=3pt] at (1.5,5.5) {}; 
	\vertex[fill, minimum size=3pt] at (1.5,6.5) {}; 	
	\vertex[fill, minimum size=3pt] at (0.5,3.5) {}; 
	\vertex[fill, minimum size=3pt] at (0.5,4.5) {};
	\vertex[fill, minimum size=3pt] at (0.5,5.5) {}; 
	\vertex[fill, minimum size=3pt] at (0.5,6.5) {};
	\draw (0.5,3.5)--(1.5,5.5)--(2.5,6.5);
	\draw (0.5,4.5)--(1.5,6.5);
	\draw[very thick, color=red] (1,7)--(4,7);
	\node at (2,-1) {\footnotesize$S(U_7)=29$};
\end{scope}
\end{tikzpicture}
\end{center}
Under the $\bbZ/2\bbZ$-action described in Theorem~\ref{thm.kparkingtriangle}, the orbits 
are $\{U_1, U_4\}$, $\{U_3,U_7\}$, $\{U_5,U_6\}$, and $\{U_2\}$.
For example, counting the possible labeled Dyck path completions arising from the $\{U_1, U_4\}$ orbit, we have
$$S(U_1)+S(U_4) = \sum_{i=0}^4{7\choose i} + \sum_{i=0}^2{7\choose i} = \sum_{i=0}^7{7\choose i} =2^7. $$
Summing over all orbits, 
$$\left|\, \mathcal{SU}_{\car62}^{(2,0)}\right|= \sum_{\calO} \sum_{U\in \calO} S(U) = 2^7 +2^7 + 2^7 + 2^6 = 7\cdot 2^6.$$
\end{example}

\begin{remark}
At $k=1$, equation~\eqref{eqn.stdU}, Theorem~\ref{thm.theta}, and Theorem~\ref{thm.kparkingtriangle} recover the analogous results for the classical caracol graph, proved in~\cite[Proposition 5.1, Theorem 5.6, and Theorem 5.9]{BGHHKMY}.

At $k=n-1$, Theorems~\ref{thm.theta} and~\ref{thm.kparkingtriangle} reduce to
$$\left|\,\calU_{\car{n+1}{n-1}}^{(n-1,i)}\right|
	= T_{n-1}(0,i) = 1, \qquad\hbox{and}\qquad
\left|\, \mathcal{SU}_{\car{n+1}{n-1}}^{(n-1,i)}\right|
	= (n-1)^{n-3-i},	
$$
where $i$ is necessarily $0$. These are the same results obtained for the Pitman--Stanley graph $\PS_n$ in~\cite{BGHHKMY}.
\end{remark}

\subsection{Volume of the $k$-caracol polytope}
Having developed all the tools necessary, we conclude this section with the computation that yields the volume of the flow polytope of $\car{n+1}{k}$ with net flow $\ba=(1,\ldots, 1,-n)$.

\begin{theorem}\label{thm.oneoneone}
For $k\in \bbN$ and $n>k$, let $a=n-k$ and $b=k(n-k)-1$. Then
$$\vol\calF_{\car{n+1}{k}}(1,\ldots, 1,-n)
=\Cat(a,b)\cdot k^{b-1}\cdot n^{a-1}.$$
\end{theorem}
\begin{proof} Combining Theorems~\ref{thm.parkinglidskii} and Equation~\eqref{eqn.stdU}, 
$$\vol\calF_{\car{n+1}{k}}(1,\ldots, 1,-n)
= \left|\,\calU_{\car{n+1}{k}}\right|
= \sum_{i=0}^{n-k-1} {m-n\choose i} \left|\,\mathcal{SU}_{\car{n+1}{k}}^{(k,i)}\right|.
$$
We have $m-n= (k+1)(n-k)-2 = a+b-1$. Applying Theorem~\ref{thm.kparkingtriangle}, we have
\begin{align*}
\vol\calF_{\car{n+1}{k}}(1,\ldots, 1,-n)
&= \sum_{i=0}^{a-1} {a+b-1\choose i} {a+b-1-i\choose b} a^{i-1} k^{a+b-2-i}\\
&= \frac{1}{a}\frac{(a+b-1)!}{b!(a-1)!} \cdot k^{b-1} \cdot
	\sum_{i=0}^{a-1} \frac{(a-1)!}{i!(a-1-i)!} a^{i} k^{a-1-i}\\	
&= \Cat(a,b) \cdot k^{b-1} \cdot n^{a-1},
\end{align*}
as claimed.
\end{proof}

\begin{remark}
At $k=1$, this recovers the result for the classical caracol graph~\cite[Theorem 5.10]{BGHHKMY}, 
$$\vol\calF_{\Car_{n+1}}(1,\ldots, 1,-n) = \Cat(n-2) \cdot n^{n-2}.$$

At $k=n-1$, we have $a=n-k=1$ and $b=ka-1=n-2$, so $\Cat(a,b)=\Cat(1,n-2)=1$, and we recover the result for the Pitman--Stanley graph,
$$\vol \calF_{\PS_n}(1,\ldots, 1, -(n-1)) = k^{b-1} = (n-1)^{n-3}.$$
\end{remark}

\section{The $k$-caracol polytope at other net flows} \label{sec.abbb}
The tools and combinatorial objects developed in the previous section can be augmented for some cases of more general net flow vectors. We now introduce unified diagrams for flow polytopes with net flow vector $\ba=(a_1,\ldots,a_n,-\sum_{i=1}^n a_i)$.

\begin{defn} \label{defn.unifieda}
Let $G$ be an acyclic directed graph with $n+1$ vertices and shifted out-degree vector $\bt$. A {\em unified diagram} for the flow polytope $\calF_G(\ba)$ is a type $(\bs, \sigma, \alpha, \Gamma)$ where $(\bs,\sigma)$ is a labeled $\bt$-Dyck path, $\Gamma$ is an out-degree gravity diagram for $\outgrav_G(\bs-\bt,0)$, and $\alpha$ is a vector in $[a_1]^{s_1}\times \cdots \times [a_n]^{s_n}$.  Let $\calU_G(\ba)$ denote this set of unified diagrams.
\end{defn}

We may interpret $\alpha$ as a second labeling on the north steps of the $\bt$-Dyck path, where the north steps in the $j$-th column can have a label chosen from $\{1,\ldots, a_j\}$ in any order.  We call $\alpha$ the {\em net flow label}. Observe that if any $a_j$ in the net flow vector is $0$, then the $\bt$-Dyck path in a corresponding unified diagram cannot have any north steps in its $j$-th column. Indeed, when $\ba=(1,0,\ldots, 0,-1)$, the set of unified diagrams for $\calF_G(\ba)$ is effectively just the set of out-degree gravity diagrams $\outgrav_G(\bv_{\mathrm{out}})$ because the only $\bt$-Dyck path allowed in the unified diagrams is $N^{m-n}E^{n}$ and it has a unique $\sigma$ labeling.


\begin{theorem}\label{thm.abbb}
For $k\in \bbN$ and $n>k$, let $a=n-k$ and $b=k(n-k)-1$. Let $\ba = (x^k,y^{n-k},-kx-(n-k)y)$ where $x\in \bbR_{>0}$ and $y\in \bbR_{\geq0}$. Then
$$\vol\calF_{\car{n+1}{k}}(\ba)
=\Cat(a,b)\cdot k^{b-1} x^b (kx+(n-k)y)^{a-1} .$$
\end{theorem}
\begin{proof}
Similar to Equation~\eqref{eqn.stdU}, when we partition the set of unified diagrams $\calU_G(\ba)$ according to standardized level-$(k,i)$ unified diagrams, there are ${m-n\choose i}$ ways to choose a parking function label set of size $i$ for the standardization, $x^{m-n-i}$ ways to choose net flow labels for the north steps of the $\bt$-Dyck path in the first $k$ columns, and $y^i$ ways to choose net flow labels for the remaining columns.
Thus we have
\begin{equation}\label{eqn.logconcave}
\vol\calF_{\car{n+1}{k}}(\ba)
= \Big|\,\calU_{\car{n+1}{k}}(\ba)\Big|
= \sum_{i=0}^{n-k-1} {m-n\choose i} x^{m-n-i} y^i \Big|\,\mathcal{SU}_{\car{n+1}{k}}^{(k,i)}\Big|.
\end{equation}
Applying Theorem~\ref{thm.kparkingtriangle} with $m-n=(k+1)(n-k)-2=a+b-1$, we compute
\begin{align*}
\vol\calF_{\car{n+1}{k}}(\ba)
&= \sum_{i=0}^{a-1} {a+b-1\choose i} x^{a+b-1-i} y^i
		{a+b-1-i\choose b} a^{i-1}k^{a+b-2-i}\\
&= (kx)^{b-1} x \cdot \Cat(a,b)\cdot 
	\sum_{i=0}^{a-1} {a-1\choose i} (kx)^{a-1-i} (ay)^i \\
&= (kx)^{b-1} x \cdot \Cat(a,b)\cdot (kx +ay)^{a-1},
\end{align*}
and obtain a generalization of Theorem~\ref{thm.oneoneone}.
\end{proof}

\begin{corollary}
For $k\in \bbN$ and $n>k$, let $a=n-k$ and $b=k(n-k)-1$. Then
$$\vol\calF_{\car{n+1}{k}}(\underbrace{1,\ldots,1}_k,\underbrace{0,\ldots, 0}_{n-k},-k)
=\Cat(a,b)\cdot k^{a+b-2}.$$
\end{corollary}

\begin{remark}
When $a_{k+1}=\cdots=a_n=0$, then the $\bt$-Dyck paths in the unified diagrams for $\calF_{\car{n+1}{k}}(1^k,0^{n-k},-k)$ can only have north steps in the first $k$ columns.  In other words, 
$$\vol\calF_{\car{n+1}{k}}(\underbrace{1,\ldots,1}_k,\underbrace{0,\ldots, 0}_{n-k},-k)
= \Cat(a,b)\cdot k^{a+b-2}
= \Big|\,\mathcal{SU}_{\car{n+1}{k}}^{(k,0)} \Big|$$
is the number of standardized level-$(k,0)$ unified diagrams, in agreement with Theorem~\ref{thm.kparkingtriangle}.
\end{remark}

\subsection{Log-concavity of the $k$-parking numbers}
Let $G=\car{n+1}{k}$ and $\ba=(x^k, y^{n-k}, -kx-(n-k)y)$ such that $x\in \bbR_{>0}$ and $y\in \bbR_{\geq0}$.
By a result of Baldoni and Vergne~\cite[Section 3.4]{BV}, the flow polytope $\calF_G(\ba)$ can be expressed as the Minkowski sum
$$\calF_G(\ba) = x\calF_G(\underbrace{1,\ldots,1}_k, \underbrace{0,\ldots,0}_{n-k}, -k)
+ y\calF_G(\underbrace{0,\ldots,0}_k, \underbrace{1,\ldots,1}_{n-k}, -(n-k)). $$ 
The {\em Aleksandrov-Fenchel inequalities}~\cite{A,F1,F2} state that there exists $V_i \in \bbR_{\geq0}$ such that for polytopes $P$ and $Q$,
$$\vol(xP+yQ) = \sum_{i=0}^d {d\choose i} x^{d-i} y^i V_i,  $$
and moreover, the $V_i$ are {\em log-concave} so that
$V_i^2 \geq V_{i-1}V_{i+1}$ for all $i$.
Combining our Equation~\eqref{eqn.logconcave} with Theorem~\ref{thm.kparkingtriangle}, we have
\begin{align*}
\vol\calF_G(\ba) 
&= \sum_{i=0}^{n-k-1} {m-n\choose i}x^{m-n-i}y^i |\,\mathcal{SU}_G^{(k,i)}|\\
&= \sum_{i=0}^{n-k-1} {m-n\choose i} (kx)^{m-n-i}y^i k^{-1} T_k(n-k-1,i), 
\end{align*}
so the Aleksandrov-Fenchel inequalities imply that the $k$-parking numbers $T_k(n-k-1,i)$
for fixed $n$ and $k$, and $i=0,\ldots, n-k-1$ are log-concave. See the Appendix for some values.



\section{A multigraph related to the $k$-caracol graph} \label{sec.mcar}

In the previous section, we applied techniques developed in~\cite{BGHHKMY} to compute the volumes of flow polytopes of graphs which are not planar.  In this section, we will see that there is a family of planar multigraphs which give rise to flow polytopes with volume formulas that are similar to the formulas of the previous sections.  

\subsection{Gravity diagrams for the $k$-multicaracol graph}
We next define the family of {\em $k$-multicaracol graphs}.
\begin{defn} Let $k,a\in \bbN$. The directed graph $G=\mcar{a+2}{k}$ on the vertex set $\{0,1,\ldots, a+1\}$ is constructed by starting with the Pitman--Stanley graph $\PS_{a+1}$, then adding the vertex $0$, and $k$ directed edges $(0,i)$ for $i=1,\ldots, a$.  
\end{defn}

\begin{figure}[ht]
\begin{center}
\begin{tikzpicture}[scale=0.9]
	\vertex[fill,label=below:\footnotesize{$0$}](a0) at (0,0) {};
	\vertex[fill,label=below:\footnotesize{$1$}](a1) at (1,0) {};
	\vertex[fill,label=below:\footnotesize{$2$}](a2) at (2,0) {};
	\vertex[fill,label=below:\footnotesize{$3$}](a3) at (3,0) {};
	\vertex[fill,label=below:\footnotesize{$4$}](a4) at (4,0) {};
		\node at (4.5,0) {$\cdots$};
	\vertex[fill,label=below :\footnotesize{$5$}](a10) at (5,0) {};
	\vertex[fill,label=below:\footnotesize{$6$}](a11) at (6,0) {};
	\vertex[fill,label=right:\footnotesize{$7$}](a12) at (7,0) {};
	\draw[-stealth, thick, color=red] (a0) to[out=30,in=130] (a1);		
	\draw[-stealth, thick] (1,0)--(1.95,0);
	\draw[-stealth, thick] (2,0)--(2.95,0);
	\draw[-stealth, thick] (3,0)--(3.95,0);
	\draw[thick] (4,0)--(4.15,0); \draw[thick] (4.75,0)--(4.85,0);
		\draw[-stealth, thick] (4.85,0)--(4.95,0);		
	\draw[-stealth, thick] (5,0)--(5.95,0);
	\draw[-stealth, thick] (6,0) to (6.95,0);
	\draw[-stealth, thick, color=red] (a0) to[out=35,in=130] (a2);
	\draw[-stealth, thick, color=red] (a0) to[out=40,in=130] (a3);
	\draw[-stealth, thick, color=red] (a0) to[out=45,in=130] (a4);
	\draw[-stealth, thick, color=red] (a0) to[out=50,in=130] (a10);
	\draw[-stealth, thick, color=red] (a0) to[out=55,in=130] (a11);
	\draw[-stealth, thick] (a10) to[out=-50,in=240] (a12);
	\draw[-stealth, thick] (a4) to[out=-50,in=240] (a12);
	\draw[-stealth, thick] (a3) to[out=-50,in=240] (a12);
	\draw[-stealth, thick] (a2) to[out=-50,in=240] (a12);
	\draw[-stealth, thick] (a1) to[out=-50,in=240] (a12);
\end{tikzpicture}
\end{center}
\caption{The graph $G=\mcar{8}{k}$. A red edge of the form $(0,i)$ in this picture represents $k$ distinct edges. 
}
\end{figure}
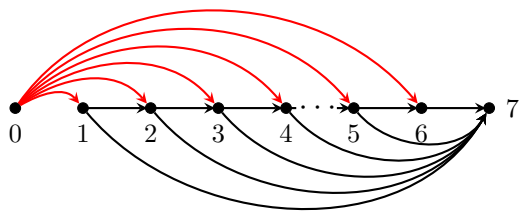

\begin{remark}
We shall see that there are many similarities between the flow polytopes $\car{n+1}{k}$ and $\mcar{a+2}{k}$, where $a=n-k$. First we note that they have the same dimension, $(k+1)(n-k)-2 = (k+1)a-2$.  If $\{f_e\}_{e\in E(\car{n+1}{k})}$ is a flow on the graph $\car{n+1}{k}$, then for each $p=1,\ldots, k-1$, the flow on the edge $(p,p+1)$ is completely determined by the flow conservation equation 
$$f_{(p,p+1)} = a_p + \sum_{(i,p)\in E(\car{n+1}{k})} f_{(i,p)}  
	- \sum_{\tiny\begin{array}{c} (p,j)\in E(\car{n+1}{k})\\ j\neq p+1 \end{array}} f_{(p,j)},$$
so if we project $\calF_{\car{n+1}{k}}(\ba)$ onto the coordinates $\{x_e\}_{e\notin \{ (i,i+1) \mid i=1,\ldots, k-1\}}$, then it may be viewed as a polytope contained in $\mcar{a+2}{k}(a_1+\cdots+a_k, a_{k+1},\ldots, a_n, -\sum a_i)$.
\end{remark}

The graph $\mcar{a+2}{k}$ has $a+2$ vertices and $m=(k+2)a-1$ edges.  Its shifted out-degree vector and shifted in-degree vector are
$$
\bt = (t_0,\ldots, t_a) = (ak-1,\underbrace{1,\ldots, 1}_{a-1},0) 
	\qquad\hbox{and}\qquad
\bu = (u_1,\ldots, u_{a+1}) = (k-1,\underbrace{k, \ldots, k}_{a-1}, a),
$$
and their coordinates sum to $m-a-1= (k+1)a-2$. We also have
$$\bv_{\mathrm{out}} = \sum_{j=0}^{a-2} (a-1-j)\alpha_j 
\qquad \hbox{and}\qquad
\bv_{\mathrm{in}} = \sum_{j=1}^a (jk-1)\alpha_j.$$

The in-degree gravity diagrams are defined on a triangular array of $jk-1$ dots in the $j$-th column for $j=1,\ldots,a$. Identical to the case of in-degree gravity diagrams for the $k$-caracol graphs, we may choose the following conventions for the in-degree gravity diagrams for $k$-multicaracol graphs:
\begin{enumerate}
\item[(a)] each line segment must be horizontal,
\item[(b)] a longer line segment must be in a row above that of a shorter line segment.
\end{enumerate}
To be precise, the set of in-degree gravity diagrams for the $k$-multicaracol graph $\mcar{a+2}{k}$ is identical to the set of in-degree gravity diagrams for the $k$-caracol graph $\car{n+1}{k}$, where $a=n-k$.
This observation immediately leads to the next result.

\begin{theorem} \label{thm.mcar_onezerozero}
For $k,a\in \bbN$,
$$\vol \calF_{\mcar{a+2}{k}} (1,0,\ldots, 0,-1) = \Cat(a, ka-1). 
$$
\end{theorem}
\begin{proof} By Corollary~\ref{cor.gravitydiagrams} and Theorem~\ref{thm.onezerozero},
$$\vol \calF_{\mcar{a+2}{k}} (1,0,\ldots, 0,-1) 
= |\ingrav_{\mcar{a+2}{k}}(\bv_{\mathrm{in}}) | 
= |\ingrav_{\car{k+a+1}{k}}(\bv_{\mathrm{in}}) | 
= \Cat(a,ka-1).$$ 
\end{proof}

\begin{remark}
In~\cite{BGHHKMY}, we introduced a polynomial for the volume of flow polytopes with properties similar to those of the Ehrhart polynomial of a polytope.
Let $G$ be a directed graph with vertex set $\{1,\ldots, n+1\}$ and $m$ edges.  For any nonnegative integer $x\in \bbZ_{\geq0}$, the directed graph $\widehat{G}(x)$ on the vertex set $\{0,1,\ldots, n+1\}$ is constructed by starting with the directed graph $G$, then adding the vertex $0$, and $x$ directed edges $(0,i)$ for $i=1,\ldots, n$. Define the polynomial
$$E_G(x) = \vol \calF_{\widehat{G}(x)}(1,0,\ldots, 0,-1). $$
In the context of this paper, $\mcar{a+2}{k} = \widehat{\PS}_{a+1}(k)$, and it follows from~\cite[Proposition 8.7]{BGHHKMY} that
\begin{equation}
\vol \calF_{\mcar{a+2}{k}} (1,0,\ldots,0,-1)
= E_{\PS_{a+1}}(k) 
= \Cat(a,ka-1).
\end{equation}
\end{remark}


By definition, the number of out-degree diagrams is equal to the number of in-degree gravity diagrams for any fixed flow polytope, so from the proof of Theorem~\ref{thm.mcar_onezerozero}, we also know that 
$$|\outgrav_{\mcar{a+2}{k}}(\bv_{\mathrm{out}}) | 
= |\outgrav_{\car{n+1}{k}}(\bv_{\mathrm{out}}) | = \Cat(a,ka-1),$$
where $a=n-k$.  We can prove this result directly via a bijection, which will be used in a later result. But first, we need to describe our conventions for the out-degree gravity diagrams for $\mcar{a+2}{k}$.  These out-degree gravity diagrams are defined on a triangular array of $a-1-j$ dots in the $j$-th column for $j=0,\ldots, a-2$, so every non-trivial line segment begins in the zero-th column indexed by $\alpha_0$, and moreoever, each line segment (including the trivial one of length zero) which begins in the zero-th column is assigned one of $k$ colours $c_1,\ldots, c_k$. In addition, we choose the following conventions for the out-degree gravity diagrams:
\begin{enumerate}
\item[(a)] each line segment must be horizontal,
\item[(b)] a longer line segment must be in a row above that of a shorter line segment,
\item[(c)] and if there are two line segments of the same length but different colours, say $c_p$ and $c_q$ with $p<q$, then the line segment of colour $c_p$ lies in a row above the line segment of colour $c_q$.
\end{enumerate}
See the diagram on the right side of Figure~\ref{fig.mcarcar} for an example of an out-degree gravity diagram for $\mcar{10}{3}$.

\begin{prop}\label{prop.mcarcar}
For $k,a\in \bbN$, 
$|\outgrav_{\mcar{a+2}{k}}(\bv_{\mathrm{out}})| = \Cat(a,ka-1)$.
\end{prop}
\begin{proof} We construct a bijection $\Xi: \outgrav_{\car{k+a+1}{k}}(\bv_{\mathrm{out}})\rightarrow \outgrav_{\mcar{a+2}{k}}(\bv_{\mathrm{out}})$ between the sets of gravity diagrams.  For the remainder of this proof, we let $\Car = \car{k+a+1}{k}$ and $\MCar = \mcar{a+2}{k}$ to simplify the notation.

Heuristically, the multigraph $\MCar$ is obtained by contracting the path of length $k-1$ on the vertices $1,\ldots, k$ in $\Car$, to the vertex $0$ in $\MCar$. 
The essential observation here is that the $k$ edges $(0,j)$ in $\MCar$ come from the $k$ edges $(i,j)$ in $\Car$ for $i=1,\ldots, k$, so a line segment in an out-degree gravity diagram for $\MCar$ which is coloured $c_i$ should be thought of as representing a positive root $\alpha_i + \cdots$ in $\Phi_{\Car}^+$.

With these observations, we define $\Xi: \outgrav_{\Car}(\bv_{\mathrm{out}})\rightarrow \outgrav_{\MCar}(\bv_{\mathrm{out}})$ as follows. Given an out-degree gravity diagram $\Gamma\in \outgrav_{\Car}$, we define $\Xi(\Gamma)$ to be the diagram obtained by `projecting' the first $k$ columns of $\Gamma$ to the zero-th column of $\Xi(\Gamma)$, where each line segment in $\Gamma$ that begins in the $i$-th column is assigned the colour $c_i$ in $\Xi(\Gamma)$.

The map $\Xi$ is well-defined because the array of dots in columns $j=k,\ldots, n-2$ in an out-degree gravity diagram for $\MCar$ is the same as the array of dots in an out-degree gravity diagram for $\Car$.  Moreover, every nontrivial line segment in $\Gamma$ contains a dot from the $k$-th column.  The conventions for the the respective out-degree gravity diagrams were chosen so that the horizontal line segments will appear in the same order.

To reverse the map $\Xi$, simply take an out-degree gravity diagram for $\MCar$ and for each line segment coloured $c_i$, extend it to a line segment which begins in the $i$-th column in the out-degree gravity diagram for $\Car$.  So $\Xi$ is a bijection.

Together with Proposition~\ref{prop.outgrav}, the result follows.
\end{proof}

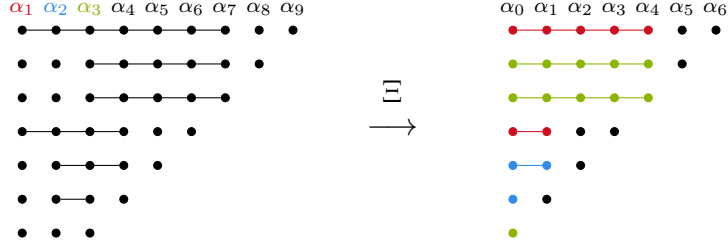
\begin{figure} [ht!]
\begin{center}
\begin{tikzpicture}[scale=0.5]

\begin{scope}[xshift=0, scale=0.9]
	\vertex[fill, minimum size=3pt, label=above:\tiny$\textcolor{lava}{\alpha_1}$](a10) at (0,0) {}; 	
	\vertex[fill, minimum size=3pt, label=above:\tiny$\textcolor{bleudefrance}{\alpha_2}$](a20) at (1,0) {}; 
	\vertex[fill, minimum size=3pt, label=above:\tiny$\textcolor{applegreen}{\alpha_3}$](a30) at (2,0) {};
	\vertex[fill, minimum size=3pt, label=above:\tiny$\alpha_4$](a40) at (3,0) {};
	\vertex[fill, minimum size=3pt, label=above:\tiny$\alpha_5$](a50) at (4,0) {};
	\vertex[fill, minimum size=3pt, label=above:\tiny$\alpha_6$](a60) at (5,0) {};
	\vertex[fill, minimum size=3pt, label=above:\tiny$\alpha_7$](a70) at (6,0) {};
	\vertex[fill, minimum size=3pt, label=above:\tiny$\alpha_8$](a80) at (7,0) {};
	\vertex[fill, minimum size=3pt, label=above:\tiny$\alpha_9$](a90) at (8,0) {};			
	\vertex[fill, minimum size=3pt](a11) at (0,-1) {}; 
	\vertex[fill, minimum size=3pt](a12) at (0,-2) {}; 
	\vertex[fill, minimum size=3pt](a13) at (0,-3) {}; 
	\vertex[fill, minimum size=3pt](a14) at (0,-4) {}; 
	\vertex[fill, minimum size=3pt](a15) at (0,-5) {}; 
	\vertex[fill, minimum size=3pt](a16) at (0,-6) {}; 
	\vertex[fill, minimum size=3pt](a21) at (1,-1) {}; 
	\vertex[fill, minimum size=3pt](a22) at (1,-2) {}; 
	\vertex[fill, minimum size=3pt](a23) at (1,-3) {};	
	\vertex[fill, minimum size=3pt](a24) at (1,-4) {}; 
	\vertex[fill, minimum size=3pt](a25) at (1,-5) {}; 
	\vertex[fill, minimum size=3pt](a26) at (1,-6) {};
	\vertex[fill, minimum size=3pt](a31) at (2,-1) {}; 
	\vertex[fill, minimum size=3pt](a32) at (2,-2) {}; 
	\vertex[fill, minimum size=3pt](a33) at (2,-3) {};	
	\vertex[fill, minimum size=3pt](a34) at (2,-4) {}; 
	\vertex[fill, minimum size=3pt](a35) at (2,-5) {}; 
	\vertex[fill, minimum size=3pt](a36) at (2,-6) {};
	\vertex[fill, minimum size=3pt](a41) at (3,-1) {}; 
	\vertex[fill, minimum size=3pt](a42) at (3,-2) {}; 
	\vertex[fill, minimum size=3pt](a43) at (3,-3) {};
	\vertex[fill, minimum size=3pt](a44) at (3,-4) {}; 
	\vertex[fill, minimum size=3pt](a45) at (3,-5) {};
	\vertex[fill, minimum size=3pt](a51) at (4,-1) {}; 
	\vertex[fill, minimum size=3pt](a52) at (4,-2) {}; 
	\vertex[fill, minimum size=3pt](a53) at (4,-3) {};
	\vertex[fill, minimum size=3pt](a54) at (4,-4) {};
	\vertex[fill, minimum size=3pt](a61) at (5,-1) {}; 
	\vertex[fill, minimum size=3pt](a62) at (5,-2) {}; 
	\vertex[fill, minimum size=3pt](a63) at (5,-3) {};
	\vertex[fill, minimum size=3pt](a71) at (6,-1) {}; 
	\vertex[fill, minimum size=3pt](a72) at (6,-2) {};
	\vertex[fill, minimum size=3pt](a81) at (7,-1) {};	
	\draw (a10) to (a70);
	\draw (a31) to (a71);
	\draw (a32) to (a72);
	\draw (a13) to (a43);
	\draw (a24) to (a44);
	\draw (a25) to (a35);					
\end{scope}

\begin{scope}[xshift=280, yshift=-75]
	\node[label=above:$\Xi$] at (0,0) {$\longrightarrow$};
\end{scope}

\begin{scope}[xshift=320, scale=0.9]
	\vertex[fill, minimum size=3pt, label=above:\tiny$\alpha_0$, color=lava](a30) at (2,0) {};
	\vertex[fill, minimum size=3pt, label=above:\tiny$\alpha_1$, color=lava](a40) at (3,0) {};
	\vertex[fill, minimum size=3pt, label=above:\tiny$\alpha_2$, color=lava](a50) at (4,0) {};
	\vertex[fill, minimum size=3pt, label=above:\tiny$\alpha_3$, color=lava](a60) at (5,0) {};
	\vertex[fill, minimum size=3pt, label=above:\tiny$\alpha_4$, color=lava](a70) at (6,0) {};
	\vertex[fill, minimum size=3pt, label=above:\tiny$\alpha_5$](a80) at (7,0) {};
	\vertex[fill, minimum size=3pt, label=above:\tiny$\alpha_6$](a90) at (8,0) {};			
	\vertex[fill, minimum size=3pt, color=applegreen](a31) at (2,-1) {}; 
	\vertex[fill, minimum size=3pt, color=applegreen](a32) at (2,-2) {}; 
	\vertex[fill, minimum size=3pt, color=lava](a33) at (2,-3) {};	
	\vertex[fill, minimum size=3pt, color=bleudefrance](a34) at (2,-4) {}; 
	\vertex[fill, minimum size=3pt, color=bleudefrance](a35) at (2,-5) {}; 
	\vertex[fill, minimum size=3pt, color=applegreen](a36) at (2,-6) {};
	\vertex[fill, minimum size=3pt, color=applegreen](a41) at (3,-1) {}; 
	\vertex[fill, minimum size=3pt, color=applegreen](a42) at (3,-2) {}; 
	\vertex[fill, minimum size=3pt, color=lava](a43) at (3,-3) {};
	\vertex[fill, minimum size=3pt, color=bleudefrance](a44) at (3,-4) {}; 
	\vertex[fill, minimum size=3pt](a45) at (3,-5) {};
	\vertex[fill, minimum size=3pt, color=applegreen](a51) at (4,-1) {}; 
	\vertex[fill, minimum size=3pt, color=applegreen](a52) at (4,-2) {}; 
	\vertex[fill, minimum size=3pt](a53) at (4,-3) {};
	\vertex[fill, minimum size=3pt](a54) at (4,-4) {};
	\vertex[fill, minimum size=3pt, color=applegreen](a61) at (5,-1) {}; 
	\vertex[fill, minimum size=3pt, color=applegreen](a62) at (5,-2) {}; 
	\vertex[fill, minimum size=3pt](a63) at (5,-3) {};
	\vertex[fill, minimum size=3pt, color=applegreen](a71) at (6,-1) {}; 
	\vertex[fill, minimum size=3pt, color=applegreen](a72) at (6,-2) {};
	\vertex[fill, minimum size=3pt](a81) at (7,-1) {};	
	\draw[color=lava] (a30) to (a70);
	\draw[color=applegreen] (a31) to (a71);
	\draw[color=applegreen] (a32) to (a72);
	\draw[color=lava] (a33) to (a43);
	\draw[color=bleudefrance] (a34) to (a44);					
\end{scope}

\end{tikzpicture}
\end{center}
\caption{The bijection between out-degree gravity diagrams for $\car{12}{3}$ and $\mcar{10}{3}$. The colours of the gravity line segments for $\mcar{10}{3}$ are red, blue, and green, in that order.}
\label{fig.mcarcar}
\end{figure}


We have seen that at net flow $(1,0,\ldots,0,-1)$, the flow polytopes of the graphs $\mcar{a+2}{k}$ and $\car{n+1}{k}$ have the same volume.
Next, we will see that the volumes of the flow polytopes of this pair of graphs are closely related at other net flows as well.



\subsection{Unified diagrams for the $k$-multicaracol graph}

\begin{theorem}\label{thm.mcarabbb}
For $k,a\in \bbN$,
$$\vol\calF_{\mcar{a+2}{k}}(kx,y,\ldots,y,-(kx+ay)) = \Cat(a,ka-1) \cdot (kx)^{ka-1} (kx+ay)^{a-1}.$$
\end{theorem}
\begin{proof} Let $\ba=(kx, y,\ldots, y, -(kx+ay))$.  
We have
$$\vol\calF_{\mcar{a+2}{k}}(\ba)
= \Big|\,\calU_{\mcar{a+2}{k}} (\ba)\Big|
= \sum_{i=0}^{a-1} {(k+1)a-2\choose i} (kx)^{(k+1)a-2-i} y^i 
	\Big|\,\calU_{\mcar{a+2}{k}}^{(0,i)}\Big|.
$$
Letting $n=a+k$, we can extend the bijection $\Xi$ in Proposition~\ref{prop.mcarcar} to the set of truncated unified diagrams 
$$\widehat{\Xi}: \calU_{\mcar{a+2}{k}}^{(0,i)} \rightarrow \calU_{\car{n+1}{k}}^{(k,i)}: \Xi(\bq,\kappa,\Gamma) \mapsto (\bq,\kappa, \Xi(\Gamma)),$$
so $|\,\calU_{\mcar{a+2}{k}}^{(0,i)}|= |\,\calU_{\car{n+1}{k}}^{(k,i)}|$.
Applying Theorem~\ref{thm.theta} and computing in the same way as in Theorem~\ref{thm.abbb},
\begin{align*}
\vol\calF_{\mcar{a+2}{k}}(\ba)
&= \sum_{i=0}^{a-1} {(k+1)a-2\choose i} (kx)^{(k+1)a-2-i} y^i {(k+1)a-2-i\choose ka-1}a^{i-1}\\
&= \Cat(a,b)\cdot (kx)^{ka-1} (kx+ay)^{a-1}.
\end{align*}
\end{proof}

\begin{corollary} We have the following specializations:
\begin{enumerate}
\item[(a)] $\displaystyle \vol\calF_{\mcar{a+2}{k}}(k,1,\ldots,1, -k-a) = \Cat(a,ka-1) \cdot k^{ka-1} (k+a)^{a-1}. $
\item[(b)] $\displaystyle \vol\calF_{\mcar{a+2}{k}}(k,0,\ldots,0, -k) = \Cat(a,ka-1) \cdot k^{(k+1)a-2}. $
\end{enumerate}
\end{corollary}

\begin{remark}
Let $a=n-k$ so that $b=ka-1$, and comparing the results of Theorem~\ref{thm.mcarabbb} and Theorem~\ref{thm.abbb}, we see that
$$\vol\calF_{\mcar{n-k+2}{k}}(kx, y^{n-k}, -kx-(n-k)y)
=k\cdot \vol\calF_{\car{n+1}{k}}(x^k, y^{n-k}, -kx-(n-k)y). $$
This implies that one can obtain a different proof of Theorem~\ref{thm.abbb} (and its specialization Theorem~\ref{thm.oneoneone}) if one can construct a $k$ to $1$ map from the set of unified diagrams for $\calF_{\mcar{n-k+2}{k}}(kx, y^{n-k}, -kx-(n-k)y)$ to the set of unified diagrams for $\calF_{\car{n+1}{k}}(x^k, y^{n-k}, -kx-(n-k)y)$.  It may be interesting to understand this map from a geometric viewpoint.
\end{remark}

\appendix
\section{Some $k$-parking triangles} \label{sec.app}
\renewcommand{\arraycolsep}{1.5mm}
\begin{table}[ht!]
\parbox{.49\linewidth}{\centering
$\begin{array}{cllllll}
\hline
r\backslash i &0&1&2&3&4&5\\ 
\hline
0&1\\
1&1&1\\
2&2&3&3\\
3&5&10&16&16\\
4&14&35&75&125&125\\
5&42&126&336&756&1296&1296\\
\hline
\end{array}
$
\caption*{The $1$-parking triangle $T_1(r,i)$.}
}\hfill
\parbox{.49\linewidth}{\centering
$\begin{array}{cllllll}
\hline
r\backslash i &0&1&2&3&4&5\\ 
\hline
0&	1\\
1&	2&1\\
2&	7&6&3\\
3&	30&36&32&16\\
4&	143&220&275&250&125\\
5&	728&1365&2184&2808&2592&1296\\
\hline
\end{array}
$
\caption*{The $2$-parking triangle $T_2(r,i)$.}
}
\end{table}

\renewcommand{\arraycolsep}{1mm}
\begin{table}[ht!]
\parbox{.49\linewidth}{\centering
$
\begin{array}{cllllll}
\hline
r\backslash i &0&1&2&3&4&5\\ 
\hline
0&	1\\
1&	3&1\\
2&	15&9&3\\
3&	91&78&48&16\\
4&	612&680&600&375&125\\
5&	4389&5985&6840&6156&3888&1296\\
\hline
\end{array}
$
\caption*{The $3$-parking triangle $T_3(r,i)$.}
}\hfill
\parbox{.49\linewidth}{\centering
$
\begin{array}{cllllll}
\hline
r\backslash i &0&1&2&3&4&5\\ 
\hline
0&	1\\
1&	4&1\\
2&	26&12&3\\
3&	204&136&64&16\\
4&	1771&1540&1050&500&125\\
5&	16380&17550&15600&10800&5184&1296\\
\hline
\end{array}
$
\caption*{The $4$-parking triangle $T_4(r,i)$.}
}
\end{table}

%




\end{document}